\documentclass{amsart}
\usepackage{amsthm, amsfonts, amssymb, amsmath, tikz, graphicx, zref,enumitem}
\usetikzlibrary{arrows}
\usepackage{cite}
\usepackage{hyperref}
\makeatletter
\let\@@pmod\pmod
\DeclareRobustCommand{\pmod}{\@ifstar\@pmods\@@pmod}
\def\@pmods#1{\mkern4mu({\operator@font mod}\mkern 6mu#1)}
\makeatother

\newcommand{\sslash}{\mathbin{/\mkern-6mu/}}

\mathchardef\mhyphen="2D
\renewcommand{\Re}{\operatorname{Re}}
\renewcommand{\Im}{\operatorname{Im}}

\newtheorem{theorem}{Theorem}
\newtheorem*{theorem*}{Theorem}
\newtheorem{definition}{Definition}
\newtheorem{lemma}{Lemma}
\newtheorem{example}{Example}
\newtheorem{proposition}{Proposition}
\newtheorem{corollary}{Corollary}
\newtheorem{claim}{Claim}

\newcommand{\tn}{\text{ d}}

\newcommand{\units}{\mathbb{G}}
\newcommand{\Vol}{\textnormal{Vol}}

\newcommand{\vc}{\mathcal{V}}
\newcommand{\vt}{\mathcal{T}}
\newcommand{\vs}{\mathcal{S}}
\newcommand{\symalg}[1]{\textnormal{Sym}^* (#1)}
\newcommand{\symalgh}[2]{\textnormal{Sym}^{#2} (#1)}

\newcommand{\wt}{\textnormal{wt}}

\newcommand{\cp}[1]{\textnormal{Critp}_{#1}}
\newcommand{\cv}[1]{\textnormal{Critv}_{#1}}

\newcommand{\Hom}{\textnormal{Hom}}
\newcommand{\Ext}{\textnormal{Ext}}

\begin{document}
	
	\title[Homological mirror symmetry of elementary birational cobordisms]{Homological mirror symmetry of elementary birational cobordisms}
	
	
	\author[G. Kerr]{Gabriel Kerr}
	\address{Department of Mathematics, Kansas State University, Manhattan, KS, 66502, USA}
	\email{gdkerr@math.ksu.edu}
	

	
	
	
	\begin{abstract}
	 The derived category of coherent sheaves $\mathcal{T}_B$ associated to a birational cobordism which is either a weighted projective space, a stacky Atiyah flip, or a stacky blow-up of a point has a conjectural mirror Fukaya-Seidel category $\mathcal{T}_A$. The potential $W$ defining $\mathcal{T}_{A}$ has base $\mathbb{C}^*$ and exhibits a great deal of symmetry. This paper investigates the structure of the Fukaya-Seidel category for the mirror potentials. A proof of homological mirror symmetry $\mathcal{T}_A \cong \mathcal{T}_B$ for these birational cobordisms is then given.
	\end{abstract}
	
	\maketitle

\section{Introduction}
Among the myriad of predictions arising from homological mirror symmetry is the conjecture of an equivalence between an $A$-model Fukaya-Seidel category $\mathcal{F} (W)$ associated to a potential $W$ and the $B$-model derived category $D^b (X)$ of coherent sheaves associated to a mirror algebraic variety $X$. This version of the conjecture usually requires that $X$ be a Fano stack and has been confirmed for a variety of special cases including weighted projective planes \cite{ako} and del Pezzo surfaces \cite{ako2, ueda}. The case of a non-Fano toric surface has also been explored in \cite{bdfkk}. While these approaches have provided evidence in favor of this classical conjecture, a proof for the case of general stacky surfaces and higher dimensional stacks has been elusive. The standard approach has instead been to replace the category $\mathcal{F} (W)$ with a conjecturally equivalent $A$-model category $\mathcal{C}_A$ and prove equivalences with this more manageable category (see \cite{abouzaid09, fltz, fu}).

In \cite{dkk}, a strategy for proving the original conjecture for arbitrary NEF toric stacks was introduced, one which could be extrapolated to the non-toric setting. This program was inspired by the result in \cite{bo}, which asserts that to certain birational maps $f : X \dashrightarrow Y$, there is a semi-orthogonal decomposition 
\begin{align} \label{eq:decomp} D (X) \cong \left<  \mathcal{T}^B_f , D (Y) \right>. \end{align}
In the toric setting, there is always a sequence of birational maps $\mathbf{f} = (f_1, \ldots, f_r)$,
\begin{align*}
	X \stackrel{f_1}{\dashrightarrow} X_1 \stackrel{f_2}{\dashrightarrow} \cdots \stackrel{f_r}{\dashrightarrow} X_r, 
\end{align*}
constituting a minimal model run on $X$ where $X_r$ is of smaller dimension than $X$. Inductively applying equation~\ref{eq:decomp}  then fully decomposes the category $D(X)$ in terms of the more elementary pieces  $\mathcal{T}^B_{f_i}$ and $D (X_r)$. As we will see, when we recast birational maps as variations of GIT, the space $X_r$ may often be considered to be the empty set. In these situations, one is left only with the component categories  $\mathcal{T}^B_{f_i}$ in the decomposition.

An A-model description of the mirror to such a minimal model sequence $\mathbf{f}$ was given in \cite{dkk}. It was shown that to $\mathbf{f}$, one may associate a degeneration $\psi_t$ of the mirror potential $W$ to $X$. Taking $D^*$ to be a punctured disc, the degeneration $\psi_t $ is a $t \in D^*$ dependent map from $\mathbb{P}^1$ to a compactified moduli space $\mathcal{M}$ of toric hypersurfaces. The potential $W_t$ is the pullback of a universal hypersurface $\mathcal{U}$ over $\mathcal{M}$. For $t = 1$, we have $W = W_1$, and as $t$ tends to zero, the map $\psi_t$ degenerates or bubbles, into a chain of $r$ projective lines 
\begin{align*}
	\psi_0 : \cup_{i = 1}^r \mathbb{P}^1 \to \mathcal{M}.
\end{align*}
This is analogous to the convergence of a holomorphic curve to a stable curve, however, the moduli space $\mathcal{M}$ is highly singular, so we refrain from using this language. Over each component $1 \leq i \leq r$ of the degenerate curve $\psi_0$, we have a map $\psi_0^i : \mathbb{P}^1 \to \mathcal{M}$ and we may pullback the universal hypersurface to obtain a component potential $W_{\mathbf{f}}^i : Y_i \to \mathbb{P}^1$. In general, the fibers of this potential will be highly degenerate themselves and not susceptible to the standard definition of a Fukaya-Seidel category. However,  in the generic case, there will be an irreducible component $Z_i$ of $Y_i$ on which $W_{\mathbf{f}}^i$ restricts to a Lefschetz fibration, so long as we consider it as a function over $\mathbb{C}^* \subset \mathbb{P}^1$ (i.e. we excise the intersection points with the $(i -1)$-st and $(i + 1)$-th components of $\psi_0$). While the non-generic case occurs frequently, it can be regarded as a suspension of the generic case.

Assuming all $W_{\mathbf{f}}^i$ are generic, \cite[Conjecture 3.22]{dkk} asserts that homological mirror symmetry may be enhanced to preserve these decompositions. In other words, Fukaya-Seidel categories $\mathcal{T}^A_{f_i} := \mathcal{F} (W_{\mathbf{f}}^i )$ are mirror to their $B$-model counterparts $\mathcal{T}_{f_i}^B$ arising from the minimal model sequence $\mathbf{f}$ and, moreover, the respective categories are naturally assembled in equivalent decompositions of $\mathcal{F} (W)$ and $D (X)$, respectively. The main results of this paper, Theorems~\ref{thm:Bmodel} and \ref{thm:Amodel}, establish the equivalence of the categories $\mathcal{T}^B_{f_i}$ and $\mathcal{T}^B_{f_i}$ when $f_i$ is an elementary birational cobordism.

Elementary birational cobordisms are found throughout algebraic geometry as stacky blowups, stacky Atiyah flips and weighted projective spaces. Moreover, factorization theorems such as \cite[Theorem~2]{wlodarczyk} show that birational maps may often be factored using these local models. Thus the conjecture's validity  opens the door to a more general approach to homological mirror symmetry, reducing the array of cases to that of minimal models. This paper proves a central part of this conjecture, namely, that there is an equivalence of semi-orthogonal components $\mathcal{T}^B_{f_i}$ and $\mathcal{T}^A_{f_i}$ when $f_i$ arises as a toric birational cobordism with zero dimensional center. 

The plan of the paper is as follows. In Section~\ref{sec:B} the categorical background will be reviewed and a complete description of elementary birational cobordisms will be given. Using a toric model, these birational maps $f_\mathbf{a} : X_+^\mathbf{a} \dashrightarrow X_-^{\mathbf{a}}$ are fully classified by specifying a lattice element $\mathbf{a} \in \mathbb{Z}^{d + 2}$.  Applying results on semi-orthogonal decompositions arising from variations of GIT from \cite{bfk}, we then give an explicit and elementary representation of the $B$-model category $\mathcal{T}^B_{\mathbf{a}}$ associated to the birational morphism in Theorem~\ref{thm:Bmodel}. 

In Section~\ref{sec:A} we describe Fukaya-Seidel categories over potentials $W : Y \to \mathbb{C}^*$ obtained by pullbacks along $z^n$. Given $\mathbf{a} \in \mathbb{Z}^{d + 2}$ satisfying certain balancing conditions, we then define the potential $W_{\mathbf{a}}$, called a circuit potential, from a $d$-dimensional pair of pants to $\mathbb{C}^*$.  Pulling back $W_{\mathbf{a}}$ along $z \mapsto z^n$ yields another potential which was conjectured to be a mirror to $f_{\mathbf{a}}$ in \cite{dkk}. The Fukaya-Seidel category associated to this pullback, denoted $\mathcal{F} (W_{\mathbf{a}}^{1/n})$,  is defined and Theorem~\ref{thm:Amodel}, which describes $\mathcal{F} (W_{\mathbf{a}}^{1/n})$ as the mirror category, is stated. As is often the case, obtaining an explicit representation of the Fukaya category takes a fair amount of labor, so the remaining portion of the paper is devoted to this and to proving Theorem~\ref{thm:Amodel}. The proof is by induction, so a detailed analysis of the one-dimensional case is given in Section~\ref{sec:1dcase}. Proceeding to the case of arbitrary dimension, Section~\ref{sec:dcase} explores the general structure of a $d$-dimensional circuit potential. Proposition~\ref{prop:fibprod} gives a surprising and useful characterization of the fiber of $W_{\mathbf{a}}$ as the pullback of a $(d-1)$-dimensional circuit along a one-dimensional circuit. This key result, combined with an application of the well developed techniques of matching paths and matching cycles in \cite{seidel}, paves the way for the induction step. Finally, in Section~\ref{sec:proof} the proof of Theorem~\ref{thm:Amodel} is completed.

\subsection*{Acknowledgements} The author thanks M. Ballard, C. Diemer, D. Favero, L. Katzarkov, P. Seidel and Y. Soibelman for helpful discussions. 

\section{$B$-model} \label{sec:B}
We will first describe basic constructions and introduce our main algebra $R_{\mathbf{a}}$ in Section~\ref{sec:catprelim}. We will then detail the birational cobordisms, viewed as variations of GIT, and establish the relationship between the category $\mathcal{T}^B_{\mathbf{a}}$ and modules over $R_\mathbf{a}$. 
\subsection{\label{sec:catprelim}Categorical preliminaries}
Let $\mathcal{C}$ be an $A_\infty$ or DG-category. We will assume a general background of such categories as presented in \cite{lefevrehasegawa, seidel}. 

\begin{definition} \label{def:directedcategory}
	Suppose $\mathcal{C}$ is an $A_\infty$-category and $\mathcal{B} = \{A_1 , \ldots, A_m\}$ is an ordered collection  of objects in $\mathcal{C}$. The directed subcategory $\mathcal{C}_\mathcal{B}$ has objects $\mathcal{B}$ and morphisms 
	\begin{align*}
	\Hom_{\mathcal{C}_\mathcal{B}} (A_i, A_j) & = \begin{cases} 1_{A_i} & \textnormal{ if } i = j, \\ \Hom_{\mathcal{C}} (A_i, A_j ) & \textnormal{ if } i < j, \\
	0 & \textnormal{ otherwise }.
	\end{cases}
	\end{align*}
	The differentials, compositions and higher compositions are induced from the respective maps in $\mathcal{C}$.
\end{definition}

Let $\mathbf{a} = (a_0, \ldots, a_{d + 1}) \in \mathbb{Z}^{d + 2}$ be an
element for which 
\begin{align} \label{eq:balance} \begin{split} \sum_{i = 0}^{d + 1} a_i & = 0, \\
a_i &\ne 0 \textnormal{ for all }0 \leq i \leq d + 1. \end{split}
\end{align} 
We will call an element satisfying the first condition \textbf{balanced}. If there are $(p + 1)$ positive and $(q + 1)$ negative coordinates, we say that $\mathbf{a}$ has \textbf{signature} $(p,q)$.

Let $V$ be a vector space with basis $\{v_0 , \ldots, v_{d + 1}\}$ and introduce an additional function 
\begin{align} \label{eq:defnu}
\nu : \{0, \ldots, {d + 1} \} \to \mathbb{Z}.
\end{align}
When $\sum \nu (i) = 0$, we will call $\nu$ balanced.
This function, together with $\mathbf{a}$, makes $V$ into a bigraded vector space with
\begin{align} \label{eq:degrees} |v_i| = \left( \deg (v_i) , \wt (v_i) \right) &= \begin{cases} ( 2 \nu (i ) , a_i ) & \textnormal{ if } a_i > 0, \\ ( 2 \nu (i) + 1, - a_i )  & \textnormal{ if } a_i < 0. \end{cases} 
\end{align}
Take $R_{\mathbf{a}, \nu}$ to be the graded-symmetric algebra $\symalg{V}$ with respect to the degree grading, equipped with the additional grading associated to weight. If the function $\nu$ is clear from the context, we simply write $R_{\mathbf{a}}$ for $R_{\mathbf{a}, \nu}$. The category of $R_\mathbf{a}$ modules which are graded with respect to degree only, and whose homomorphisms respect the degree, will be denoted $R_{\mathbf{a}}\mhyphen \textnormal{mod}$. The category of bigraded modules will be denoted $R_{\mathbf{a}}\mhyphen \textnormal{mod}^\mathbb{Z}$. For $k \in \mathbb{Z}$, let $R_{\mathbf{a}} (k) \in R_{\mathbf{a}}\mhyphen \textnormal{mod}^\mathbb{Z}$ be the  $R_\mathbf{a}$-module with weights shifted by $-k$, i.e. with $|v_i| = ( \deg (v_i), |a_i| - k)$.

\begin{definition} \label{def:can}
	Given $\mathbf{a}$, $\nu$ and $n \in \mathbb{N}$, let $\mathcal{B}_{\mathbf{a},  n} = \{R_{\mathbf{a}}, R_{\mathbf{a}} (1), \ldots, R_{\mathbf{a}} (n - 1)  \}$ and define the directed subcategory 
	
	\begin{align*} 	\mathcal{C}_{\mathbf{a}, \nu, n} := \left( R_{\mathbf{a}}\mhyphen \textnormal{mod}^\mathbb{Z} \right)_\mathcal{B}. \end{align*}
	Write $\mathcal{D}_{\mathbf{a}, \nu, n}$ for its derived category.
\end{definition}
Note that $\oplus_{k = 0}^{n - 1} R_{\mathbf{a}} (k)$ generates $\mathcal{C}_{\mathbf{a}, \nu,  n}$ and $\mathcal{B}_{\mathbf{a}, n}$ is an exceptional collection. Thus, one way of describing the derived category is via the Yoneda functor as  
\begin{align*} \mathcal{D}_{\mathbf{a}, \nu, n} \cong  D \left( \textnormal{End}_{R_{\mathbf{a}}} \left( \oplus_{k = 0}^{n - 1} R_{\mathbf{a}} (k)  \right)\mhyphen \textnormal{mod} \right).
\end{align*}  
We now make an observation about Koszul duality for this endomorphism algebra. In one of the first known instances of Koszul duality \cite{beilinson, BGS}, it was shown that the Koszul dual of a super-symmetric algebra $\symalg{V}$ was the super-symmetric algebra $\symalg{V^*[1]}$. Translating this result into our notation gives \begin{align*} R_{\mathbf{a}, \nu}^! = R_{-\mathbf{a}, - \nu}. \end{align*} 
From \cite{bondal}, it follows that this duality carries over to the exceptional collections and directed subcategories. To state this result precisely, recall that, given an exceptional collection with $n$-objects in an $A_\infty$-category $\mathcal{A}$, one may mutate the collection by performing a sequence of twists $\sigma_i$ for $1 \leq i \leq n - 1$. As these twists satisfy the braid relations, this gives a braid group $B_{n - 1}$ action on the set of exceptional collections in $\mathcal{A}$. The square root $\Delta$ of the positive generator of the center (a half-twist of the set of strands) then takes an exceptional collection $\mathcal{B}$ to $\Delta \mathcal{B}$ which, following \cite{seidel}, we call the Koszul dual collection.
\begin{proposition} \label{prop:duality}
	The Koszul dual of $\textnormal{End}_{R_{\mathbf{a}}} \left( \oplus_{k = 0}^{n - 1} R_{\mathbf{a}} (k)  \right)$ is $\textnormal{End}_{R_{\mathbf{-a}}} \left( \oplus_{k = 0}^{n - 1} R_{\mathbf{-a}} (k)  \right)$. Furthermore,  $\mathcal{B}_{-\mathbf{a}, n} = \Delta \mathcal{B}_{\mathbf{a}, n}$.
\end{proposition}
\begin{proof}
	This follows from a basic application of more general results in \cite{bondal}.
\end{proof}
\begin{corollary}
	There categories $\mathcal{D}_{\mathbf{a}, \nu, n}$ and $\mathcal{D}_{-\mathbf{a}, - \nu, n}$ are equivalent.
\end{corollary}

\subsection{Birational cobordisms induced by $\mathbf{a} \in \mathbb{Z}^{d + 1}$}\label{sec:bircob}
The basic birational maps we consider are generalizations of the Mori fibration $\pi : \mathbb{P}^d \to pt$, a blow-down morphism and an Atiyah flip. The generalization is to the stacky setting and can be accomplished by viewing these maps as a birational cobordism \cite{wlodarczyk} or, more precisely, as variations of GIT quotients \cite{bfk, hl}. We will use the element $\mathbf{a} \in \mathbb{Z}^{d + 1}$ satisfying equation~\eqref{eq:balance} and, reordering the coordinates if necessary, we assume that $a_{d + 1} < 0$. Consider the $\mathbb{C}^*$-action on $\mathbb{C}^{d + 1}$ given by 
\begin{align*}
	t \cdot \left(z_0 , \ldots, z_{d} \right) = \left( t^{a_0} z_1 , \ldots, t^{a_{d }} z_{d } \right).
\end{align*}
Following Reid \cite{reid}, we take 
\begin{align*}
B_- = \{(z_0, \ldots, z_{d }): z_i = 0 \textnormal{ for } a_i < 0\}, \\ B_+ = \{(z_0, \ldots, z_{d}): z_i = 0 \textnormal{ for } a_i > 0\}.
\end{align*}
We write $X^{\mathbf{a}}_\pm = \left( \mathbb{C}^{d + 1} \backslash B_\pm \right) / \mathbb{C}^*$, $X = \mathbb{C}^{d + 1} $  and define the birational map 
\begin{equation} \label{eq:birmap}
\begin{tikzpicture}[baseline=(current  bounding  box.center), scale=1.5]
\node (A) at (0,0) {$[X / \mathbb{C}^*]$};
\node (B) at (-1,1) {$X_+^{\mathbf{a}}$};
\node (C) at (1,1) {$X_-^{\mathbf{a}}$};
\path[->,font=\scriptsize]
(B) edge node {} (A)
(C) edge node {} (A);
\path[dashed,->,font=\scriptsize]
(B) edge node[above]{$f_{\mathbf{a}}$} (C);
\end{tikzpicture} 
\end{equation}	
There are toric models for $X_\pm^{\mathbf{a}}$ which were explored in \cite{morelli}. Indeed, every elementary birational map $f : Y_1 \dashrightarrow Y_2$ of DM toric stacks that has zero dimensional center is obtained by applying $f_{\mathbf{a}}$ or its inverse to an open subset $X^{\mathbf{a}}_\pm \subset Y_1$. Owing to the fact that $a_{d + 1} <  0$, results from \cite[Theorem~1]{bfk}, \cite{hl, kawamata}, and in the non-stacky case \cite{bo}, there is a semi-orthogonal decomposition 
\begin{align} \label{eq:sodecomp}
	D (X^{\mathbf{a}}_+ ) \cong \left< \mathcal{T}^B_{\mathbf{a}}, D (X^{\mathbf{a}}_-) \right>.
\end{align}
Here, and throughout the paper, we write $D(Y)$ for the bounded derived category of coherent sheaves on a stack $Y$. 

To state this result more explicitly and make use of the extraordinary machinery developed in \cite{bfk}, we briefly review their setup and one of their main theorems, rewritten slightly to simplify the application to the results of this paper. Take $Y$ to be a smooth variety with the action of a reductive linear algebraic group $G$ and two $G$-equivariant line bundles $\mathcal{L}_\pm$. Let  $Y (-) = Y^{ss}(\mathcal{L}_-)$ and $Y(+) = Y^{ss} (\mathcal{L}_+)$ be the semi-stable points and $Y_\pm = [Y (\pm ) / G]$ their quotient stacks. Assume that 
\begin{enumerate}
	\item $\mathcal{L}_t = \mathcal{L}_-^{\frac{1 - t}{2}} \otimes \mathcal{L}_+^{\frac{1 + t}{2}}$ has constant semi-stable locus $ Y^{ss}(\mathcal{L}_t) = Y(-)$ for $t \in [-1, 0)$ and $Y^{ss}(\mathcal{L}_t) = Y (+) $ for $t \in (0, 1]$. 
	\item $Y^{ss}(0) \backslash \left( Y(-) \cup Y (+) \right)$ is connected and for any element $y$, the stabilizer $G_y$ is isomorphic to $\mathbb{G}_m$.
\end{enumerate}
These conditions ensure that there is a one-parameter subgroup $\lambda : \mathbb{G}_m \to G$ and a connected component $Z_\lambda$ of the fixed locus $Y^\lambda$ of $\lambda$ satisfying $Y^{ss} (0) = Y (\pm) \sqcup S_{\pm \lambda}$. Here we write $S_{\pm \lambda}$ for those points $y \in Y$ for which $\lim_{t\to 0} \lambda(t^{\pm 1}) \cdot y \in Z_\lambda$ for $t \in \mathbb{G}_m$. 

Take $C(\lambda)$ to be the centralizer of $\lambda$ in $G$ and assume there is a splitting $C(\lambda ) = \lambda \times G_\lambda$. Let $[Y^\lambda \sslash G_\lambda]$ be the GIT quotient of the $\lambda$ fixed locus by $G_\lambda$ using the polarization $\mathcal{L}_0$. Finally, write $\mu$ for the weight of $\lambda$ on the anti-canonical bundle of $X$ along $Z_\lambda$.

\begin{theorem}[{\cite[Theorem~1]{bfk}}] \label{thm:bfk}
	If $\mu > 0$ then there are fully faithful functors $\Phi_d^+ : D (Y_- ) \to D (Y_+)$ and, for $0 \leq j \leq \mu - 1$,  $\Upsilon_j^+ : D ([Y^\lambda \sslash G_\lambda]) \to D (Y_+ )$ yielding a semi-orthogonal decomposition.
	\begin{align*}
	 D (Y_+ ) = \left< \Upsilon_0^+, \ldots, \Upsilon_{\mu - 1}^+, D (Y_-) \right> .
	\end{align*}
\end{theorem}

This theorem shows that in the decomposition \eqref{eq:sodecomp}, the category $\mathcal{T}^{B}_{\mathbf{a}}$ may be further decomposed into the subcategories $\Upsilon^+_j$. In fact, in our case, the categories $\Upsilon^+_j = \left< E_j \right>$ are generated by a single exceptional object $E_j$ and the semi-orthogonal decomposition of $\mathcal{T}^B_\mathbf{a}$ gives a full exceptional collection. We now give an explicit description of the $B$-model category $\mathcal{T}^{B}_{\mathbf{a}}$ associated to $f_{\mathbf{a}}$. We will utilize more results from \cite{bfk} in the proof, but only quote them.
\begin{theorem} \label{thm:Bmodel}
Let $\nu$ be any function for which $\nu ( i ) = 0$ for all $i \ne d + 1$. Then there is an equivalence of categories \begin{align*} \mathcal{T}^B_{\mathbf{a}} \cong \mathcal{D}_{\mathbf{a}, \nu, - a_{d + 1}}.
	\end{align*}
\end{theorem}

\begin{proof}
	We first identify the stacks $X_\pm$ introduced above with those using the notation of Theorem~\ref{thm:bfk}. Our space $X = \mathbb{A}^{d + 1}$ equals $Y$ and the group $G$ acting on $X$ is simply $\mathbb{C}^*$ with the one parameter subgroup $\lambda : \mathbb{G}_m \to G$ as the identity. The $G$-equivariant line bundles $\mathcal{L}_+$ and $\mathcal{L}_-$ correspond to any positive and negative characters of $\mathbb{C}^*$, respectively. It is easy to see that $B_\pm = S^\lambda_\pm$. As $G$ is abelian, its centralizer $C (\lambda ) = G = \mathbb{C}^*$ while its quotient $G_\lambda = \{1\}$.  As $a_i \ne 0$ for all $i$, the fixed locus of the $\mathbb{C}^*$-action on $X$ is simply $Z_\lambda = (0, \ldots, 0) = Y^\lambda$ and thus its quotient $Y^\lambda \sslash G_\lambda$ also equals a point. Furthermore, the parameter $\mu$ which is the weight of $\lambda$ acting on the anti-canonical bundle of $X$ along $Z_\lambda^0$ is given by
	\begin{align*} t \cdot \textnormal{d}z_0 \wedge \cdots \wedge \textnormal{d} z_{d } & = \textnormal{d}(t^{a_0} z_0) \wedge \cdots \wedge \textnormal{d}(t^{a_{d }} z_{d }) ,  \\ & = t^{a_0 + \cdots + a_{d}} \left(  \textnormal{d}z_0 \wedge \cdots \wedge \textnormal{d} z_{d } \right), \\ & = t^{-a_{d + 1}} \left(  \textnormal{d}z_0 \wedge \cdots \wedge \textnormal{d} z_{d} \right). \end{align*}
	Here we utilized the balancing condition~\eqref{eq:balance} and have $\mu = - a_{d + 1} > 0$.
	
	For any $k \in \mathbb{Z}$, let $\mathcal{O}_{X_\pm} (k)$ be the line bundle on $X_\pm$ obtained from the equivariant line bundle on $X \backslash B_\pm$ with character $k$. By Theorem~\ref{thm:bfk}, the  functor $\Upsilon_k^+ : D (Y^\lambda \sslash G_\lambda) \to D (X_+ )$ for any $k \in \mathbb{Z}$ is defined by sending the structure sheaf $\mathcal{O}_{pt}$ over the point $Y^\lambda \sslash G_{\lambda}$ to the structure sheaf $\mathcal{O}_{[B_- \sslash \lambda]} \otimes \mathcal{O}_{X_+} (k)$. This functor is fully faithful (as $Y^\lambda \sslash G_\lambda$ is a point) and Theorem~\ref{thm:bfk} asserts that
	\begin{align*} \mathcal{B} := \{ \Upsilon^+_{0} (\mathcal{O}_{pt}), \ldots,  \Upsilon^+_{-a_{d + 1} - 1} (\mathcal{O}_{pt}) \}
	\end{align*}
	is isomorphic to a full exceptional collection $\mathcal{E} = \{E_0, \ldots, E_{a_0 - 1}\}$ for $\mathcal{T}^B_{\mathbf{a}}$. Furthermore, application of \cite[Corollary~3.4.8]{bfk} shows that the sheaves $\Upsilon^+_{k} (\mathcal{O}_{pt})$ can be replaced with sheaves $\mathcal{O}_{B_-} (k)$ in the so-called window subcategory of $D ([X / G]) = D^{eq} (X)$. We use this, along with some elementary observations about Ext-groups to compute the $A_\infty$-algebra associated to $\mathcal{B}$.
	
	First, since $X = \mathbb{C}^{d +1}$ is affine, we have that $D^{eq} (\mathbb{C}^{d+1})$ is the homotopy category of the DG category of chain complexes in the category of finitely generated, graded $S = \mathbb{C}[z_0, \ldots, z_{d}]$-modules. Take $W$ to be the bigraded vector space generated by $\{w_i : a_i < 0 \}$ with $|w_i| = (\deg (w_i) , \wt (w_i) ) = (1, a_i)$. Then grading $S \otimes_{\mathbb{C}} \symalg{W} $ by degree, the differential $d$ defined by multiplication by $ \sum_{a_i < 0} z_i w_i$ yields the Koszul complex resolving $\mathcal{O}_{B_-}[\sum_{a_i < 0} a_i]$. Here again, $\symalg{W}$ is meant to be the graded-symmetric algebra of a graded vector space, and in this case is simply an exterior algebra. Using this resolution, and assuming $\nu = 0$, one computes cohomology to obtain an isomorphism
	\begin{align*}
		\Ext_{S}^* (\mathcal{O}_{B_-} , \mathcal{O}_{B_-} ) & = R_{\mathbf{a}, \nu}.
	\end{align*}
	This isomorphism respects the weight grading. Furthermore, the equivalence is induced by a quasi-isomorphism of differential graded algebras.
	This implies that 
	\begin{align*} 
		\Ext^* \left( E_i, E_j \right) & \cong \Ext^* \left( \mathcal{O}_{B_-} (i) , \mathcal{O}_{B_-} (j) \right), \\
		& = \{ v \in R_{\mathbf{a}, \nu} : \wt (v) = j - i \}, \\ & = \Hom_{\mathcal{C}_{\mathbf{a}, \nu, - a_{d + 1}} }\left( R_{\mathbf{a}} (i), R_{\mathbf{a}}(j) \right).
		\end{align*}

	As these isomorphisms are compatible with the DG structure and composition, this implies that $\mathcal{B}$ is a formal exceptional collection and that $\mathcal{T}^B_{\mathbf{a}} \cong \mathcal{D}_{\mathbf{a}, \nu , -a_{d + 1}}$. Finally, to see that we may allow $\nu ({d +1})$ to equal any value, observe that, since $\mu = |\mathcal{B}| < - a_{d + 1}$, there is no morphism corresponding to  $v_{d + 1}$ between objects in the collection $\mathcal{B}$, or in $\mathcal{B}_{\mathbf{a}, - a_{d + 1}}$. 
\end{proof}
\begin{example} The most elementary non-trivial example of $\mathbf{a} = (1, 1, -2)$ yields the classical  Beilinson~\cite{beilinson} exceptional collection $\mathcal{O}$ and $\mathcal{O} (1)$. Theorem~\ref{thm:Bmodel} gives the following table.
	\begin{table}[h] \def\arraystretch{1.5} \begin{tabular}{c | c | c | c}  $\mathbf{a}$ & $R_{\mathbf{a}}$ & $X_+ / \mathbb{C}^*$ & $X_- / \mathbb{C}^*$  \\ \hline $(1,1,-2)$ & $\textnormal{Sym}^* [v_0, v_1] \otimes \bigwedge^* (v_2)$ & $\mathbb{P}^1$ & $ \emptyset$ \end{tabular} \end{table}
	\begin{figure}[h]
		\begin{tikzpicture}[cross line/.style={preaction={draw=white, -, line width=6pt}}, scale=2.3]
		\node (A) at (0,0) {$R_{\mathbf{a}} (0)$};
		\node (B) at (1,0) {$R_{\mathbf{a}} (1)$};
		\path[->,font=\scriptsize]
		(A.10) edge node[above] {$v_1$} (B.170)
		(A.350) edge node[below] {$v_0$} (B.190);
		\end{tikzpicture}
		\caption{\label{fig:quiv1} Full exceptional collection for $\mathcal{D}_{\mathbf{a}, \nu, 2} \cong D (\mathbb{P}^1)$.}
	\end{figure}
\end{example}
The more general case of signature $(d, 0)$ will yield a full exceptional collection of a weighted projective space.
\begin{example} Again we choose a basic one dimensional example and observe that, while larger $|a_{d + 1}|$ provides several exceptional objects, the structure is not significantly more complicated.
	
\begin{table}[h] \def\arraystretch{1.5} \begin{tabular}{c | c | c | c}  $\mathbf{a}$ & $R_{\mathbf{a}}$ & $X_+ / \mathbb{C}^*$ & $X_- / \mathbb{C}^*$  \\ \hline $(2, 3, -5)$ & $\textnormal{Sym}^* [v_0, v_1] \otimes \bigwedge^* (v_2)$ & $\mathbb{P} (2, 3)$ & $ \emptyset$ \end{tabular} \end{table}
\begin{figure}[h]
	\begin{tikzpicture}[cross line/.style={preaction={draw=white, -, line width=6pt}}, scale=2.3]
	\node (A) at (0,0) {$R_{\mathbf{a}} (0)$};
	\node (B) at (1,0) {$R_{\mathbf{a}} (1)$};
	\node (C) at (2,0) {$R_{\mathbf{a}} (2)$};
	\node(D) at (3,0) {$R_{\mathbf{a}} (3)$};
	\node (E) at (4,0) {$R_{\mathbf{a}} (4)$};
	\path[->,font=\scriptsize]
	(A) edge[bend left=17] node[above] {$v_0$} (C)
	(B) edge[cross line, bend left=17] node[above] {$v_0$} (D)
	(A) edge[bend left=30] node[above] {$v_1$} (D)
	(C) edge[cross line, bend left=17] node[above] {$v_0$} (E)
	(B) edge[cross line, bend left=30] node[above] {$v_1$} (E);
	\end{tikzpicture}
		\caption{\label{fig:quiv2} Full exceptional collection for $\mathcal{D}_{\mathbf{a}, \nu, 5} \cong D (\mathbb{P} (2, 3))$.}
\end{figure}

\end{example}

Finally, we describe an elementary birational cobordism with $(p, q) \ne (d , 0)$.

\begin{example}
As was mentioned in the introduction, elementary birational cobordisms describe blow-ups of points and flips, as well as projective spaces. When the signature $(p,q)$ is $(d-1, 1)$, the cobordism produces the weighted blow-up of a point. In this case, the category $\mathcal{D}_{\mathbf{a}, \nu, n}$ does not describe the entire derived category $D (X_+)$, but rather the semi-orthogonal component $\mathcal{T}_{\mathbf{a}}^B$ obtained from blowing up. The associated sheaves are equivariant structure sheaves of the (stacky) exceptional divisor tensored with a line bundle.

\begin{table}[h] \def\arraystretch{1.5} \begin{tabular}{c | c | c | c}  $\mathbf{a}$ & $R_{\mathbf{a}}$ & $X_+ / \mathbb{C}^*$ & $X_- / \mathbb{C}^*$  \\ \hline $(1, 2, 3, -1, -5)$ & $\textnormal{Sym}^* [v_0, v_1, v_2] \otimes \bigwedge^* ( v_3, v_4 )$ & $\mathcal{O}_{\mathbb{P} (1,2,3)}(-1)$ & $\mathbb{C}^3$ \end{tabular} \end{table} 
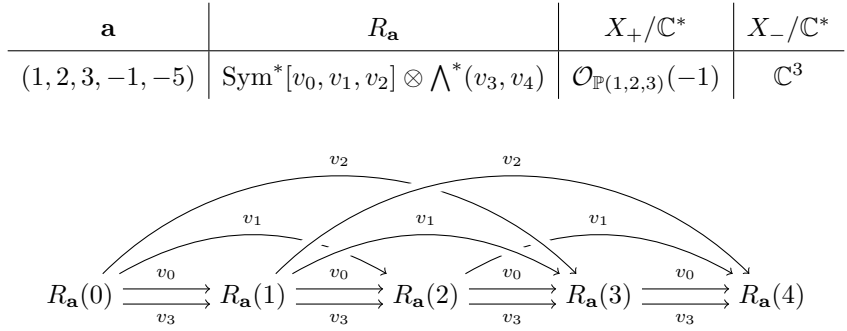
\begin{figure}[h]
	\begin{tikzpicture}[baseline=(current  bounding  box.center), cross line/.style={preaction={draw=white, -, line width=6pt}}, scale=2.3]
	\node (A) at (0,0) {$R_{\mathbf{a}} (0)$};
	\node (B) at (1,0) {$R_{\mathbf{a}} (1)$};
	\node (C) at (2,0) {$R_{\mathbf{a}} (2)$};
	\node (D) at (3, 0) {$R_{\mathbf{a}} (3)$};
	\node (E) at (4, 0) {$R_{\mathbf{a}} (4)$};
	\path[->,font=\scriptsize]
	(A.10) edge node[above] {$v_0$} (B.170)
	(B.10) edge node[above] {$v_0$} (C.170)
	(C.10) edge node[above] {$v_0$} (D.170)
	(D.10) edge node[above] {$v_0$} (E.170)
	(A.350) edge node[below] {$v_3$} (B.190)
	(B.350) edge node[below] {$v_3$} (C.190)
	(C.350) edge node[below] {$v_3$} (D.190)
	(D.350) edge node[below] {$v_3$} (E.190)
	(A) edge[bend left=30] node[above] {$v_1$} (C)
	(C) edge[bend left=30] node[above] {$v_1$} (E)
	(B) edge[cross line, bend left=30] node[above] {$v_1$} (D)
	(A) edge[cross line, bend left=45] node[above] {$v_2$} (D)
	(B) edge[cross line, bend left=45] node[above] {$v_2$} (E);
	\end{tikzpicture} 
	\caption{\label{fig:quiv3} Full exceptional collection for $\mathcal{D}_{\mathbf{a}, \nu, 5} \cong \mathcal{T}_{\mathbf{a}}^B$.}
\end{figure}
\end{example}

\section{$A$-model} \label{sec:A}

In this section we analyze the $A$-model mirror to the birational cobordism discussed in Section~\ref{sec:bircob}. 

\subsection{Fukaya-Seidel categories over $\mathbb{C}^*$}
We first consider the problem of defining a Fukaya-Seidel category $\mathcal{T}^A_{\mathbf{a}}$ associated to a mirror potential with values in $\mathbb{C}^*$. We will review and establish some of the notation and constructions in Fukaya-Seidel categories which are particularly important for the main results of this article. The seminal reference for the definition and basic properties of the usual Fukaya-Seidel category  is \cite{seidel} (where it is called the directed Fukaya category) and we do not intend to posit a unique definition. However, as the base of the LG model $W_\mathbf{a}$ is not simply connected, we will have to provide some additional data to obtain a well defined category as in \cite{kerr}. We will do this for a slightly broader class of potentials than the mirror potential $W_{\mathbf{a}}$. For the rest of this section, assume $E$ is a $d$-dimensional K\"ahler manifold with a nowhere-zero holomorphic volume form $\eta_E$, $S$ is a Riemannian surface and $\pi : E \to S$ is a Lefschetz fibration. We will write $\eta^2_E$ for a quadratic volume form on $E$ which allows Lagrangian submanifolds to be graded. Denote the critical points of $\pi$ by $\cp{\pi}$ and the critical values by $\cv{\pi}$.  Given a base point $* \in S$, a $\pi$\textbf{-admissible path} $\delta : [0, 1] \to S$ is a smoothly embedded curve with $\delta (0) = *$, $\delta (t) \not\in \cv{\pi}$ for all $t \ne 1$ and $\delta (1) \in \cv{\pi}$. 

Away from the critical values of $\pi$, one can define a symplectic connection on $E$ by taking the symplectic orthogonal of the tangent spaces to the fibers as the horizontal distribution. In particular, for any path $\gamma : [0, b) \to S \backslash \cv{\pi}$, and any $0 \leq t < b$ one can perform symplectic parallel transport 
\begin{align*} \mathbf{P}_{\gamma, t} : \pi^{-1} (\gamma (0)) \to \pi^{-1} (\gamma (t)). \end{align*}
Given a $\pi$-admissible path $\delta$ for which the critical point $p$ lies above $\delta (1)$, one defines the \textbf{vanishing cycle} and the \textbf{vanishing thimble} of $\delta$ as 
\begin{align} 
 \label{eq:vc}\vc_\delta &:= \left\{ e \in \pi^{-1} (*) : \lim_{t \to 1} \mathbf{P}_{\delta, t} (e) = p \right\}, \\
 \label{eq:vt}\vt_\delta & := p \cup \left( \cup_{t \in [0, 1)} \mathbf{P}_{\delta, t} (\vc_\delta ) \right).
\end{align}
Using the fact that $\pi$ is a Lefschetz fibration, there is a local model near the critical points of $\pi$ and one can show that $\vc_\delta$ is an exact Lagrangian $(d -1)$-sphere bounding the Lagrangian ball $\vt_\delta$. 

Given a Lagrangian thimble $\vt \subset E$, one may consider a branched double cover of $E$ and produce a sphere inside this cover by gluing two copies of $\vt$ along the branch locus. Precisely, if $f : \tilde{S} \to S$ is a branched cover of Riemann surfaces with ordinary double points we take $f^* \pi : f^*E \to \tilde{S}$ to be the pullback of $\pi : E \to S$. Then $f^* \pi$ is a Lefschetz fibration and, if the conditions spelled out below are satisfied, its construction allows one to define exact Lagrangian $d$-spheres in $f^* E$.

\begin{definition} \label{def:vs} Given  a Lefschetz fibration $\pi : E \to S$ and a map $f : \tilde{S} \to S$ with ordinary double points, a $\pi$-admissible path $\delta$ is called $(\pi, f)$-admissible if $\delta (t) \in \cv{f}$ exactly when $t = 0$. Write $\vs_\delta^f \subset f^* E$ for the Lagrangian sphere equal to the pullback $f^* \vt_{\delta}$.
\end{definition} 

A set $\mathcal{B} = \{\delta_1, \ldots , \delta_m\}$ of $\pi$-admissible paths will be called a $\pi$\textbf{-admissible collection} if 
\begin{enumerate}
	\item for $i \ne j$,  $\delta_i (t) = \delta_j (s)$ if and only if $t = 0 = s$,
	\item $\delta_1^\prime (0), \ldots, \delta_m^\prime (0)$ is ordered clockwise about $*$.
\end{enumerate}
When $S$ is simply connected and $\cv{\pi} = \{\delta_1 (1), \ldots, \delta_m (1)\}$, we say that $\mathcal{B}$ is a $\pi$\textbf{-distinguished basis}. In this case, the Fukaya-Seidel category of $\pi$ can be defined using the Fukaya category of the fiber $\pi^{-1} (*)$ along with the ordering of the paths in $\mathcal{B}$. This is done using the general construction given in Definition~\ref{def:directedcategory} of a directed subcategory of a given $A_\infty$ category $\mathcal{C}$. For each path $\delta_i \in \mathcal{B}$, we decorate $\vc_{\delta_i}$ with a grading and relative Pin structure to obtain a Lagrangian brane $L_i$ in the Fukaya category of the fiber $\mathcal{F} (\pi^{-1} (*))$. Let $\mathbf{B} = \{L_1, \ldots, L_m\} \subset \mathcal{F} (\pi^{-1} (*))$ be the ordered collection of Lagrangian branes.

\begin{definition} \label{defn:fscat}
	If $S$ is simply connected, the Fukaya-Seidel category $\mathcal{F} (\pi)$ of $\pi$ is the category of twisted complexes $\textnormal{Tw} (\mathcal{F} (\pi^{-1} (*))_\mathbf{B})$. 
\end{definition}

In fact, this is but one of several equivalent formulations of $\mathcal{F} (\pi )$.  It is a non-trivial fact \cite[Theorem~17.20]{seidel} that $\mathcal{F} (\pi)$ does not depend on the choice of distinguished basis. Moreover, noting that the objects $L_i \in \mathcal{F} (\pi)$ form a full exceptional collection, it was shown in loc. cit. that different choices of bases are related algebraically through mutations of exceptional collections.

To provide a definition for the case of a non-simply connected surface $S$, we take the easiest approach and equip ourselves with the additional data of a simply connected region $S^\circ$ in $S$ which contains $\cv{\pi}$. Our category is then  
\begin{align*}  \mathcal{F} (\pi, S^\circ) := \mathcal{F} (\pi |_{\pi^{-1}(S^\circ)}) . \end{align*}  In general, making different choices of domains $S^\circ$ will drastically alter the category. However, for certain potentials over $S = \mathbb{C}^*$, we will see that this choice does not have such an effect.  We now restrict to the case of $S = \mathbb{C}^*$. 
\begin{definition}
	A Lefschetz fibration $\pi : E \to \mathbb{C}^*$ is called atomic if it has a unique critical point.
\end{definition}

Given an atomic Lefschetz fibration $\pi$, we will denote its critical point by $p$ and its critical value by $q$. 

\begin{example} \label{eg:hmsp}
	One of the most important examples of an atomic LG model comes from a quotient of the mirror potential to a weighted projective space $\mathbb{P} (a_0, \ldots, a_d)$. We recall the construction of the mirror potential from \cite{givental, hv}. Let $\bar{\mathbf{a}} = \left(a_0, \ldots, a_d\right)$ and consider the subtorus
\begin{align}\label{eq:weightedprojmirror}
G_{\bar{\mathbf{a}}} := \left\{ (z_0, \ldots, z_d) \in (\mathbb{C}^*)^{d + 1}: \prod_{i = 0}^d z_i^{a_i} = 1 \right\}. 
\end{align}
Then the mirror potential $\mathbf{w}_{\bar{\mathbf{a}}} : G_{\bar{\mathbf{a}}} \to \mathbb{C}$ to $\mathbb{P} (a_0, \ldots, a_d)$ is the restriction of $\mathbf{w} = z_0 + \cdots + z_d$ to $G_{\bar{\mathbf{a}}}$. Letting $n = \sum_{i = 0}^d a_i$, one can easily compute that the critical points of $\mathbf{w}_{\bar{\mathbf{a}}}$ are 
\begin{align*}
\cp{\mathbf{w}_{\bar{\mathbf{a}}}} & = \left\{ (a_0 t, \ldots, a_d t) : t^{n} = \prod_{i = 0}^d a_i^{-a_i} \right\}.
\end{align*}
	
Note that there are $n$ critical points and that the diagonal action of the group of $n$-th roots of unity $\units_n$ on $G_{\bar{\mathbf{a}}}$ restricts to a transitive action on $\cp{W_{\bar{\mathbf{a}}}}$. Also note that $\mathbf{w}_{\bar{\mathbf{a}}}$ is equivariant with respect to this action. Thus we may quotient both the total space and the base to obtain 
\begin{align} \label{eq:barpi}
\bar{\pi} : \frac{G_{\bar{\mathbf{a}}}}{\units_n} \to \frac{\mathbb{C}}{\units_n}.
\end{align}
The function $\bar{\pi}$ has exactly one critical point away from the zero fiber where the action on $\mathbb{C}$ is fixed point free. Over zero $\bar{\pi}$ is branched and therefore singular. However, if we excise the zero fiber, we obtain the atomic LG model
\begin{align*}
\pi : \frac{G_{\bar{\mathbf{a}}} \backslash (\mathbf{w}_{\bar{\mathbf{a}}}^{-1} (0)) }{\units_n} \to \mathbb{C}^*.
\end{align*}
\end{example}

\begin{figure}
\includegraphics[scale=.9]{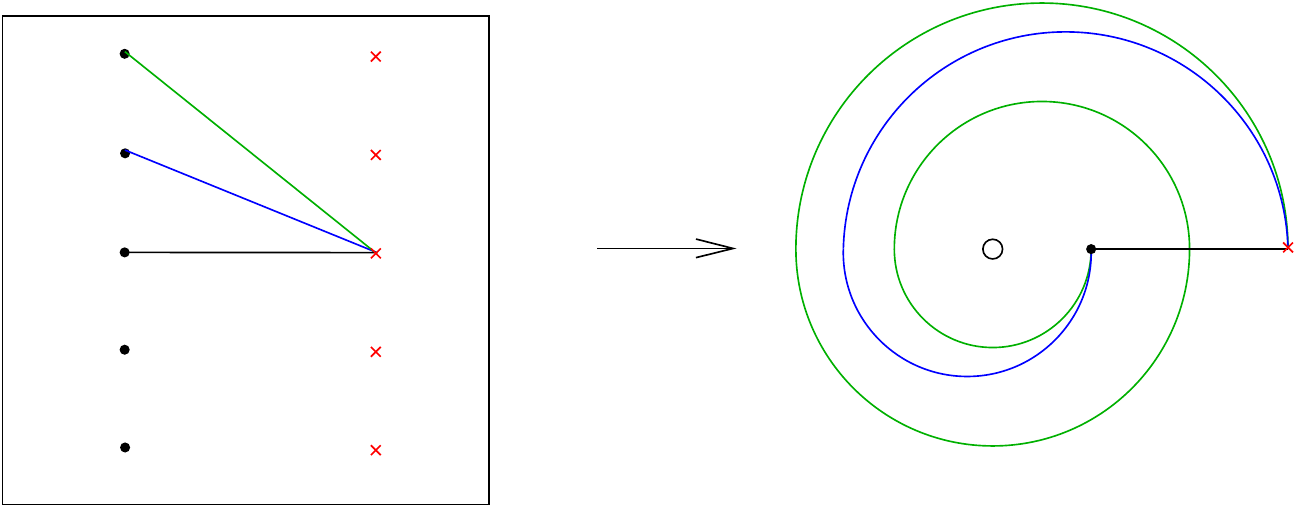}
\caption{\label{fig:paths} Three logarithmic paths and their images in $\mathbf{B}_*$.}
\end{figure}

Returning to the general case of an atomic Lefschetz fibration $\pi : E \to \mathbb{C}^*$, choose a base point $* = e^u q$ so that $\Re (u) \ne 0$. By choosing different logarithms $u$ of $* / q$, we may index the set of isotopy classes 
\begin{align*} \left\{ [\delta^{u + 2  \pi i k} :  k \in \mathbb{Z} \right\}
 \end{align*} of $\pi$-admissible paths from $*$ to $q$ where
\begin{align} \label{eq:deltadef}
\delta^u (s) = e^{(1 - s)u} q .
\end{align}
It is sometimes useful to vary the basepoint $u$, so we will frequently use the notation in equation \eqref{eq:deltadef}. However, when we fix a base point $* = e^u q$, we will take $u_* \in \mathbb{C} $ be the unique logarithm of $*/ q$ for which $ 0 \leq \Im (u_* ) < 2 \pi$. We write $\delta_k$ for $\delta^{u_* + 2\pi i k}$. Three of the logarithmic paths $\ln (\delta^u)$ and their images are illustrated in Figure~\ref{fig:paths}.

\begin{definition}
	Given an atomic Lefschetz fibration $\pi : E \to \mathbb{C}^*$, the $n$-unfolded Fukaya-Seidel category of $\pi$ is the category $\mathcal{F} (\pi^{1/n} ) := \mathcal{F} (\tilde{\pi}, S^\circ )$ where $\tilde{\pi}$ is the pullback of $\pi$ along $z \mapsto z^n$ and $S^\circ = \mathbb{C}^* \backslash \left( e^{i \theta } \cdot \mathbb{R}_{> 0} \right)$ for $n \theta \not\equiv \arg (q ) \pmod{2\pi}$.
\end{definition}
As the next proposition indicates, the $n$-unfolded category of $\pi$ is independent of the choice of cut and can be computed directly from the vanishing cycles associated to $\delta_k$. To make this computation, we first choose a natural distinguished basis of paths and fix a coherent brane structure on their associated vanishing cycles. First we equip the vanishing thimble $\vt_{\delta_k} \in \pi^{-1} (*)$ with an arbitrary Lagrangian brane structure. This induces a brane structure on $\vc_{\delta_k}$ and we write the associated object as $L_k \in \mathcal{F} (\pi^{-1} (*))$. For $j \in \mathbb{N}$, we then equip $\vc_{\delta_{k + j}}$ with a brane structure induced by the one on $L_k$. This can be done by performing symplectic monodromy around $0$ exactly $j$ times, which is a graded symplectomorphism.  Write
\begin{align} \label{eq:naturalcol}
\mathbf{B} (k ) = \{L_k, L_{k + 1} , \ldots, L_{k + n - 1}\} 
\end{align}
for the associated collection of objects in $\mathcal{F} (\pi^{-1} (*))$.

\begin{proposition} \label{prop:unfoldedfuk} For any $k \in \mathbb{Z}$, there is an equivalence of categories \begin{align*} \mathcal{F} (\pi^{1/n}) \cong \textnormal{Tw} (\mathcal{F} (\pi^{-1}(*))_{\mathbf{B}(k)}). \end{align*}
\end{proposition}
\begin{proof}
	Consider the pullback $\tilde{\pi}$ of $\pi$ along $z \mapsto z^n$.
\begin{equation} \label{eq:pullback}
\begin{tikzpicture}[baseline=(current  bounding  box.center), scale=1.5]
\node (A) at (0,1) {$\tilde{E}$};
\node (B) at (1,1) {$E$};
\node (C) at (0,0) {$\mathbb{C}^*$};
\node (D) at (1,0) {$\mathbb{C}^*$};
\path[->,font=\scriptsize]
(A) edge node[above]{$\phi$} (B)
(A) edge node[left]{$\tilde{\pi}$} (C)
(B) edge node[right]{$\pi$} (D)
(C) edge node[above]{$z^n$} (D);
\end{tikzpicture} 
\end{equation}	
	The critical values of $\tilde{\pi}$ are the $n$-th roots of $q$. For any choice of $\theta$, let $\tilde{q}$ be the $n$-th root which follows $e^{i\theta}$ clockwise and take $\tilde{*} = e^{u} \tilde{q}$ for $u \in \mathbb{R}_{< 0}$. Using unique path lifting of $z \mapsto z^n$ with respect to the basepoint $*$ lifting to $\tilde{*} \in \mathbb{C} \backslash e^{i \theta} \mathbb{R}_{> 0}$, one obtains a unique lift $\tilde{\delta}_i : [0,1] \to \mathbb{C}^*$ for every $i$. It is clear that $\tilde{\delta}_i \subset \mathbb{C} \backslash e^{i\theta} \mathbb{R}_{> 0}$ if $0 \leq i < n$, so that $\tilde{\mathcal{B}} = \{\tilde{\delta}_0, \ldots, \tilde{\delta}_{n - 1}\}$ is a distinguished basis of paths for $\tilde{\pi}$. Take $\tilde{\mathbf{B}}$ to be the respective collection of vanishing cycles in $\mathcal{F} (\tilde{\pi}^{-1} (\tilde{*}))$. Note that $\phi$ yields a symplectomorphism from  $\tilde{\pi}^{-1} (\tilde{*})$  to $\pi^{-1} (*)$ and the symplectic connection of $\pi$ pulls back to that of $\tilde{\pi}$. Thus the vanishing cycle $\vc_{\tilde{\delta}_i}$ is mapped isomorphically to $\vc_{\delta_i}$ via $\phi$ so that $\phi$ induces an equivalence of categories from $\mathcal{F} (\tilde{\pi}^{-1} (\tilde{*}))$ to $\mathcal{F} (\pi^{-1} (*))$ taking $\tilde{\mathbf{B}}$ to $\mathbf{B}$ for $k = 0$. The result for $k = 0$ then follows  from Definition~\ref{defn:fscat}. For general $k$, we note that monodromy about $0$ induces an symplectomorphism from $\pi^{-1} (*)$ to itself which carries $\vc_{\delta_i}$ to $\vc_{\delta_{i + 1}}$ so that, for varying $k$, the associated directed subcategories with respect to $\mathbf{B}$  are equivalent.
\end{proof}

The class of atomic LG models that we consider arises from pencils on K\"ahler varieties. We utilize this additional structure and make a definition.

\begin{definition} \label{def:atomLP}
	Let $X$ be a projective variety with ample line bundle $\mathcal{L}$. Given two linearly equivalent, effective divisors $D_0$, $D_\infty$ such that $D_0 \cup D_\infty$ is a normal crossing divisor, we say the pencil \[[s_{D_0}: s_{D_\infty}] : X \backslash (D_0 \cap D_\infty) \to \mathbb{P}^1 \] is an atomic Lefschetz pencil if its restriction to $E = X \backslash (D_0 \cup D_\infty)$ is an atomic Lefschetz fibration.
\end{definition}

When an atomic Lefschetz fibration arises as an atomic Lefschetz pencil, one may define and utilize additional invariants. Assume that, for $i = 0$ or $\infty$,  $D_i = \sum_{j = 1}^{r_i} b_j D_i^j$ where $D_i^j$ is an irreducible component of $D_i$. Let $D_i^\circ = D_i \backslash (D_0 \cap D_\infty)$ and, for any $n \in \mathbb{N}$, extend the pullback in equation~\eqref{eq:pullback} 
\begin{equation} \label{eq:pullback2}
\begin{tikzpicture}[baseline=(current  bounding  box.center), scale=1.5]
\node (A) at (0,1) {$\tilde{E}_0$};
\node (B) at (1,1) {$E \cup D_0^\circ$};
\node (C) at (0,0) {$\mathbb{C}$};
\node (D) at (1,0) {$\mathbb{C}$};
\path[->,font=\scriptsize]
(A) edge node[above]{$\phi$} (B)
(A) edge node[left]{$\tilde{\pi}_0$} (C)
(B) edge node[right]{$\pi$} (D)
(C) edge node[above]{$z^{n}$} (D);
\end{tikzpicture} 
\end{equation}
 by adding $D_0^\circ$. Observe that for $n = b_j$,  $\tilde{\pi}_0$ is regular on the pullback of $D_0^j \backslash  (D_0^j \cap D_\infty)$ in the partial compactification  $\tilde{E}_0$ of $\tilde{E}$. Now, let 
 \[S = \{ z : z \in \mathbb{C}, |z| \leq |*|, z \ne * \} \] be the disc with boundary and $*$ removed as in the right hand side of  Figure~\ref{fig:bondcond}. We equip $S$ with a strip like end near $*$ which is a trivialization $\epsilon : \mathbb{R}_{> 0} \times [0,1] \to S$ for which $\lim_{s \to \infty} \epsilon (s, t) = *$ and $\epsilon (s, i) \in \partial S$ for $i \in \{0,1\}$ (see \cite[Section~8d]{seidel}).
 Let $\tilde{E}_S = \tilde{\pi}_0^{-1} (S)$ and $F$ be the union of the counter-clockwise parallel transports of $\vc_{\delta_0}$ along the boundary of $S$. So $\tilde{\pi}_0 : F \to \partial S$ gives a Lagrangian boundary condition for the fibration $\tilde{\pi}_0 : \tilde{E}_S \to S$ and one can consider the zero dimensional component $\mathcal{M}_{n}(\pi)$ in the moduli space  of sections of
 \begin{align} 
 \label{eq:totalsec} \tilde{\pi}_0 : (\tilde{E}_S, F) \to (S, \partial S).
 \end{align} 
For any such section $u : S \to \tilde{E}_S$, we have that, $\lim_{s \to \infty} u \circ \varepsilon (s, t) \in \vc_{\delta_0} \cap \tilde{\vc{}} = CF^* (\vc_{\delta_0} , \tilde{\vc})$ where $\tilde{\vc}$ is Hamiltonian isotopic to $\vc_{\delta_n}$. Moreover, for every $1 \leq j \leq r_0$ we define a subspace of $\mathcal{M}_n (\pi)$ which isolates those sections for which $u (0) \in D_0^{j}$. Precisely, for $p \in \mathcal{V}_{\delta_0} \cap \tilde{\mathcal{V}}$, $1 \leq j \leq n$ we take 
\[\mathcal{M}^{p, j}_{n} (\pi ) = 
\left\{ u \in \mathcal{M}_n (\pi ) :  u(0) \in D_0^j, \lim_{s \to \infty} u (\varepsilon (s, t)) = p  \right\} .\] Then we obtain the partition 

\begin{figure}[t]
\begin{picture}(0,0)%
\includegraphics{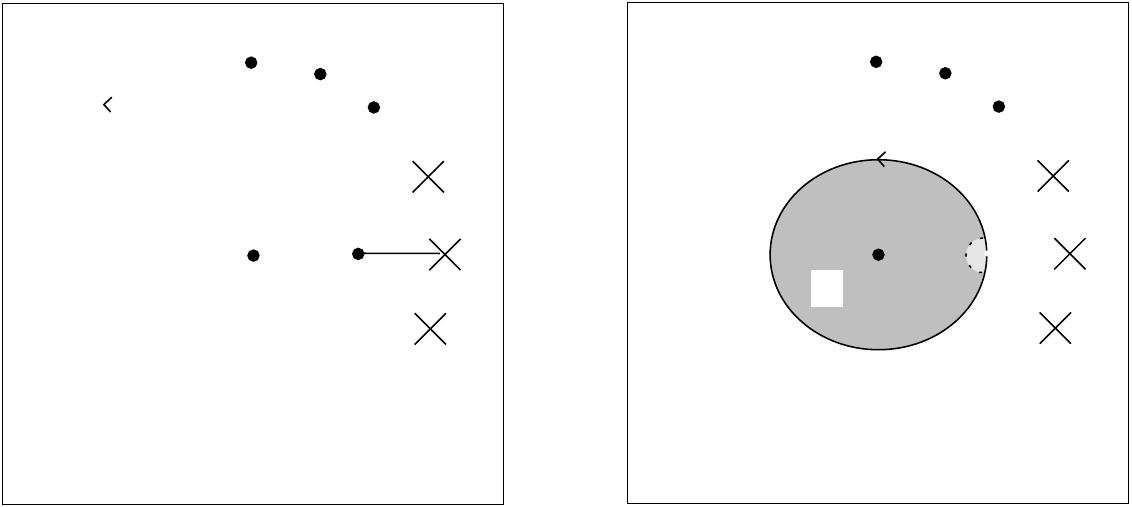}%
\end{picture}%
\setlength{\unitlength}{3947sp}%
\begin{picture}(5430,2432)(889,-2177)
\put(4792,-1190){\makebox(0,0)[lb]{\smash{$S$}}}
\put(2476,-886){\makebox(0,0)[lb]{\smash{$*$}}}
\put(5642,-880){\makebox(0,0)[lb]{\smash{$\vc_0$}}}
\put(5657,-1201){\makebox(0,0)[lb]{\smash{$\tilde{\vc}$}}}
\end{picture}%
\caption{\label{fig:bondcond}Disc with Lagrangian boundary condition}
\end{figure}
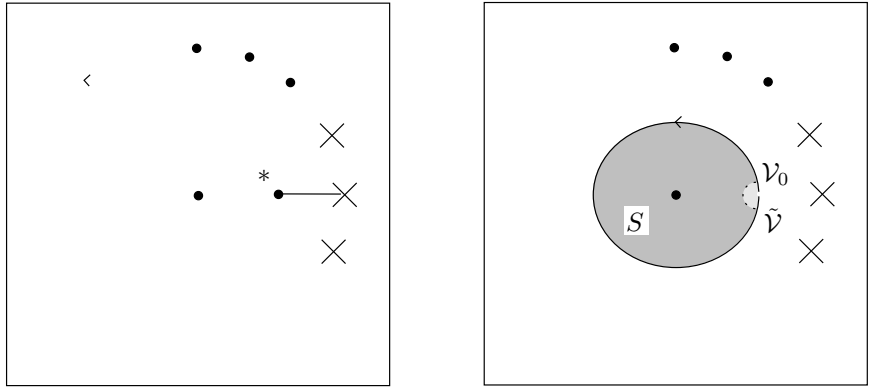

\begin{align} \label{eq:refmodspace} \mathcal{M}_n (\pi) = \bigsqcup_{p \in \vc_{\delta_0}  \cap \tilde{\vc}, 1 \leq j \leq r_0} \mathcal{M}^{p, j}_n (\pi). \end{align}
Using the Pin structure on $\vc_0$ to orient $\mathcal{M}_n (\pi )$, the usual TQFT formalism applies in this setting to show that the signed count $\# \mathcal{M}^{p, j}_n (\pi ) = 0$ for coboundaries $p \in CF^* (\mathcal{V}_{\delta_0}, \tilde{\mathcal{V}})$ so that, for all $1 \leq j \leq r_0$, one obtains the invariant
\begin{align*}
\kappa_{n}^j : H^* (\textnormal{Hom}_{\mathcal{F} (\pi^{-1} (*))} (\vc_{\delta_0} , \vc_{\delta_n})) \cong HF^* (\vc_{\delta_0}, \tilde{\vc}) \to \mathbb{Z}
\end{align*} 
where 
\begin{align} \label{eq:defkappa}
\kappa_{n}^j (p) = \# \mathcal{M}^{p , j}_n (\pi ).
\end{align}
This additional invariant for atomic Lefschetz pencils will be used in the induction proof for Theorem~\ref{thm:Amodel} in Section~\ref{sec:proof}.

\subsection{Circuit LG models} In this section we will review the mirror potential of an elementary birational cobordism and state the main theorem of this paper.  Consider $\mathbf{a} = (a_0, \ldots, a_{d + 1}) \in \mathbb{Z}^{d + 2}$ satisfying equations~\eqref{eq:balance} and of signature $(p,q)$. If necessary, permute the coordinates of $\mathbf{a}$ so that $a_i > 0$ for $0 \leq i \leq p$ and $a_i < 0$ for $p < i \leq d + 1$. The volume of $\mathbf{a}$ is defined to be the constant 
\begin{align} \label{eq:vola} \Vol (\mathbf{a}) := \sum_{i = 0}^p a_i . \end{align}
Let $[Z_0: \cdots : Z_{d + 1}]$ be homogeneous coordinates of $\mathbb{P}^{d + 1}$. For any subset $I \subseteq \{0, \ldots, {d +1}\}$ take 
\begin{align} \label{eq:coordplane} C_I = \{[Z_0: \cdots : Z_{d + 1}] : Z_i = 0 \textnormal{ for all } i \in I\}\end{align}
to be the associated coordinate plane.  Recall that  the $d$-dimensional pair of pants is the subvariety
\begin{align} 
\label{eq:pairpants} P_d := \left\{[Z_0: \cdots :Z_{d + 1}] \in \mathbb{P}^{d + 1} : \sum_{i = 0}^{d + 1}  Z_i  = 0, Z_i \ne 0 \right\} .
\end{align}
We will write $\bar{P}_d \cong \mathbb{P}^d$ for the closure of $P_d$ in $\mathbb{P}^{d + 1}$. Let  $ \mathcal{L}_{\mathbf{a}} = \mathcal{O}_{P_d} (\Vol  (\mathbf{a}))$ and define the divisors and sections  
\begin{equation} \label{eq:sections} 
\begin{split}
D_0  = \sum_{i = 0}^p a_i C_i |_{\bar{P}_d}, & \hspace{2cm} s_0  = \prod_{i = 0}^{p} Z_i^{a_i}, \\ D_\infty  = - \sum_{i = p + 1}^{d + 1} a_i C_i |_{\bar{P}_d}, & \hspace{2cm}
s_\infty  = \prod_{i = p + 1}^{d + 1} Z_i^{-a_i}.
\end{split}
\end{equation}
Observe that $B_{\mathbf{a}} = \cup_{i \leq p < j} C_{\{i,j\}}$ is the base locus of the associated pencil $\{D_{0}, D_{\infty} \}$.
\begin{definition}
	The circuit pencil associated to $\mathbf{a}$ on $\mathbb{P}^{d + 1}$ is defined by the map $\psi_{\mathbf{a}} : \mathbb{P}^{d + 1} \backslash B_{\mathbf{a}} \to \mathbb{P}^1$ given by $[s_0: s_\infty]$. The circuit LG model $W_{\mathbf{a}}$ is the restriction of $\psi_{\mathbf{a}}$ to the pair of pants $P_d$.
\end{definition}
We now review some facts from \cite{gkz, dkk} on the  basic properties of $W_{\mathbf{a}}$. 
\begin{lemma} \label{lem:circuit_basics}
The circuit LG model $W_{\mathbf{a}}$ is an atomic Lefschetz pencil with unique critical point $p_{\mathbf{a}} = \left[a_0: \cdots : a_{d + 1} \right]$ and critical value $q_{\mathbf{a}} = \prod_{i = 0}^{d + 1} a_i^{a_i}$. 
\end{lemma}
\begin{proof}
	This follows at once from a basic computation of the critical points for $W_{\mathbf{a}}$. Indeed, wedging the exterior derivative of the constraint $s([Z_0: \cdots : Z_{d + 1}]) = \sum Z_i = 0$ and $\psi_{\mathbf{a}}$ gives
	\begin{align*}
	 \tn s \wedge \tn \psi_{\mathbf{a}} = \psi_{\mathbf{a}} \sum_{i < j} \left( \frac{a_j}{Z_j} - \frac{ a_i}{Z_i} \right) \tn Z_i \wedge \tn Z_j . 
	\end{align*}
	So a critical point $[Z_0: \cdots :Z_{d + 1}]$ of $W_{\mathbf{a}}$ satisfies $\frac{Z_j}{a_j} = \frac{ Z_i}{a_i}$ for all $i$ and $\sum Z_i = 0$. This implies the result that $p  = [a_0: \cdots : a_{d + 1}]$ is the unique critical point. To see that it is a Morse critical point, one applies the complex Morse Lemma by computing the Hessian of $W_{\mathbf{a}}$ and observing that it is non-degenerate. We cite \cite[Proposition~2.13]{dkk} for this computation.
\end{proof}
To define the Fukaya-Seidel category associated to $W_\mathbf{a}$, and in particular to consistently grade the morphisms between Lagrangians, we specify a non-zero holomorphic volume form. We will keep some flexibility here, but index the possible choices by using our function $\nu$ from equation~\ref{eq:defnu}. Fix the meromorphic $(d + 1)$-form 
\begin{equation*}
\eta_{\mathbb{P}^{d + 1}} = \sum_{i = 0}^{d + 1}  \prod_{j \ne i} \frac{\tn Z_j}{Z_j}
\end{equation*}
which is holomorphic and non-zero on $\mathbb{P}^{d + 1} \backslash (D_0 \cup D_\infty)$. To obtain a non-zero $d$-form on $P_d$, assume that $\sum_{ i = 0}^{ d + 1} \nu (i) = 0$ and take
\begin{align*} \varrho_{\nu} & = \prod_{i = 0}^{d + 1} Z_i^{- \nu (i)}  \sum_{i = 0}^{d + 1} \frac{\tn Z_i}{Z_{d + 1}}
\end{align*}
to define 
\begin{align} \label{eq:etadef}
\eta_{\nu} = \eta_{\mathbb{P}^{d + 1}} / \varrho_{\nu}.
\end{align}
For later use, we express $\eta_{\nu}$ in local coordinates  $(z_1, \ldots, z_d) \in (\mathbb{C}^*)^d$ where $-1 \ne \sum z_i$ and pull back via the chart $[(-1 - \sum z_i): z_1: \cdots :z_d : 1] \in P_d$. Here a short computation shows
\begin{align} \label{eq:etalc}
\eta_\nu = \left(-1 - \sum z_i \right)^{\nu (0) - 1} \prod_{i = 1}^d z_i^{\nu (i) - 1} \tn z_1 \wedge \cdots \wedge \tn z_d .
\end{align}

The $A$-model category associated to $\mathbf{a}$ may then be defined as follows.
\begin{definition}
	Suppose $\mathbf{a}, \nu \in \mathbb{Z}^{d + 1}$ are balanced, $n \in \mathbb{N}$, $u \in \mathbb{R}_{< 0}$ and $* = e^u q_{\mathbf{a}}$. Denote by $\mathcal{A}_{\mathbf{a}, \nu, n}$ the directed subcategory $\mathcal{F} (W^{-1}_{\mathbf{a}} (*)) )_\mathbf{B}$ where $\mathbf{B} = \{L_0, \ldots, L_{n -1 }\}$ with grading induced by $\eta^2_{\nu}$. If $a_{d + 1} < 0$, let $\nu = 0$ and write $\mathcal{T}^A_{\mathbf{a}}$ for the $n$-unfolded Fukaya-Seidel category $\textnormal{Tw} (\mathcal{A}_{\mathbf{a}, \nu, -a_{d + 1}}) \cong \mathcal{F} (W_{\mathbf{a}}^{1/-a_{d + 1}})$.
\end{definition}
The definition of $\mathcal{A}_{\mathbf{a}, \nu, n}$ has an ambiguity up to Hamiltonian deformations of $L_i$, although it is well defined up to quasi-equivalence. To remove this ambiguity, given a collection $\mathcal{B} = \{L_0, \ldots, L_{n - 1}\} \subset W^{-1}_{\mathbf{a}} (*)$ for which $L_i$ is Hamiltonian equivalent to $\vc_{\delta_i}$ in $W^{-1}_\mathbf{a} (*)$, we will say ${\mathcal{B}}$ is a \textbf{Hamiltonian representative basis} of $W_{\mathbf{a}}$ and assert that $\mathcal{A}_{\mathbf{a}, \nu, n}$ is defined by $\mathcal{B}$.  The main theorem of the paper gives an explicit description of $\mathcal{A}_{\mathbf{a}, \nu, n}$ in cases where $n$ is bounded by the volume of $\mathbf{a}$.
\begin{theorem} \label{thm:Amodel}
	Given balanced $\mathbf{a}, \nu \in \mathbb{Z}^{d +1}$ and $n \in \mathbb{N}$ with $0 \leq n <  \Vol (\mathbf{a})$, there exists a Hamiltonian representative basis $\mathcal{B}$ of $W_{\mathbf{a}}$ defining $\mathcal{A}_{\mathbf{a}, \nu, n}$ and an isomorphism
	\begin{align} \label{eq:equivalence} \Xi_{\mathbf{a}, n} :  \mathcal{C}_{\mathbf{a}, \nu, n}  \to \mathcal{A}_{\mathbf{a}, \nu, n}  \end{align}
	for which \begin{enumerate}[label=(\roman*),ref=\thetheorem(\roman*)]
		\item \label{thm:Amodel:1} $\Xi_{\mathbf{a}, n} (R_{\mathbf{a}} (k)) = L_k $, 
		\item \label{thm:Amodel:2} $\Xi_{\mathbf{a}, n}$ takes the monomial basis vectors to the canonical basis of intersections.
	\end{enumerate} 
	Furthermore, for any $0 \leq j \leq p$, one has 
	\begin{align}
		\label{eq:divisor}
	\kappa_{a_j}^j (\Xi_{\mathbf{a}, n} (x)) & = \begin{cases} 1 & \textnormal{ if } x = v_j , \\ 0 & \textnormal{ otherwise.} \end{cases}
	\end{align}
\end{theorem}
This theorem will be proved by induction on $d$ in the following sections. For now, we note that Theorems~\ref{thm:Amodel} and \ref{thm:Bmodel}, yield the following important corollary.
\begin{corollary} \label{cor:main_result}
	The categories $\mathcal{T}_{\mathbf{a}}^B$ and $\mathcal{T}_{\mathbf{a}}^A$ are equivalent.
\end{corollary}
This corollary gives further evidence of the link between birational geometry on the $B$-model side and degenerations on the $A$-model side of mirror symmetry. As a special case of this corollary, we obtain the folklore result.
\begin{corollary} \label{cor:wps}
	The homological mirror symmetry conjecture holds for weighted projective spaces $\mathbb{P} (a_0, \ldots, a_d)$.
\end{corollary}
This result was proven in \cite{ako} for weighted projective planes and \cite{fu} for ordinary (unweighted) projective space. 
\begin{proof}[Proof of Corollary~\ref{cor:wps}]
Given positive weights $a_i$ with $0 \leq i \leq d$ take $\bar{\mathbf{a}} = (a_0, \ldots, a_{d})$, $a_{d + 1} = - \sum_{i = 0}^d a_i$ and $\mathbf{a} = (a_0, \ldots, a_{d + 1})$. The VGIT for this collection of weights is a transition from $X_+^{\mathbf{a}} = \mathbb{P} (a_0, \ldots, a_{d + 1})$ to the vacuum $X_-^{\mathbf{a}} = \emptyset$. By equation~\eqref{eq:sodecomp}, we have that $D (\mathbb{P} (a_0, \ldots, a_{d + 1}) ) = \mathcal{T}^{B}_{\mathbf{a}}$.

On the $A$-model side, recall from equation~\eqref{eq:weightedprojmirror} that $G_{\bar{\mathbf{a}}} = \{(z_0, \ldots, z_{d}) : \prod z_i^{a_i} = 1\}$. Write $D_{d + 1}$ for the divisor $\{[Z_0: \cdots : Z_{d + 1}] : Z_{d + 1} = 0, \sum_{i = 0}^{d} Z_i = 0\}$ in $\bar{P}_d$. Consider the map $\phi : G_{\bar{\mathbf{a}}} \to P_d \cup D_{d + 1}$ obtained by taking 
\begin{align*} \phi (z_0, \ldots, z_d) = \left[z_0: \cdots : z_{d}: - \sum z_i \right] . \end{align*}

It is not hard to see that $\phi$ is simply the quotient by $\mathbb{G}_{a_{d + 1}}$. Using the definitions of the mirror potential $\mathbf{w}_{\bar{\mathbf{a}}}$ of $\mathbb{P} (a_0, \ldots, a_n)$ in Example~\ref{eg:hmsp}  and of the quotient $\bar{\pi}$ from equation~\eqref{eq:barpi}, one observes that $W_\mathbf{a} \circ \phi = 1/\bar{\pi}$. This yields the fiber square
\begin{equation*} 
\begin{tikzpicture}[baseline=(current  bounding  box.center), scale=2]
\node (A) at (0,1) {$G_{\mathbf{a}}$};
\node (B) at (1,1) {$P_d \cup D_{d + 1}$};
\node (C) at (0,0) {$\mathbb{C}$};
\node (D) at (1,0) {$\mathbb{P}^1 \backslash [0:1]$};
\path[->,font=\scriptsize]
(A) edge node[above]{$\phi$} (B)
(A) edge node[left]{$-\mathbf{w}_{\mathbf{a}}$} (C)
(B) edge node[right]{$W_{\mathbf{a}}$} (D)
(C) edge node[above]{$z^{a_{d + 1}}$} (D);
\end{tikzpicture} 
\end{equation*}	
In particular, both the Fukaya-Seidel category  $\mathcal{F} (\mathbf{w}_{\mathbf{a}})$ and $\mathcal{T}_{\mathbf{a}}^{A}$ are quasi-equivalent to the category of twisted complexes of the algebra generated by vanishing cycles of a distinguished basis of paths for $\mathbf{w}_{\mathbf{a}}$. This yields the equivalence $\mathcal{T}^A_{\mathbf{a}} = \mathcal{F} (\mathbf{w}_{\mathbf{a}})$, and, applying Corollary~\ref{cor:main_result}, finishes the proof.
\end{proof}


\subsection{One-dimensional case}\label{sec:1dcase}
In this section we will consider the case of $d = 1$. Up to isotopy, we will describe the thimbles of $\delta_k$ for any given $k$ and utilize  Proposition~\ref{prop:unfoldedfuk} to prove equation~\eqref{eq:equivalence} in Theorem~\ref{thm:Amodel}.  We will first establish a bit of notation. 

Let $0 < \varepsilon$ be a small real number and 
$h : \mathbb{R}_{> 0} \to [0, 1]$ be a smooth function satisfying
\begin{enumerate}
	\item $h$ is  monotonically decreasing,
	\item $h (t) = 1 $ for $t \leq \varepsilon $ and,
	\item $h (t) = 0 $ for $t \geq 1$. 
\end{enumerate}

Now suppose $S$ is the surface with a point $s \in S$ and a local unit disc chart $U \subset S$ around $S$ with coordinate $z$. Given a real number $\theta$ we define a twist $\tau_{\theta}$ of $U$ by taking
\begin{align} \label{eq:twist}
\tau_{\theta, s} (z) = e^{i \theta h (|z|) )}
\end{align}
and extend $\tau_\theta$ to $S$ by taking the identity outside of $U$.
If $(s_1, \ldots, s_k)$ are a tuple of distinct points and $(\theta_1, \ldots, \theta_k)$ are real numbers we write $\tau_{(\theta_1, \ldots, \theta_k), (s_1, \ldots, s_k)}$ for the composition, which is well defined assuming that the local coordinates were chosen to be disjoint.  

We now state a basic result for $d = 1$. We write $\units_n$ for the $n$-th roots of unity in $\mathbb{C}$.
\begin{proposition}\label{prop:1dthimble}
	Assume that $\mathbf{a} = (a_0, a_1, a_2)$ has signature $(1,0)$. Then there is a trivialization $\phi : \mathbb{C} \backslash \{0, -1\} \to P_1$ and $\varepsilon > 0$  for which 
	\begin{enumerate}[label=(\roman*),ref=\theproposition(\roman*)]
		\item \label{prop:1dthimble:1}  the fiber $W_{\mathbf{a}}^{-1} (*) = \sqrt[a_0]{\varepsilon} \units_{a_0} \cup \sqrt[a_1]{\varepsilon^{-1}} \units_{a_1}$,
		\item \label{prop:1dthimble:2} the vanishing thimble of $\delta_0$ equals the line segment \[\vt_{\delta_0} = \left[\sqrt[a_0]{\varepsilon}, \sqrt[a_1]{\varepsilon^{-1}}\right] , \]
		\item \label{prop:1dthimble:3}the vanishing thimble of $\delta_k$ is isotopic to the curve  \[\tau^k_{(2 \pi / a_0, 2 \pi / a_1), (0, \infty)} (\vt_{\delta_0}) . \]
	\end{enumerate}
\end{proposition}
\begin{proof} Consider the coordinate $z \in \mathbb{C} \backslash \{0, -1\}$ with the chart $z \mapsto [z:1:-z - 1]$. Then the circuit potential has the form
 \begin{align} \label{eq:explicit} W_{\mathbf{a}} (z) = (-1)^{a_2} z^{a_0}(z + 1)^{a_2} .\end{align}
 One checks that the critical point of $W_{\mathbf{a}}$ is $\frac{a_0}{a_1}$ with value $a_0^{a_0}a_1^{a_1}a_2^{a_2}$ as in Lemma~\ref{lem:circuit_basics}.
 Locally  around $0$, $W_{\mathbf{a}} (z) \sim (-1)^{a_2} z^{a_0}$. Take $u \in \mathbb{R}_{< 0}$ to be sufficiently negative and choose the basepoint $* = e^u a_0^{a_0}a_1^{a_1}a_2^{a_2}$. We can identify the fiber $W^{-1}_\mathbf{a} (*)$ with the $a_0$-th roots of the positive real number $\varepsilon = e^u a_0^{a_0} a_1^{a_1} |a_2|^{a_2}$. For small $\varepsilon$, we may perturb the coordinate $z$ near $0$ so that $\phi$ identifies $W^{-1} (*)$ with $\sqrt[a_0]{\varepsilon} \units_{a_0}$. 
 Using the coordinate $w = z^{-1}$, we obtain
 \[ W_{\mathbf{a}} (w) = (-1)^{a_2} w^{a_1} (1 + w)^{a_2} \]
 so that an identical argument yields an identification of the fiber $W^{-1}_\mathbf{a} (*)$ with the $a_1$-th roots of the positive real number $\varepsilon^{-1}$ where again $\varepsilon = e^u a_0^{a_0} a_1^{a_1} |a_2|^{a_2}$. Thus Proposition~\ref{prop:1dthimble:1} holds.
 
 \begin{figure}[h]
 \begin{picture}(0,0)%
 \includegraphics{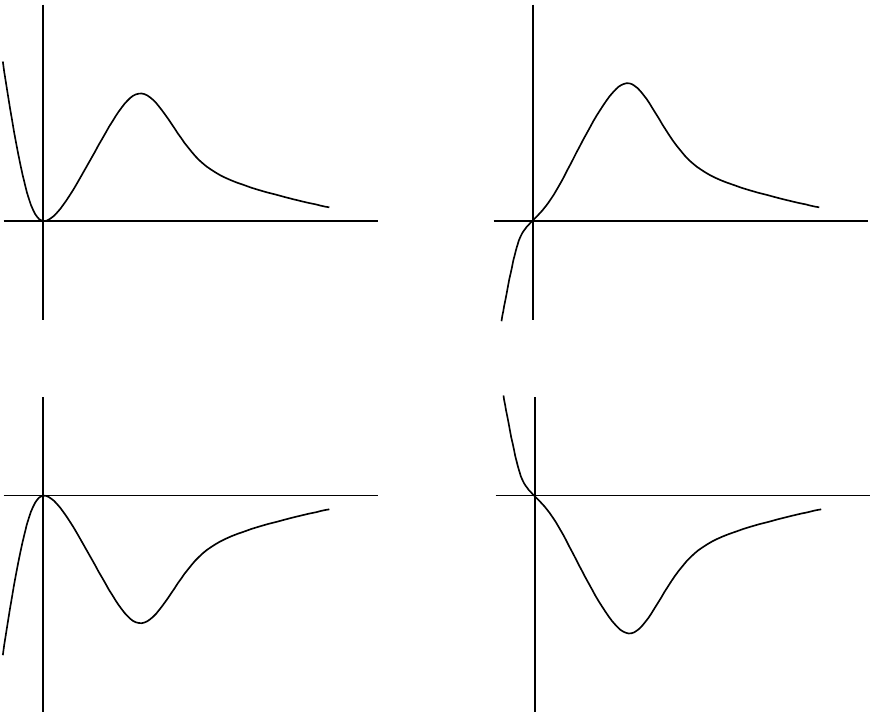}%
 \end{picture}%
 \setlength{\unitlength}{4144sp}%
\begin{picture}(3987,3276)(1605,-2875)
\put(2701,-1006){\makebox(0,0)[lb]{\smash{$(0,0)$}}}
\put(2701,-2606){\makebox(0,0)[lb]{\smash{$(0,1)$}}}
\put(4906,-2606){\makebox(0,0)[lb]{\smash{$(1,1)$}}}
\put(4906,-1006){\makebox(0,0)[lb]{\smash{$(1,0)$}}}
\end{picture}%
 \caption{\label{fig:reallocus} Graphs of $W_{\mathbf{a}}|_\mathbb{R}$ with parity $(a_0, a_1) \in \mathbb{Z}/2 \times \mathbb{Z} /2$. }
 \end{figure}
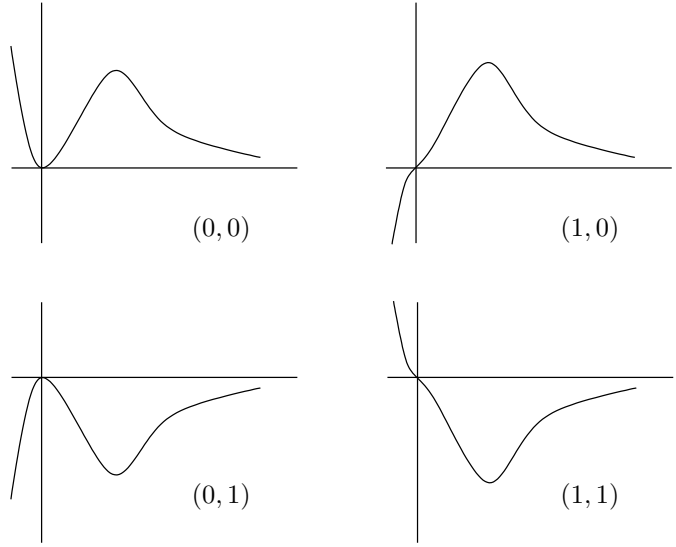
 
 Now, depending on the parity of $a_0$ and $a_2$, the graph of the restriction of $W_{\mathbf{a}}$ to the real locus is one of the four possibilities in Figure~\ref{fig:reallocus}. In each case, we have that the interval $\left[\sqrt[a_0]{\varepsilon}, \sqrt[a_1]{\varepsilon^{-1}}\right]$ maps to the connected component of $W^{-1}_{\mathbf{a}} (\textnormal{im} (\delta_0))$ which contains the critical point $a_0 / a_1$. But by Definition~\ref{eq:vt}, the vanishing thimble of a path for a Lefschetz fibration with domain a Riemann surface equals the component over the path containing the critical point. This verifies the claim~\ref{prop:1dthimble:2}.
 
For \ref{prop:1dthimble:3}, we note that $W_\mathbf{a}$ has zeros of order $a_0$ and $a_1$ at $z = 0$ and $z = \infty$ respectively. Thus, in local coordinates around $0$ and $\infty$, the fiber $W^{-1}(e^{i \theta} *)$ equals  $e^{i \theta  / a_0}\sqrt[a_0]{\varepsilon} \units_{a_0} \cup e^{i \theta  / a_1}\sqrt[a_1]{\varepsilon^{-1}} \units_{a_1} $. As $\delta_k$ is isotopic, relative its boundary, to the concatenation of $\delta_0$ with a $k$-fold counter-clockwise loop around $0$, we have that the vanishing thimble $\vt_{\delta_k}$ is that of $\delta_0$ twisted by $\tau_{(2\pi k/ a_0, 2 \pi k/ a_1), (0, \infty)}$ (here we assume that $\varepsilon$ is sufficiently small so that the fiber $W_{\mathbf{a}}^{-1} (*)$ is close to $0$ and $\infty$ in local charts).
\end{proof}

We apply Proposition~\ref{prop:1dthimble} to prove the base case of our induction argument for Theorem~\ref{thm:Amodel}. 

\begin{proposition} \label{prop:thma1d}
	Theorem~\ref{thm:Amodel} holds in dimension $d = 1$.
\end{proposition}

\begin{proof}
	We first describe the equivalence $\Xi_{\mathbf{a}, n}$ in equation~\eqref{eq:equivalence}. By Proposition~\ref{prop:duality}, it suffices to show this for $\mathbf{a}$ with signature $(1,0)$ as the case of $(0,1)$ is Koszul dual. On generating objects, we define $\Xi_{\mathbf{a}, n} (R_{\mathbf{a}} (k)) =L_{k}$ where $L_k$ is the vanishing cycle $\vc_{\delta_k}$ endowed with a brane structure relative to $\eta_\nu^2$. For $d = 1$, the vanishing cycles and the fiber $W^{-1}_{\mathbf{a}}$ are both zero dimensional. Furthermore, by Proposition~\ref{prop:1dthimble:1} and \ref{prop:1dthimble:3} we have that there is a small $\varepsilon > 0$ and a trivialization $P_1 = \mathbb{C} \backslash \{0, -1\}$ for which $L_k = \{p_0, p_1\}$ where
	\begin{align*} p_0 & =  e^{2 \pi k i / a_0} \sqrt[a_0]{\varepsilon} , & p_1 & = e^{ - 2 \pi k i / a_1} \sqrt[a_1]{\varepsilon^{-1}} . \end{align*} 
	 The grading of the vanishing cycles and their intersection points is obtained by considering the grading of the vanishing thimble $\vt_{\delta_{k}}$ relative to $\eta_\nu^2$ at $p_0$ and $p_1$.
	Given the trivialization $[z: 1 : -z - 1]$ on $P_1$, we have that 
	\begin{align*} \eta_{\nu}  = z^{\nu (0)} (-1 - z)^{\nu (2)} \frac{\tn z}{z} . 
	\end{align*} 
	Proposition~\ref{prop:1dthimble:2}, we may grade $\vt_{\delta_0}$ by taking a real valued argument for $\log (\eta_{\nu}^2 |_{T \vt_{\delta_0}})$. Precisely, letting $\theta_k (p_i)$ be the phase of $\eta_\nu^2$ at $p_i$ on $L_k$, this gives 
	\begin{align*} \theta_0 (p_0) = 0 = \theta_0 (p_1). \end{align*} 
	Proposition~\ref{prop:1dthimble:3} gives the remaining thimbles $\vt_{\delta_k}$ in terms of the monodromy twists $\tau_{(2\pi/a_0, 2 \pi a_1), (0, \infty)}$ of $\vt_{\delta_0}$. 	To accomplish transversality between two vanishing thimbles, taking small $\epsilon > 0$, we slightly perturb $\vt_{\delta_k}$ to $\tilde{\vt}_{\delta_k}$ near $p_r$ by rotating $e^{-2 \pi i \epsilon / a_r}$. As, the order of $\eta_{\nu}$ at $0$ and $\infty$ is $\nu (0) - 1$ and $\nu (1) - 1$ respectively, one observes that the induced grading of $\tilde{\vt}_{\delta_k}$ is then
	\begin{align*}
	\theta_k (p_0 ) & = 2 \pi k \left( \nu(0) - \epsilon \right)  / a_0 , & \theta_k (p_1 ) & = 2 \pi k \left( \nu (1) - \epsilon \right) / a_1.
	\end{align*}
	
	Given $j < k$ if $L_j$ intersects $L_k$ amongst the points $\sqrt[a_0]{\varepsilon} \units_{a_0}$, respectively $\sqrt[a_1]{\varepsilon} \units_{a_1}$, we denote the intersection $p_0^{\frac{k - j}{a_0}}$, respectively $p_1^{\frac{k - j}{a_1}}$. Then
		\begin{align*}
			\Hom_{\mathcal{A}_{\mathbf{a}, \nu, n} }  (L_j, L_k ) & = \begin{cases} 0 & \textnormal{if } j \not\equiv k\pmod*{a_0} \textnormal{ and } j \not\equiv k \pmod*{a_1} , \\ \mathbb{K} \cdot p_0^{\frac{k - j}{a_0}} & \textnormal{if } j \equiv k\pmod*{a_0} \textnormal{ and } j \not\equiv k \pmod*{a_1} ,
			\\ \mathbb{K} \cdot p_1^{\frac{k - j}{a_1}} & \textnormal{if } j \not\equiv k\pmod*{a_0} \textnormal{ and } j \equiv k \pmod*{a_1} ,
			\\ \mathbb{K} \cdot \left\{ p_0^{\frac{k - j}{a_0}}, p_1^{\frac{k - j}{a_1}} \right\} & \textnormal{if } j \equiv k\pmod*{a_0} \textnormal{ and } j \equiv k \pmod*{a_1} .
			\end{cases}
		\end{align*} 
The absolute index of these intersections point gives their grading and, for $r \in \{0, 1\}$ and using \cite[Section~13c]{seidel}, they are computed as
\begin{align*}
i \left( p_r^{\frac{k - j}{a_r}} \right) & = \left\lfloor \frac{1}{\pi} \left( \theta_k (p_r ) - \theta_j (p_r )  \right) \right\rfloor + 1 , \\ & = \left\lfloor \frac{1}{\pi} \left( \frac{2 \pi (k - j) (\nu (r) - \epsilon)}{a_r}  \right) \right\rfloor + 1, \\ & = 2 \nu (r) \frac{k - j}{a_r} . 
	\end{align*}
Taking the case of $\nu = 0$, degree considerations imply that the differentials and higher compositions all vanish. Changing $\nu$, however, does not affect the moduli spaces determining these compositions, so they will vanish for any choice. On the other hand, an elementary perturbation argument shows that
\[ m_2 \left(p^{\frac{k - j}{a_r}}_{r}, p^{\frac{j - i}{a_r}}_{r} \right) = p^{\frac{k - i}{a_r}}_{r}  \] 
where $m_2$ denotes the composition determined by the $A_\infty$-structure. In this case, $m_2$ agrees with the second $A_\infty$ map, but generally may differ by a sign depending on the parity of the first factor (see \cite[Section 1a]{seidel}).

One then extends the functor $\Xi_{\mathbf{a}, n}$ to morphisms by taking  
	\begin{align*} \Xi_{\mathbf{a}, n} (v_0^{s}) & = p_0^s, \\ \Xi_{\mathbf{a}, n} (v_1^s) & = p_1^s. \end{align*} It is clear that $\Xi_{\mathbf{a}, n}$ then defines the isomorphism of categories specified in \eqref{eq:equivalence} and satisfying Theorem~\ref{thm:Amodel:1} and \ref{thm:Amodel:2}. We note that, since $0 \leq i , j \leq n < \Vol (\mathbf{a})$ and  $\Xi_{\mathbf{a}, n} (v_2)$ is a morphism from $L_i$ to $L_{i + \Vol (\mathbf{a})}$ we do not encounter this morphism in the $n$-unfolded category $\mathcal{A}_{\mathbf{a}, \nu, n}$.
	
	We now show that equation~\eqref{eq:divisor} holds for $d = 1$. This can easily be seen when $\mathbf{a}$ has signature $(1,0)$ or $(0,1)$, and we give the argument for the former case. Using the trivialization of $P_1 = \mathbb{C} \backslash \{0, -1\}$, denote $z_0, z_{-1}$ and $z_\infty$ for the compactification points on $\{Z_0 + Z_1 + Z_2 = 0\} \subset \mathbb{P}^2$. The divisors $D_0$ and $D_1$ from Definition~\ref{def:atomLP} are then $a_0 z_0 + a_1 z_{\infty}$ and $-a_2 z_{-1}$ respectively. Consider the pullback in equation~\eqref{eq:pullback2} with $n = a_0$. Write $D$ for a small disc around $0$ and $D^*$ for $D \backslash \{0\}$. For a local disc chart $D  \subset P_1 \cup \{z_0\}$ around $z_0$ we have the diagram 
	\begin{equation} \label{eq:pullback3}
	\begin{tikzpicture}[baseline=(current  bounding  box.center), scale=1.5]
	\node (A) at (0,1) {$\tilde{D}$};
	\node (B) at (1,1) {$D$};
	\node (C) at (0,0) {$D$};
	\node (D) at (1,0) {$D$};
	\path[->,font=\scriptsize]
	(A) edge node[above]{$\phi$} (B)
	(A) edge node[left]{$\tilde{\pi}|_{\tilde{D}}$} (C)
	(B) edge node[right]{$z^{a_0}$} (D)
	(C) edge node[above]{$z^{a_0}$} (D);
	\end{tikzpicture} 
	\end{equation}
where the $\tilde{D} \subset \tilde{E}$ is the union $\cup_{i=1}^{a_0} D$ of discs with common intersection point $0$. Parameterizing the $l$-th disc by $w_l$ we have that $\phi (w_l) = \zeta^l w_l$ where $\zeta = e^{2 \pi l i/ a_0}\in \units_{a_0}$ and $\tilde{\pi} (w_l) = w_l$. 

Now, $p = p_0^{\frac{k - j}{a_0}} \in \Hom_{\mathcal{A}_{\mathbf{a}, \nu, n}}(L_j, L_k)$ corresponds to a point of the form $\zeta^l \sqrt[a_0]{\varepsilon}$ so that the pullback of $p$ along $\phi$ is on the boundary of the radius $\sqrt[a_0]{\varepsilon}$ circle on the $l$-th disc in $\tilde{D} = \cup_{i = 1}^{a_0} D$. Thus the moduli space $\mathcal{M}^{p, 0}_{a_0}$ contains the unique section $u_p :  \sqrt[a_0]{\varepsilon}  D \to \tilde{D}$  sending $z$ to the $w_l = z$ in the $l$-th component. It is clear that $u_p$ is the unique section of $\tilde{\pi}|_{\tilde{D}}$ with boundary condition containing the boundary of the $l$-th disc. To see that $u_p$ is unique up to isomorphism in $\mathcal{M}^{p,0}_{a_0}$, assume otherwise and compose $u :  \sqrt[a_0]{\varepsilon}D \to \tilde{E}$ with $\phi$ and $z^{a_0}$ to obtain a map $\tilde{u} := (\phi \circ u)^{a_0} : \sqrt[a_0]{\varepsilon}D \to \mathbb{C}$. Since $\tilde{u}$ sends the boundary of $\sqrt[a_0]{\varepsilon}D$ to $\varepsilon D$, the maximum principle implies that $\textnormal{im} (\tilde{u}) \subset \varepsilon D$. This in turn shows that $\textnormal{im} (u) \subset \tilde{D}$ and $\mathcal{M}^p_{a_0} = \mathcal{M}_{a_0}^{p, 0} = \{u\}$. It follows that
	\begin{align*} 
		\kappa_{a_0}^i (p) = \begin{cases} 1 & \textnormal{if } i = 0, \\ 0 & \textnormal{if } i = 1.
	\end{cases}
	\end{align*}
An identical argument applies for $a_1$ and $p = p_1^{\frac{k - j}{a_1}}$ which verifies equation~\eqref{eq:divisor} and the proposition.
\end{proof}
To conclude this section, we describe holomorphic polygons in $P_1$ with boundary conditions labeled by thimbles and their intersection points. We first establish the general notation. Let $Y$ be a symplectic surface with a regular compatible complex structure $J$ and $F \subset Y$ a finite set of points. Suppose $\mathcal{L} = \{L_0, L_1, \ldots, L_{n - 1}\}$ is an ordered set of graded Lagrangian submanifolds, possibly with boundary, such that $\partial L_i  \subset F$ and $L_i \backslash \partial L_i \subset Y \backslash F$. We equip each $L_i \in \mathcal{L}$ with a Hamiltonian $H_i : Y \to \mathbb{R}$ for which $\tn H_i |_F = 0$ and write $\tilde{H} = \{H_0, \ldots, H_n\}$ for their union. Take $\psi^t_i$ to be the time $t$ flow map of $H_i$

\begin{definition} \label{defn:transverse}
	Given a finite ordered collection of Lagrangian branes $\mathcal{L} = \{L_i\}_{0 \leq i < n}$ along with Hamiltonians $\mathcal{H} = \{H_i\}_{0 \leq i < n}$, we call the pair $(\mathcal{L}, \mathcal{H})$ generic if 
	\begin{enumerate}[label=(\roman*),ref=\thedefinition(\roman*)]
		\item \label{defn:transverse:1} Lagrangians in $\mathcal{L}$ pairwise intersect transversely,	
		\item \label{defn:transverse:2} pairwise intersections $\psi^t_{j} L_j \cap \psi_k^t L_k$ have a constant number of points for $t \geq 0$,
		\item \label{defn:transverse:3} for every subset $I \subset \{0, \ldots, n - 1\}$ and every $t > 0$, we have that $\{\psi^t_i (L_i) : i \in I\}$ are transverse submanifolds in $Y \backslash F$.
	\end{enumerate}
\end{definition}
By property~\ref{defn:transverse:2} we can label the intersection points of two Lagrangians  $\psi_j^t L_j \cap \psi_k^t L_k$ for any time $t$ by $L_j \cap L_k$. We write the graded vector space generated by this intersection as
\begin{align*}
(L_j, L_k) := \oplus_{p \in L_j \cap L_k} \mathbb{K}[i(p)].
\end{align*}

Now suppose that $D$ is a pointed disc with marked points $\zeta_0, \ldots, \zeta_m \in \partial D$ oriented counter-clockwise. Given a set map $f : \{0, \ldots, m\} \to \{0, \cdots, n\}$, we say that $(D, \{\zeta_i\} , f)$ is a \textbf{labeled pointed disc}. Denote the arc on $\partial D$ from $\zeta_i$ to $\zeta_{i +1}$ as $\partial_i D$.
\begin{definition} \label{def:directed}
	Given a labeled pointed disc $(D, \{\zeta_i\} , f)$ and a generic pair $(\mathcal{L}, \mathcal{H})$ we say that a  $J$-holomorphic map $u : D \to Y$ is directed relative to $(\mathcal{L}, \mathcal{H})$ if
	\begin{enumerate} 
		\item the Maslov index of $u$ is $2$,
		\item $u (\partial_i D) \subset \psi^1 L_{f(i)}$,
		\item $f$ is an increasing function.
	\end{enumerate}
	If $u (\zeta_i) = p_i \in  (L_{f (i)},  L_{f(i + 1)})$ we write \begin{align*} m^u (p_{m-1}, \ldots, p_0) = p_{m} . \end{align*}
\end{definition} 
Turning our attention back to the one-dimensional circuit, we consider  $(Y, F) = (P_1, W^{-1}_\mathbf{a} (*))$ and write
\begin{align*}
\mathcal{T}_{\mathbf{a}, n}	 = \{\vt_{\delta_0}, \ldots, \vt_{\delta_{n -1}} \}.
\end{align*}
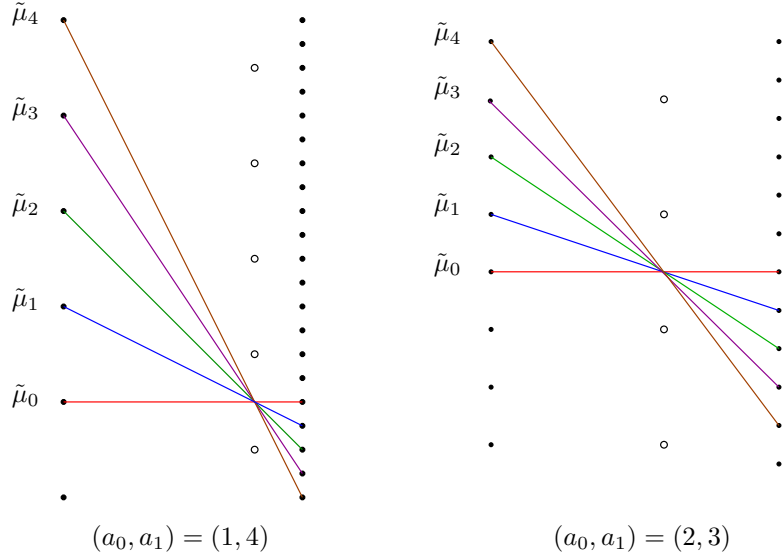
\begin{figure}[t]
\begin{picture}(0,0)%
\includegraphics{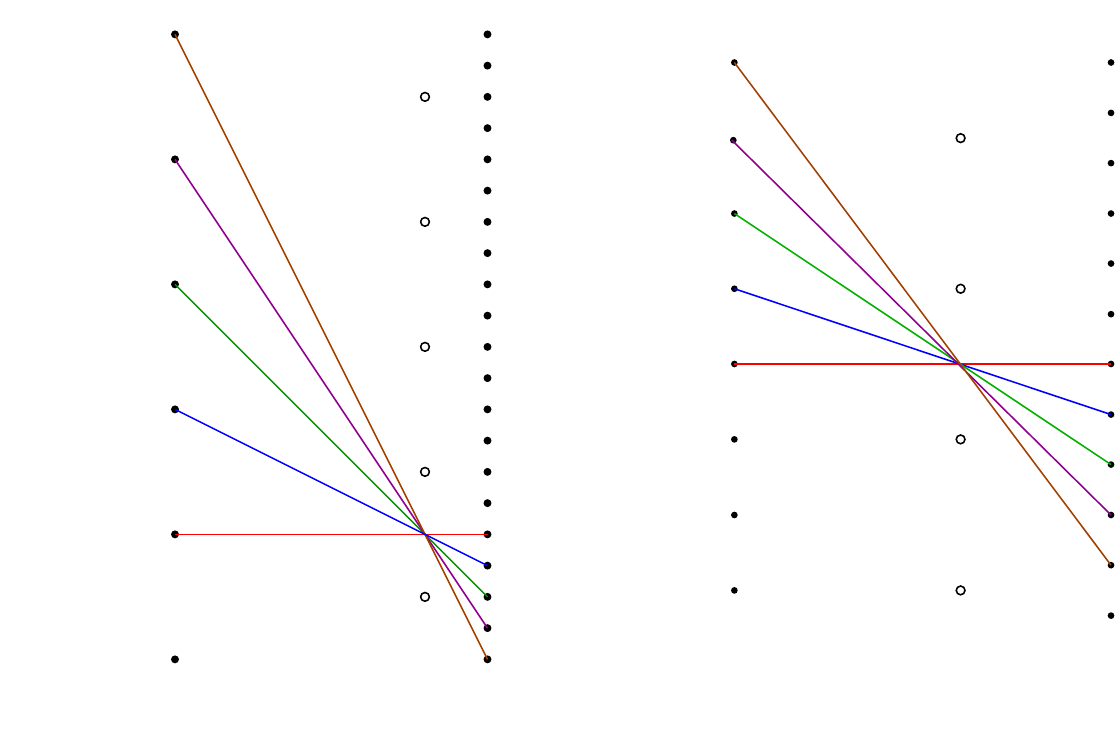}%
\end{picture}%
\setlength{\unitlength}{3947sp}%
\begin{picture}(5353,3528)(361,-3430)
\put(1376,-3361){\makebox(0,0)[lb]{\smash{$(a_0, a_1) = (1,4)$}}}
\put(4276,-3361){\makebox(0,0)[lb]{\smash{$(a_0,a_1) = (2,3)$}}}
\put(876,-1861){\makebox(0,0)[lb]{\smash{$\tilde{\mu}_1$}}}
\put(876,-1261){\makebox(0,0)[lb]{\smash{$\tilde{\mu}_2$}}}
\put(876,-661){\makebox(0,0)[lb]{\smash{$\tilde{\mu}_3$}}}
\put(876,-61){\makebox(0,0)[lb]{\smash{$\tilde{\mu}_4$}}}
\put(3526,-1636){\makebox(0,0)[lb]{\smash{$\tilde{\mu}_0$}}}
\put(3526,-1261){\makebox(0,0)[lb]{\smash{$\tilde{\mu}_1$}}}
\put(3526,-886){\makebox(0,0)[lb]{\smash{$\tilde{\mu}_2$}}}
\put(3526,-526){\makebox(0,0)[lb]{\smash{$\tilde{\mu}_3$}}}
\put(3526,-184){\makebox(0,0)[lb]{\smash{$\tilde{\mu}_4$}}}
\put(876,-2462){\makebox(0,0)[lb]{\smash{$\tilde{\mu}_0$}}}
\end{picture}%
	\caption{\label{fig:tildemu}Images of the paths $\tilde{\mu}_i$.} 
\end{figure}
Using Proposition~\ref{prop:1dthimble}, we now describe a collection of paths 
\begin{align} \label{eq:mupaths} \mathbf{\mu}_{\mathbf{a}, n} = \{\mu_0, \ldots, \mu_{n - 1 }\}
\end{align} 
in $Y$ isotopic, relative to $F$, to the thimbles in $\mathcal{T}_{\mathbf{a},n}$. Utilizing the parametrization $\mathbb{C} \backslash \{0, -1\} \cong P_1$, take $C = \ln (\varepsilon )$. For any $k \in \mathbb{Z}$ define the line segment 
\begin{align*}
\tilde{\mu}_k : [C / a_0 , -C / a_1] \to \mathbb{C}
\end{align*} 
via
\begin{align} \label{eq:tildemudef}
	\tilde{\mu}_k (t) = t \left( 1 + \frac{2\pi k }{C} i \right).
\end{align}

Define the path $\mu_k$ via $\mu_k (t) := \exp ({\tilde{\mu}_k (t)})$.
\begin{proposition} \label{prop:intnumb}
	The path $\mu_k$ is isotopic to $\vt_{\delta_k}$ in $P_1$ relative to $W^{-1}_\mathbf{a} (*)$. Furthermore, for $j < k$, we have
\begin{align} \label{eq:numberintersections}
\left| \mu_j \cap \mu_k \right| = \left\lfloor \frac{k - j}{a_0} \right\rfloor + \left\lfloor \frac{k - j}{a_1} \right\rfloor + 1.
\end{align}
\end{proposition}
\begin{proof}
Consider the space $S := \mathbb{C} \backslash \left(\pi i + 2 \pi i\mathbb{Z}\right)$ and the cover $\exp : S \to P_1$ where we have identified $P_1$ with $\mathbb{C} \backslash \{0, -1\}$. By Proposition~\ref{prop:1dthimble:1} we have that $\ln (\vt_{\delta_0}) \subset S$ is $\textnormal{im} (\tilde{\mu}_0 ) + 2\pi i \mathbb{Z}$. Furthermore, one observes that the twist map $\tau_{(2 \pi / a_0, 2 \pi / a_1)}$ on $P_1$ lifts to a map isotopic to the shear map 
\begin{align} \label{eq:shear} \tilde{\tau}_C (z) = \frac{2\pi i}{C} \Re (z) + z. \end{align} 
Thus by Proposition~\ref{prop:1dthimble:3} we have that $\vt_{\delta_k}$ is isotopic to $\exp (\tilde{\tau}_C^k (\textnormal{im} (\tilde{\mu}_0 ))) = \exp (\textnormal{im} (\tilde{\mu}_k )) = \textnormal{im} (\mu_k )$ proving the first claim.

To verify equation~\eqref{eq:numberintersections}, one checks that there is a one to one correspondence between points in $ \mu_j \cap \mu_k$ and integers $n \in \mathbb{Z}$ for which $\tilde{\mu}_j \cap (\tilde{\mu}_k + 2 \pi i n) \ne \emptyset$. As the real values of $\tilde{\mu}_k (t)$ and $\tilde{\mu}_j (s)$ are $t$ and $s$ respectively, an intersection point of the latter has the form $\tilde{\mu}_k (t) = \tilde{p} = \tilde{\mu}_j (t) + 2\pi i n$ implying $t = \frac{C n}{j - k} $. As $t \in [C / a_0 , -C / a_1]$  this implies 
\begin{align} \label{eq:bounds} -\frac{k - j}{a_0} \leq n \leq \frac{k - j}{a_1}.
\end{align}
The number of such integers $n$ satisfying this inequality is precisely $\lfloor (k - j) / a_0 \rfloor + \lfloor (k - j) / a_1 \rfloor + 1$. 
\end{proof}
We now label the intersection points between $\mu_j$ and $\mu_k$. For $j < k$, we write
\begin{align} \label{eq:tildezn}
\tilde{z}^n & := \tilde{\mu}_k \left( \frac{C n}{j - k} \right) = \tilde{\mu}_k \cap (\tilde{\mu}_j + 2 \pi i n) ,
\end{align} 
and write $z^n = \exp (\tilde{z}^n)$. Then applying the arguments from Proposition~\ref{prop:intnumb}, we see that 
\begin{align} \label{eq:intsects} \mu_j \cap \mu_k = \left\{ z^n : n \in \mathbb{Z}, - \frac{k - j}{a_0}  \leq n \leq \frac{k - j}{a_1}  \right\}.
\end{align}

With this labeling, we are able to detail the directed discs relative to the collection $\mathbf{\mu}_{\mathbf{a}, n}$. 
\begin{proposition} \label{prop:mpd}
	There are Hamiltonian perturbations $\mathcal{H}$ for $\mathbf{\mu}_{\mathbf{a}, n}$ for which 
	\begin{enumerate}[label=(\roman*),ref=\theproposition(\roman*)]
		\item \label{prop:mpd:1} $(\mathbf{\mu}_{\mathbf{a}, n}, \mathcal{H})$ is generic,
		\item \label{prop:mpd:2} any $u : D \to Y \cup \{-1\}$ directed relative to $(\mathbf{\mu}_{\mathbf{a}, n}, \mathcal{H})$ must be a holomorphic triangle, a constant bigon, or a constant polygon with image in $F$. 
		\item \label{prop:mpd:3} for $k_0 < k_1 < k_2$ and any $z^{n_1} \in (\mu_{k_0}, \mu_{k_1})$ and $z^{n_2} \in (\mu_{k_1}, \mu_{k_2})$ there exists a unique directed triangle $u : D \to Y \cup \{-1\}$ with 
		\begin{align} \label{eq:mult} m^u (z^{n_2}, z^{n_1}) = z^{n_1 + n_2} \end{align}
		and every non-constant holomorphic triangle satisfies such an equation.
	\end{enumerate}
 Furthermore, for a non-constant $u$ with boundary condition (iii) either
	\begin{enumerate}[label=(\alph*),ref=\theproposition(\alph*)]
		\item \label{prop:mpd:a} $n_1$ has the same sign as $n_2$ and the image of $u$ is contained in $Y$ or,
		\item \label{prop:mpd:b} $n_1$ has the opposite sign as $n_2$ and the image of $u$ contains $-1$.
	\end{enumerate}
\end{proposition}
\begin{figure}
		\includegraphics[]{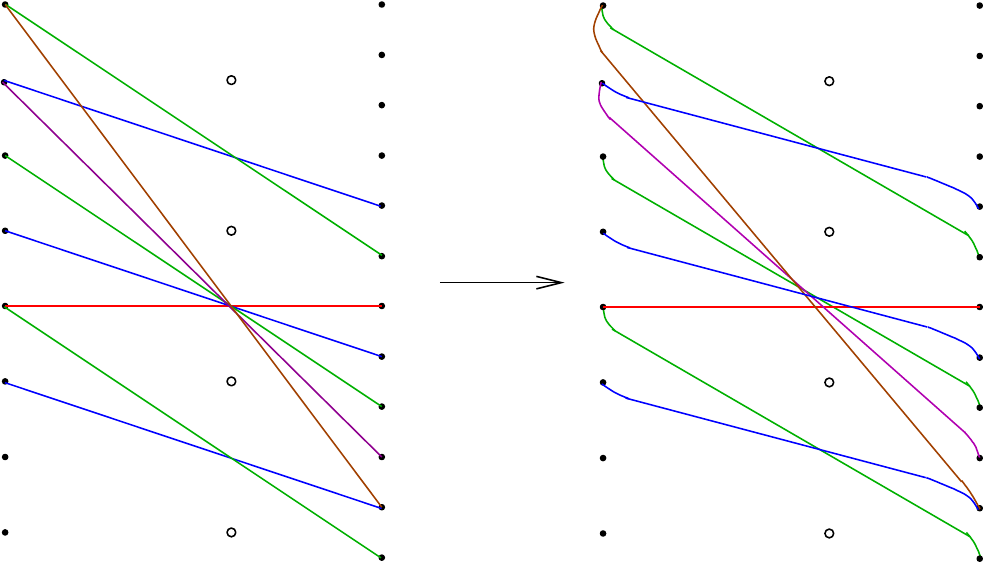}
		\caption{\label{fig:perturbations} Logarithms of the perturbations $\psi_k^t (\mu_k)$ for $(a_0, a_1) = (2, 3)$.}
\end{figure}
\begin{proof}
Define $H_k$ to be a sufficiently small Hamiltonian which rotates clockwise around $F = W^{-1}_\mathbf{a} (*)$ for small time and yields a small twist map for $\sqrt[a_0]{2 \varepsilon} < |z| < \sqrt[a_1]{(2 \varepsilon)^{-1}}$. We may choose $H_k$ so that the tangent line  of $\ln (\psi_k^{t} (\mu_k))$ at any point has slope approximately equal to $\frac{2 \pi k}{C} $, which is equal to the slope of $\tilde{\mu}_i = \ln (\mu_i)$. Such a collection of perturbed Hamiltonians is illustrated in Figure~\ref{fig:perturbations}. It is clear that, for generic choices of twists, the collection $\mathcal{H} = \{H_0, \ldots, H_{n - 1}\}$ makes $(\mathbf{\mu}_{\mathbf{a}, n}, \mathcal{H})$ generic.

Now, for any directed $u : D \to Y \cup \{-1\}$, we have that there is a lift $\tilde{u} = \ln (u): D \to \mathbb{C}$  which is unique up to a translation by $2 \pi i \mathbb{Z}$.
In order for $u$ to be a Maslov index $2$ disc, either $\tilde{u}$ is constant or it is a convex polygon with boundary edges lying on $\psi^1 (\tilde{\mu}_k)$ for $0 \leq k < n$. In the former case, by Definition~\ref{defn:transverse:3} the curves $\psi^1 (\tilde{\mu}_k)$ intersect transversely in the interior of $Y \backslash F$. In particular, no three of them intersect in a single point in $Y \backslash F$, so that any directed constant map with more than two boundary components must have image in $F$. 

For the latter case, suppose $(D , \{\zeta_i\}, f)$ is the labeled pointed disc domain of $u$ with $(m + 1)$ marked points $\{\zeta_i\}_{0 \leq i \leq m}$. Working on the cover $\mathbb{C}^*$ of $Y  \cup \{-1\}$, take $\eta^2 : T \mathbb{C}^* \otimes T \mathbb{C}^* \to \mathbb{C}$ to be the quadratic holomorphic form given by $\tn z^2$ (different choices will not affect this argument).  Each Lagrangian $\psi^1_k (L_k)$ can be graded by choosing a continuous family real arguments $\arg (\eta|_{T_z L_k} )$ for $z \in L_k$ which, by our choice of perturbations,  we may take to be arbitrarily close to $2 \tan^{-1} (2 \pi k / C)$. The absolute index (see \cite[Section~13c]{seidel}) of the intersection point $ p_k : = \tilde{u} (\zeta_k )$ for $0 < k \leq m$ is then 
\begin{align*}
i (p_k ) = \left\lfloor \frac{1}{\pi} \left( \tan^{-1} (2 \pi f(k ) / C) - \tan^{-1} (2 \pi f(k - 1) / C) \right) \right\rfloor + 1.
\end{align*}
For $k = 0$ we have $i(p_k) = \left\lfloor \frac{1}{\pi} \left( \tan^{-1} (2 \pi f(m ) / C) - \tan^{-1} (2 \pi f(0) / C) \right) \right\rfloor + 1$.
Thus, for an arbitrary holomorphic disc $\tilde{u}$ with boundary conditions in $\mathbf{\mu}_{\mathbf{a}, n}$, we have $i (p_k) \in \{0, 1 \}$ for all $p_k$. However, since  $\tilde{u}$ is directed,  $f$ is an increasing function which implies that $p_k = 0$ for all $0 \leq k \leq m$. Writing $M (\tilde{u})$ for the Maslov index of $\tilde{u}$, this then implies that
\begin{align*}
	M (\tilde{u}) = m + i(p_0) -  \sum_{k = 0}^{m - 1} i (p_k) = m 
\end{align*}
so that the only non-constant directed maps must have $m = 2$ and are holomorphic triangles.

To show \ref{prop:mpd:3}, it suffices to consider the lift $\tilde{u}$ of a holomorphic triangle $u : D \to Y \cup \{-1\}$ with domain the pointed disc $(D, \{\zeta_0, \zeta_1, \zeta_2\}, f)$ where $f(0) = k_0$, $f(1) = k_1$ and $f(2) = k_2$. Using the Hamiltonians in $\mathcal{H}$, we may assume the map $u$ has been perturbed so that the boundary is mapped to $\mu_{k_i}$ versus $\psi^1_{k_i} (\mu_{k_i})$. This may collapse certain triangles to constant maps, but otherwise it deforms a holomorphic triangle $\tilde{u}$ to a convex triangle. Assume $u (\zeta_i) = z^{n_i}$ for $i \in \{0, 1, 2\}$ and choose a lift $\tilde{u}$ so that the arc $\partial_0 D$ from $\zeta_0$ to $\zeta_1$ is mapped to $\tilde{\mu}_{k_0}$. Then the image of  $\tilde{u}$ is bounded by $\tilde{\mu}_{k_0}$, $(\tilde{\mu}_{k_1} + 2 \pi i n_1)$ and $(\tilde{\mu}_{k_2} + 2 \pi i n_0)$ where 
\begin{align*} - \frac{k_1 - k_0}{a_0}  \leq & n_1 \leq  \frac{k_1 - k_0}{a_1},  & -  \frac{k_2 - k_0}{a_0}  \leq n_0 \leq  \frac{k_2 - k_0}{a_1} .\end{align*} 
Then $(\tilde{\mu}_{k_1} + 2 \pi i n_1) \cap (\tilde{\mu}_{k_2} + 2 \pi i n_0)$ is equivalent modulo $2\pi i \mathbb{Z}$ to $\tilde{\mu}_{k_1} \cap (\tilde{\mu}_{k_2} + 2 \pi i (n_0 - n_1)) = \tilde{\mu}_{k_1} \cap (\tilde{\mu}_{k_2} + 2 \pi i n_2)$ so that $n_0 = n_1 + n_2$ which verifies the claim. Conversely, one easily shows that $n_1$ and $n_2$ satisfying the respective bounds in inequality~\eqref{eq:bounds}, $n_0 = n_1 + n_2$ satisfies this inequality as well and there is a holomorphic triangle $u$ for which equation~\eqref{eq:mult} holds.

For claims~\ref{prop:mpd:a} and \ref{prop:mpd:b} we simply observe that if $n_1$ and $n_2$ have the same sign, then so does $n_0$ and the holomorphic triangle $\tilde{u}$ lies in the positive or negative real half plane and does not contain $\pi i + 2 \pi i \mathbb{Z}$. On the other hand, if $n_1$ and $n_2$ have opposite signs, then the triangle $\tilde{u}$ intersects $i \mathbb{R}$ in an interval. But since, for every $k, n \in \mathbb{Z}$,  $(\tilde{\mu}_k + 2\pi i n)$ intersects $i \mathbb{R}$ in $2 \pi i \mathbb{Z}$ we have that this interval must intersect $\pi i + 2 \pi i \mathbb{Z}$ non-trivially so that the image of $u$ contains $-1$.
\end{proof}
Proposition~\ref{prop:mpd} is closely related to the calculation of the wrapped Fukaya category for the cylinder and the pair of pants \cite{aaeko, fss}. However, there are two minor differences which prevent one from referencing such computations directly. First, the paths $\mu_j$ are compact paths in a non-compact surface, versus the case of a Lagrangian with Legendrian boundary (or a non-compact Lagrangian with a cylindrical end \cite{abouzaid12}). Secondly, the lengths of these compact paths are asymmetric about circle of radius $- C^{-1}$, making the book-keeping of intersections slightly more delicate. 

\subsection{Circuit Lefschetz bifibrations} \label{sec:dcase}

In this section we will prepare the groundwork to perform the induction step for the proof of Theorem~\ref{thm:Amodel}. An essential ingredient in this step is the use of matching paths and Lefschetz bifibrations which factor the circuit LG model $W_\mathbf{a}$. We first define a slight variant of Lefschetz bifibrations adapted to the setting of pencils on K\"ahler manifolds. 

Let $X$ be a complete K\"ahler manifold with four effective, smooth normal crossing divisors $\mathbf{D} = \{D_{v, 0}, D_{v,\infty}, D_{h, 0}, D_{h, \infty}\}$ such that $D_{v,0}$ and $D_{h,0}$ are linearly equivalent to $D_{v,\infty}$ and $D_{h,\infty}$ respectively. We call the pencils associated to $D_{v,i}$ and $D_{h, i}$ the vertical and horizontal linear systems respectively. Write $B_v = D_{v, 0} \cap D_{v, \infty}$ and $B_h = D_{h, 0} \cap D_{h, \infty}$ for the base loci of these linear systems and $\psi_{\mathbf{D}, v} : X \backslash B_v \to \mathbb{P}^1$, $\psi_{\mathbf{D}, h} : X \backslash B_h \to \mathbb{P}^1$ for the induced maps.  For $t \in  \mathbb{P}^1$ we say the divisor $D_{h, t} = \psi_{\mathbf{D}, h}^{-1} (t)$ is an \textbf{irregular horizontal divisor} if it does not transversely intersect $D_{v, 0} \cup D_{v, \infty}$. Finally, we write $\Psi_{\mathbf{D}}$ for the product $ (\psi_{\mathbf{D}, v}, \psi_{\mathbf{D}, h}) : X \backslash (B_v \cup B_h) \to \mathbb{P}^1 \times \mathbb{P}^1$. 

\begin{definition} \label{defn:lbp} The collection $\mathbf{D}$ is said to generate a Lefschetz bipencil if
	\begin{enumerate}[label=(\roman*),ref=\thedefinition(\roman*)]
		\item \label{defn:lbp:1} $D_{h, 0}$ and $D_{h,\infty}$ are components of $D_{v, 0}$,
		\item \label{defn:lbp:2} $\psi_{\mathbf{D}, v} : X \backslash \left( D_{v, 0} \cup D_{v, \infty} \right) \to \mathbb{C}^*$ is a Lefschetz pencil,
		\item \label{defn:lbp:3} there are finitely many irregular horizontal divisors $D_{h, t_1}, \ldots, D_{h, t_k}$,
		\item \label{defn:lbp:4} the closure of the critical locus $\Delta_{\mathbf{D}} := \overline{\cp{\Psi_{\mathbf{D}}}}$ does not intersect $B_h$ and satisfies 
		\begin{align*} 
		\Delta_{\mathbf{D}} \cap D_{h, t_i} \subset D_{v, 0} \cup D_{v, \infty}
		\end{align*}
		for all $1 \leq i \leq k$.
	\end{enumerate} 
\end{definition}

If $\mathbf{D}$ generates a Lefschetz bipencil, then it can be used to obtain an exact symplectic Lefschetz fibration as defined in \cite[Section~15c]{seidel}. Consider the spaces
\begin{align*} E_{\mathbf{D}} & = X \backslash  \left( D_{v, 0} \cup D_{v, \infty} \cup \left(\cup_{i = 1}^k D_{h, t_i} \right) \right), \\ Y & = \mathbb{C}^* \times \left( \mathbb{P}^1 \backslash \{t_1, \ldots, t_k\} \right).
\end{align*} 
The exact symplectic structure on $E_{\mathbf{D}}$ can be defined by taking the K\"ahler potential $\left(\log (\prod s_{D_{h, t_i}})\right)^2$. Writing $\pi_1$ for the projection to the first coordinate, we then obtain the sequence of maps
\begin{equation}\label{eq:slfib}
\begin{tikzpicture}[baseline=(current  bounding  box.center), scale=1.5]
\node (A) at (0,0) {$E_{\mathbf{D}}$};
\node (B) at (1,0) {$Y$};
\node (C) at (2,0) {$\mathbb{C}^*$};
\path[->,font=\scriptsize]
(A) edge node[below]{$\Psi_{\mathbf{D}}$} (B)
(B) edge node[below]{$\pi_1$} (C)
(A) edge[bend left] node[above]{$\psi_{\mathbf{D}, v}$} (C);
\end{tikzpicture} 
\end{equation}
Because $E_{\mathbf{D}} \subsetneq E = X \backslash (D_{v, 0} \cup D_{v, \infty}) \subset X$, the Fukaya-Seidel category of $\psi_{\mathbf{D}, v} : E_{\mathbf{D}} \to \mathbb{C}^*$ differs from the Lefschetz pencil $\psi_{\mathbf{D}, v} : E  \to \mathbb{C}^*$ and we will need to reincorporate the excised divisors $D_{h, t_i}$ into the picture in order to relate these  categories. We will do this in the next section, but first we describe our main example of a Lefschetz bipencil.

Let $d > 1$ and $\mathbf{a} = (a_0, \ldots, a_{d + 1}) \in \mathbb{Z}^{d + 2}$ satisfy equations~\eqref{eq:balance}. Assume $\mathbf{a}$ has signature $(p,q)$ with $p \geq 1$ and take $X = \bar{P}_d \subset \mathbb{P}^{d + 1}$. We consider the collection of effective divisors $\mathbf{D} (\mathbf{a}) = \{D_{v, 0}, D_{v, \infty}, D_{h, 0}, D_{h, \infty}\}$ where the vertical divisors $D_{v, i} = D_i$ are given from equation~\eqref{eq:sections} and 
\begin{align*}
D_{h,0} & = C_0 |_{\bar{P}_d} , & D_{h, \infty} &= C_1 |_{\bar{P}_d}. 
\end{align*}
With this data, we have that $E = P_d$ and $\psi_{\mathbf{D} (\mathbf{a}), v} = W_{\mathbf{a}}$.
\begin{proposition} \label{prop:circuitbiLefschetz}
 The collection $\mathbf{D} (\mathbf{a} )$ on $\bar{P}_d$ generates a Lefschetz bipencil with irregular horizontal divisors $\{D_{h, 0}, D_{h, -1}, D_{h, \infty}\}$. Furthermore, the critical points and values of $\Psi_{\mathbf{D} (\mathbf{a} )}$ are 
 \begin{align}
 \label{eq:cpPsi}	\cp{\Psi_{\mathbf{D} (\mathbf{a} )}} & = \{ [a_0 + z: a_1 - z : a_2 : \cdots :a_{d + 1} ] : z \in \mathbb{C} \backslash \{-a_0, a_1\} \}, \\
 \label{eq:cvPsi}	\cv{\Psi_{\mathbf{D} (\mathbf{a} )}} & = \left\{ \left( \prod_{j = 2}^{d + 1} a_j^{a_j} (a_0 + z)^{a_0}(a_1 - z)^{a_1},  \frac{a_0 + z}{a_1 - z}  \right) :  z \in \mathbb{C} \backslash \{-a_0, a_1\} \right\}.
 \end{align}
\end{proposition}
\begin{proof}
	One verifies Definition~\ref{defn:lbp:1} immediately from the definition of the divisors in $\mathbf{D} (\mathbf{a})$ while \ref{defn:lbp:2} was observed in Lemma~\ref{lem:circuit_basics}. For property~\ref{defn:lbp:3} we take coordinates $[Z_0: \cdots : Z_{d + 1}] \in \bar{P}_d  \subset \mathbb{P}^{d + 1}$ and observe that for $t \ne 0, -1, \infty$, we have that $D_{h, t} = \{[tX : X: Z_2: \cdots :Z_{d + 1} ] : (t + 1) X+ Z_2 + \cdots + Z_{d + 1} = 0 \} \cong \mathbb{P}^{d}$. A check shows that the intersection $D_{h, t}$ with the support of $D_{v, 0} \cup D_{v, \infty}$ then gives transverse coordinate hyperplanes. On the other hand at $t = 0, \infty$, $D_{h, t} \subset D_{v, 0}$ while at $t = -1$, $D_{h, -1}$ has a non-transverse intersection with $D_{v,\infty}$ at the point $[1:-1: 0: \cdots :0]$. 

	For the last property~\ref{defn:lbp:4}, we first compute the critical points of $\Psi_{\mathbf{D} (\mathbf{a})}$ in $E = P_d$. Letting $[Z_0: \cdots : Z_{d + 1}] \in E  \subset \mathbb{P}^{d + 1}$, we have
		\begin{align} \label{eq:circbil} \Psi_{\mathbf{D} (\mathbf{a} )} ([Z_0: \cdots :Z_{d + 1}]) & = \left(W_{\mathbf{a}} ([Z_0: \cdots :Z_{d + 1}]) , \frac{Z_0}{Z_1} \right). 
		\end{align}
 Let $g = \frac{Z_0}{Z_1}$ and observe  
	\[\tn g = \frac{Z_0}{Z_1} \left( \frac{\tn Z_0}{Z_0} - \frac{\tn Z_1}{Z_1}  \right) \] which is non-zero for all $[Z_0: \cdots : Z_{d + 1}] \in E$.  Next, one checks that $[Z_0: \cdots : Z_{d + 1}] \in \cp{\Psi_{\mathbf{D} (\mathbf{a})}}$ if and only if 
	
	\begin{align*} 0 & =  \tn g \wedge \tn W_{\mathbf{a}} \wedge \tn \left( \sum_{i = 0}^{d + 1} Z_i \right), \\ & = 
	\tn g \wedge \left( \sum_{1 < i < j} \left(\frac{a_i}{Z_i} - \frac{a_j}{Z_j} \right) \tn Z_i \wedge \tn Z_j \right) + , \\
	&  \hspace{.5in} + \sum_{1 < j} \left[	 \frac{1}{Z_0} \left( \frac{a_1}{Z_1} - \frac{a_j}{Z_j}  \right) + \frac{1}{Z_1} \left( \frac{a_0}{Z_0} - \frac{a_j}{Z_j} \right) \right] \tn Z_0 \wedge \tn Z_1 \wedge \tn Z_j.
	\end{align*} 
	Solving these equations shows that there exists $ t \ne 0$, $z \in \mathbb{C} \backslash \{-a_0, a_1\}$ for which $Z_i = a_i t$ for $1 < i \leq d + 1$, $Z_0 = t (a_0 + z)$ and $Z_1 =  t (a_1 - z)$. This verifies equations~\eqref{eq:cpPsi} and \eqref{eq:cvPsi}. The only irregular divisor not contained in $D_{v, 0}$ is $D_{h, -1}$ and we see that $[a_0 + z: a_1 - z:a_2: \cdots :a_{d + 1}] \in D_{h, -1} \cap \Delta_{\mathbf{D} (\mathbf{a})}$ implies either $a_0 + z = - a_1 + z$ and $a_0 = -a_1$, or $z = \infty$. Since $a_0 , a_1 > 0$, the first case does not occur, implying $D_{h, -1} \cap \cv{\Psi_{\mathbf{D} (\mathbf{a})}} = \emptyset$ or $D_{h, -1} \cap \Delta_{\mathbf{D} (\mathbf{a})} \subset D_{v, 0} \cup D_{v, \infty}$.
\end{proof}

By Proposition~\ref{prop:circuitbiLefschetz} we have that ${\mathbf{D} (\mathbf{a})}$ generates a Lefschetz bipencil and that the induced diagram 
\begin{equation*}
\begin{tikzpicture}[baseline=(current  bounding  box.center), scale=2.3]
\node (A) at (0,0) {$E_{\mathbf{D}(\mathbf{a})}$};
\node (B) at (1,0) {$\mathbb{C}^* \times P_1$};
\node (C) at (2,0) {$\mathbb{C}^*$};
\path[->,font=\scriptsize]
(A) edge node[below]{$\Psi_{\mathbf{D}(\mathbf{a})}$} (B)
(B) edge node[below]{$\pi_1$} (C)
(A) edge[bend left=25] node[above]{$W^\circ_{\mathbf{a}}$} (C);
\end{tikzpicture}
\end{equation*} 
is a symplectic Lefschetz bifibration where $E_{\mathbf{D}(\mathbf{a})} \cong P_d \backslash D_{h, -1}$ and we define $W^\circ_{\mathbf{a}}$ as the restriction of $W_{\mathbf{a}}$ to $E_{\mathbf{D} (\mathbf{a})}$.  In this diagram, we have implicitly identified $\mathbb{P}^1 \backslash \{[0:1],[1:0],[-1:1]\}$ with $P_1$. As we will utilize this identification shortly, we make it explicit by taking
\begin{align*}
	\Psi_{\mathbf{D} (\mathbf{a})} ([Z_0: \cdots : Z_{d + 1}]) = \left( W_{\mathbf{a}} ([Z_0: \cdots : Z_{d + 1}]), [  Z_0: Z_1: -Z_0 - Z_1] \right).
\end{align*} 
Fixing more notation, we write $E_{\mathbf{a}}$ for $P_d$ and, noting that $D_{v, 0} = \sum_{i = 0}^p a_i C_i|_{\bar{P}_d}$, for any $0 \leq i \leq p$ we write $\bar{E}_{\mathbf{a}, i}$ for the partial compactification $E_\mathbf{a} \cup C_i$. We then extend $W_{\mathbf{a}}$ to a map $\bar{W}_{\mathbf{a}} : \bar{E}_{\mathbf{a}, i} \to \mathbb{C}$ by sending $C_i$ to zero. These partial compactifications then fit into the following commutative diagram.
\begin{equation}
\label{eq:circuitbifib2}
\begin{tikzpicture}[baseline=(current  bounding  box.center), cross line/.style={preaction={draw=white, -, line width=6pt}}, scale=2.3]
\node (A) at (0,2) {$E_{\mathbf{D}(\mathbf{a})}$};
\node (B) at (1,2) {$\mathbb{C}^* \times P_1$};
\node (C) at (2,2) {$\mathbb{C}^*$};
\node (A2) at (0,1) {$E_{\mathbf{a}}$};
\node (B2) at (1,1) {$\mathbb{C}^* \times \mathbb{C}^*$};
\node (C2) at (2,1) {$\mathbb{C}^*$};
\node (A3) at (0,0) {$\bar{E}_{\mathbf{a}, 0}$};
\node (B3) at (1,0) {$\mathbb{C} \times \mathbb{C}$};
\node (C3) at (2,0) {$\mathbb{C}$};
\path[->,font=\scriptsize]
(A) edge node[below]{$\Psi_{\mathbf{D}(\mathbf{a})}$} (B)
(B) edge node[below]{$\pi_1$} (C)
(A) edge[bend left=25] node[above]{${W}^\circ_{\mathbf{a}}$} (C)
(A2) edge node[below]{$\Psi_{\mathbf{D}(\mathbf{a})}$} (B2)
(B2) edge node[below]{$\pi_1$} (C2)
(A2) edge[bend left=25] node[left=28pt]{$W_{\mathbf{a}}$} (C2)
(A3) edge node[below]{$\Psi_{\mathbf{D}(\mathbf{a})}$} (B3)
(B3) edge node[below]{$\pi_1$} (C3)
(A3) edge[bend left=25] node[left=28pt]{$\bar{W}_{\mathbf{a}}$} (C3);
\path[right hook->]
(A) edge node{} (A2)
(A2) edge node{} (A3)
(C) edge node{} (C2)
(C2) edge node{} (C3)
(B) edge [cross line] node{} (B2)
(B2) edge [cross line] node{} (B3);
\end{tikzpicture}
\end{equation} 
We will denote the fibers of $W^\circ_{\mathbf{a}}$, $W_{\mathbf{a}}$ and $\bar{W}_{\mathbf{a}}$ over $t$ as $F^\circ_t$, $F_t$ and $\bar{F}_t$ respectively.

Our next task in the induction argument is to understand the Lefschetz bifibration $\Psi_{\mathbf{D} (\mathbf{a})}$ purely in terms of lower dimensional circuit potentials.  Given balanced $\mathbf{a}, \nu_{\mathbf{a}} \in \mathbb{Z}^{d + 2}$,  define lower rank lattice elements 
\begin{align} \label{eq:abc} 
\begin{split}
\mathbf{b} : = (b_0, \ldots, b_d) & = (a_0 + a_1, a_2, \ldots, a_{d + 1}) \in \mathbb{Z}^{d + 1}, \\  \left( \nu_{\mathbf{b}}(0), \ldots, \nu_{\mathbf{b}} (d) \right) & = \left( \nu_{\mathbf{a}} (0) + \nu_{\mathbf{a}} (1), \nu_{\mathbf{a}} (2), \ldots, \nu_{\mathbf{a}} (d + 1) \right), \\ 
\mathbf{c} := (c_0, c_1, c_2) &  = (a_0, a_1, -a_0 - a_1 ) \in \mathbb{Z}^3, \\ \left( \nu_{\mathbf{c}} (0), \nu_{\mathbf{c}} (1), \nu_{\mathbf{c}} (2) \right) & =  \left( \nu_{\mathbf{a}} (0), \nu_{\mathbf{a}} (1), - \nu_{\mathbf{a}} (0) -  \nu_{\mathbf{a}} (1) \right) .
\end{split}
\end{align} 
The elements $\mathbf{b}$ and $\mathbf{c}$ clearly satisfy the  equations~\eqref{eq:balance} and are of signature $(p - 1, q)$ and $(1, 0)$ respectively. If $\nu_{\mathbf{a}}$ is balanced then one easily checks that $\nu_{\mathbf{b}}$ and $\nu_{\mathbf{c}}$ are balanced. Also take $\rho : E_{\mathbf{D} (\mathbf{a})} \to P_{d - 1}$ to be the function  
\begin{align*} \rho ([Z_0: \cdots :Z_{d + 1}]) = [Z_0 + Z_1: Z_2: \cdots Z_{d + 1}].
\end{align*} 
Writing $\Psi_{\mathbf{D} (\mathbf{a})}^t $ for the restriction of $\pi_2 \circ \Psi_{\mathbf{D} (\mathbf{a})}$ to $F_t$, we state the following proposition.

\begin{proposition} \label{prop:fibprod} There is an isomorphism 
	\begin{align*} h : E_{\mathbf{D} (\mathbf{a})} \to P_1 \times P_{d - 1}\end{align*} for which $h^* ( \eta_{\nu_{\mathbf{c}}} \wedge \eta_{\nu_{\mathbf{b}}}) = (-1)^{\nu_{\mathbf{a}}(2) - 1} \eta_{\nu_{\mathbf{a}}}$. Furthermore, for any $t \in \mathbb{C}^*$ fibers $F^\circ_t$ and $F_t$ fit into the pullback diagrams
\begin{equation}
\label{eq:circuitbifib}
\begin{tikzpicture}[baseline=(current  bounding  box.center), scale=2.3]
\node (A) at (0,1) {$F^\circ_t$};
\node (B) at (1,1) {$P_{d - 1}$};
\node (C) at (0,0) {$P_1$};
\node (D) at (1,0) {$\mathbb{C}^*,$};
\node (A2) at (2.5,1) {$F_t$};
\node (B2) at (3.5,1) {$P_{d - 1} \cup C_0$};
\node (C2) at (2.5,0) {$\mathbb{C}^*$};
\node (D2) at (3.5,0) {$\mathbb{C}.$};
\path[->,font=\scriptsize]
(A2) edge node[above]{$\rho$} (B2)
(B2) edge node[right]{$\bar{W}_{\mathbf{b}}$} (D2)
(A2) edge node[left]{$\Psi^t_{\mathbf{D} (\mathbf{a})}$} (C2)
(C2) edge node[below]{$(-1)^{c_2} t / W_{\mathbf{c}}$} (D2)
(A) edge node[above]{$\rho$} (B)
(B) edge node[right]{$W_{\mathbf{b}}$} (D)
(A) edge node[left]{$\Psi^t_{\mathbf{D} (\mathbf{a})}$} (C)
(C) edge node[below]{$(-1)^{c_2} t / W_{\mathbf{c}}$} (D);
\end{tikzpicture}
\end{equation}
\end{proposition}
\begin{proof}
Define the maps $g: P_1 \times P_{d - 1} \to E_{\mathbf{D} (\mathbf{a} )}$ and $h : E_{\mathbf{D} (\mathbf{a} )} \to P_1 \times P_{d - 1}$ via
\begin{align*}
g([U_0: U_1 : U_2], [V_0: \cdots : V_{d}])  & = [ - V_0 U_0: -V_0 U_1: U_2 V_1: \cdots :U_2 V_{d}], \\
h (Z) & = ( \pi_2 ( \Psi_{\mathbf{D} (\mathbf{a})}(Z)), \rho (Z) ) ,
\end{align*}
where $Z = [Z_0: \cdots :Z_{d + 1}] \in E_{\mathbf{D} (\mathbf{a})} \subset P_d$. One easily computes that $g$ and $h$ yield a pair of inverse isomorphisms between $P_1 \times P_{d - 1}$ and $E_{\mathbf{D}(\mathbf{a})}$. To see that $h^* (\eta_{\nu_{\mathbf{c}}} \wedge \eta_{\nu_{\mathbf{b}}}) = \eta_{\nu_{\mathbf{a}}}$ one utilizes the coordinate representation of $\eta_\nu$ given in equation~\eqref{eq:etalc}.

 To see that $h$ induces the fiber squares in \eqref{eq:circuitbifib}, we have that $W_{\mathbf{a}} (g(U, V)) = t$ if and only if 
\begin{align*} t & =  (-1)^{a_0 + a_1}U_0^{a_0} U_1^{a_1} U_2^{\sum_{i = 2}^{d + 1} a_i} V_0^{a_0 + a_1} \prod_{i = 1}^{d} V_i^{a_{i + 1}},  \\
& = (-1)^{c_2}  W_{\mathbf{b}} (V)  W_{\mathbf{c}}(U).
\end{align*}
Thus, $g$ and $h$ restrict to inverse isomorphisms between $ P_1 \times_{\mathbb{C}^*} P_{d - 1}$ and  $F^\circ_t$ yielding the left hand diagram. The extension to $F_t$ is immediate from the fact that $D_{h, 0}$ maps to $C_0$ via $\rho$.
\end{proof}

A basic corollary of Proposition~\ref{prop:fibprod} gives a description of the critical values of $\Phi^t_{\mathbf{D} (\mathbf{a})}$ in terms of the one-dimensional circuit potential studied in Section~\ref{sec:1dcase}.

\begin{corollary} \label{cor:critvt}
For any $t \in \mathbb{C}^*$, the set of critical values $\cv{\Psi_{\mathbf{D} (\mathbf{a})}^t}$ equals the fiber $W_{\mathbf{c}}^{-1} ( (-1)^{c_2}  t / q_{\mathbf{b}})$.
\end{corollary}
\begin{proof}
From Proposition~\ref{prop:fibprod}, the fiber of $\Psi_{\mathbf{D} (\mathbf{a})}^t$ over $q$ is isomorphic to the fiber of $W_{\mathbf{b}}$ over $(-1)^{c_2}t W_{\mathbf{c}} (q)$. By Lemma~\ref{lem:circuit_basics}, $q_{\mathbf{b}}$ is the unique critical value of $W_{\mathbf{b}}$ so $(\Psi_{\mathbf{D} (\mathbf{a})}^t)^{-1} (q)$ is singular if and only if $(-1)^{c_2}t / W_{\mathbf{c}} (q) = q_{\mathbf{b}}$.
\end{proof}

Using Proposition~\ref{prop:fibprod} and Corollary~\ref{cor:critvt}, we will accomplish the basic task of identifying the vanishing cycles of $W^\circ_{\mathbf{a}}$ associated to the paths $\delta_k$ defined in equation~\eqref{eq:deltadef}. To do this, we will recall and apply the general machinery of matching paths and matching cycles (see \cite[Section~16]{seidel}). As our Lefschetz bifibrations satisfy particularly strong properties, we will only need a basic version of this machinery which we now discuss. First, let $D_+$ be the closed upper half-disc $\{z \in \mathbb{C}: |z| \leq 1, \Im (z) \geq 0\}$, $\partial_+ D_+ = \{z \in D_+ : |z| = 1\}$ its upper boundary and $D^\circ_+$ its interior.

\begin{definition} \label{def:matchingpath}
	Let $\pi : Y \to S$ be a Lefschetz fibration from a K\"ahler surface $Y$ which restricts to a Lefschetz fibration on a complex curve $\Delta \subset Y$. Given a $\pi|_\Delta$-admissible curve $\delta$ from $q_0$ to $q_1$, an embedded path $\mu_\delta : [-1,1] \to \pi^{-1} (q_0)$ will be called a matching path of $\delta$ if there exists an embedding  $M_\delta : D_+ \to S$ such that 
	\begin{enumerate}
		\item $M_\delta |_{[-1, 1]} = \mu_\delta$, 
		\item $M_\delta |_{D^\circ_+}$ is an embedding into $Y \backslash \Delta$,
		\item $\pi \left( M_\delta (a + ti) \right) = \delta (t)$ for all $a + ti \in D_+$,
		\item $M_\delta (\partial_+ D_+) = \vt_\delta$.
	\end{enumerate}
\end{definition}
Matching paths can be visualized by imagining the movie of points $A_t := \Delta \cap \pi^{-1} (\delta (1 - t))$ in the fiber $F_t := \pi^{-1} (\delta (1 - t))$ with $t$ starting at $0$ and ending at $1$. As $\pi|_\Delta$ is a Lefschetz fibration, at $t = 0$ one of the points in $A_t$ has multiplicity $2$ and for small $t$, it separates into two points in $F_t$. Drawing a path between these points and continuing it via a smooth isotopy until $t = 1$ yields a matching path to $\delta$. Here it is important that the interior of the path never intersects a point of $A_t$. We note that matching paths are defined up to relative isotopy in $\pi^{-1} (q_0) \backslash \Delta$.

In fact, the notion of a \textbf{matching path} for an ordinary Lefschetz fibration $\pi : E \to S$ is a priori distinct from the above definition. In this more basic scenario, suppose $\mu : [-1, 1] \to S$ is a path with $\mu (\pm 1) \in \cv{\pi}$, $\mu (-1, 1) \cap \cv{\pi} = \emptyset$ and let $\mu_\pm : [0, 1] \to S$ be the restrictions $\mu_\pm (t) = \mu (\pm t)$. Then we say that $\mu$ is a matching path if the vanishing cycles $\vc_{\mu_+}$ and $\vc_{\mu_-}$ are Hamiltonian isotopic Lagrangian submanifolds of the fiber $\pi^{-1} (\mu (0))$. The main  advantage of having a matching path is that, assuming one performs a fiberwise Hamiltonian isotopy $I$ of vanishing cycles along $\mu$, one may glue the vanishing thimbles $\vt_{\mu_\pm}$ together to obtain a  Lagrangian sphere $\Sigma_{\mu, I} \subset E$. The sphere $\Sigma_{\mu, I}$ is called the \textbf{matching cycle} of $\mu$. We now recall the following result relating matching paths, matching cycles and matching paths of admissible paths.

\begin{proposition}{\cite[Lemma~16.15]{seidel}} \label{prop:isotopy} Let $\Psi : E \to Y$, $\psi : Y \to S$ and $\pi = \psi \circ \Psi$ form a Lefschetz bifibration and $\Delta = \cv{\Psi}$. Assume $\delta$ is a $\psi|_\Delta$-admissible path from $q_0$ to $q$. Then: 
	\begin{enumerate} \item $\delta$ is a $\pi$-admissible path, 
		\item if $\mu_\delta :[-1,1] \to \psi^{-1} (q_0)$ is a matching path of $\delta$, then it is a matching path for $\Psi^{q_0} : \pi^{-1} (q_0) \to \psi^{-1} (q_0)$,  
		\item for an appropriate isotopy $I$, the matching cycle $\Sigma_{\mu_{\delta}, I}$ is isotopic, as a framed Lagrangian sphere in $\pi^{-1} (q_0)$, to the vanishing cycle $\vc_{\delta}$.
		\end{enumerate}
\end{proposition}

In order to utilize Proposition~\ref{prop:fibprod} to produce matching paths, it will be helpful to notationally distinguish $W_{\mathbf{a}}$, $W_{\mathbf{b}}$ and $W_{\mathbf{c}}$-distinguished paths. Given any $u \in \mathbb{C}$ with $\Re (u) \ne 0$ and $\lfloor \Im (u) / 2\pi \rfloor = k$, take $*_{\mathbf{a}} = e^{u } q_{\mathbf{a}}$ and let $\alpha_k (s) = e^{u (1 - s)}q_{\mathbf{a}}$ be the $W_{\mathbf{a}}$-admissible path $\delta_k$ defined in equation~\eqref{eq:deltadef}. Let $r_{\mathbf{a}} = \ln ((-1)^{c_2} q_{\mathbf{a}}^{-1})  \in \mathbb{C}$ satisfy $0 \leq \Im ( r_{\mathbf{a}} ) < 2\pi $ and choose $v$ with $\Re (v) < 0$. Taking $k = \lfloor (v + r_\mathbf{a})/ 2 \pi \rfloor$,  we define $v$-dependent admissible paths $\beta_k^v$ and $\gamma_k^v$ via
\begin{align} \label{eq:bgdef}
\beta^v_k (s) & = e^{(v + r_{\mathbf{a}})(1 - s)} q_{\mathbf{b}} , & \gamma^v_k (s) & = e^{(v + r_{\mathbf{a}})(1 - s)} q_{\mathbf{c}} .
\end{align}
The $k$-subscript is redundant, and will occasionally be omitted, but will be nonetheless useful to distinguish different admissible paths. Noting that $q_{\mathbf{a}} = (-1)^{c_2} q_{\mathbf{b}} q_{\mathbf{c}}$, one also may easily compute the relation
\begin{align} \label{eq:betagamma}
e^v \beta^v (1 - s)^{-1} = \gamma^v (s).
\end{align}

\begin{proposition} \label{prop:matchingpaths}
Let $Y =  \mathbb{C}^* \times P_1$, $\pi_1 : Y \to \mathbb{C}^*$, $\Delta = \cv{\Psi_{\mathbf{D} (\mathbf{a})}}$ and $u \in \mathbb{C}$ with $\Im (u) = 2 \pi k$ and $\Re (u ) < 0$. Then: 
\begin{enumerate}[label=(\roman*),ref=\theproposition(\roman*)]
	\item \label{prop:matchingpaths:1} for $q_0 = *_{\mathbf{a}}$ and $q_1 = q_{\mathbf{a}}$, the $W_{\mathbf{a}}$-admissible curve  $\alpha_k$ is $\pi_1|_{\Delta}$-admissible,
	\item \label{prop:matchingpaths:2} the vanishing thimble $\vt_{\gamma^u_k}$ is a matching path of $\alpha_k$.
\end{enumerate}
\end{proposition} 
\begin{proof}
To prove \ref{prop:matchingpaths:1}, we observe that  equation~\eqref{eq:cvPsi} of Proposition~\ref{prop:circuitbiLefschetz} shows that $\pi_1|_{\Delta}$ is in fact equivalent to yet another one-dimensional circuit potential. Thus by Lemma~\ref{lem:circuit_basics}, there is a unique Morse critical value which must equal the critical value $q_{\mathbf{a}}$ of $\pi_1 \circ \Psi_{D (\mathbf{a})} = W_{\mathbf{a}}$ implying that any $W_{\mathbf{a}}$-admissible path is also $\pi_1|_{\Delta}$-admissible.  In the terminology of \cite[Section~15]{seidel}, we see that the fake critical value set $\textnormal{Fakev} (W^\circ_\mathbf{a}, \Psi_{\mathbf{D} ( \mathbf{a}) })$ is empty so that for every $t \ne q_{\mathbf{a}}$, the map $\Psi_{\mathbf{D} (\mathbf{a})}^t : F^\circ_t \to P_1$ is a Lefschetz fibration.

For \ref{prop:matchingpaths:2} we show there exists a map $M_{\alpha_k} : D_+ \to Y$ satisfying Definition~\ref{def:matchingpath} with $\mu_{\alpha_k} = \vt_{\gamma^u_{k}}$. To avoid notational confusion, take $\vt^{\Delta}_{\alpha_k}$ to be the vanishing thimble of $\alpha_k$ with respect to the function $\pi_1|_\Delta$. As $\pi_1|_\Delta$ is Morse, there is a parametrization  $M^\partial_{\alpha_k} : \partial_+ D_+ \to \vt^{\Delta}_{\alpha_k} \subset \Delta$ for which $\pi_1 (M^\partial_{\alpha_k} (r + s i)) =  \alpha_k (s)$. We will extend $M^{\partial}_{\alpha_k}$ to $M_{\alpha_k} : D_+ \to Y$.

First we take $u_s = u (1 - s)$  and let $k_s = \lfloor u_s / 2 \pi \rfloor$. For $s \in [0, 1)$, define $\tilde{\gamma}_s : [- \sqrt{1 - s^2}, \sqrt{1 - s^2}] \to P_1$ to be a smoothly varying parameterization of the vanishing thimble $\vt_{s}$ of $\gamma^{u_s}_{k_s}$. By the definition of $\gamma^{u_s}_{k_s}$, we have  
\begin{align*} W_{\mathbf{c}} (\tilde{\gamma}_s (\pm \sqrt{1 - s^2})) & = \frac{\alpha_k (s)}{q_{\mathbf{a}}} q_{\mathbf{c}} , \\ & = \frac{(-1)^{c_2} \alpha_k (s)}{q_{\mathbf{b}}}.
\end{align*} Thus by Corollary~\ref{cor:critvt} and the fact that $\Delta = \cv{\Psi^t_{\mathbf{D} (\mathbf{a} )}}$ we have that $\tilde{\gamma}_s (\pm \sqrt{1 - s^2} ) \in \left(\pi_1|_{\Delta} \right)^{-1} (\alpha_k (s))$.  

As $s$ tends to $1$, $\alpha_k (s)$ tends to $q_{\mathbf{a}}$ which implies that $\gamma^{\alpha_{k} (s) / q_{\mathbf{a}}}_{k_s}$ is approaching a constant path. In turn, the vanishing thimble, $\vt_s \subset P_1$ parametrized by $\tilde{\gamma}_s$ is also approaching the critical point $p_{\mathbf{c}} \in \left(\pi_1|_{\Delta} \right)^{-1} (\alpha_k (1))$. This implies that $\tilde{\gamma}_s$ sends its endpoints to the vanishing thimble of $\pi_1|_{\Delta}$ over $\alpha_k (s)$ and in particular to the second coordinate of $M^\partial_{\alpha_k} (\sqrt{1 - s^2} + i s)$. Thus we may extend $M^\partial_{\alpha_k}$ as 
\begin{align*}
M_{\alpha_k} (r + s i) = \left(\alpha_{k} (s), \tilde{\gamma} (r) \right).
\end{align*}
The fact that $M_{\alpha_k}$ satisfies Definition~\ref{def:matchingpath} to give the matching path $\tilde{\gamma}_0$ of $\alpha_k$ is an immediate consequence of its construction. As $\tilde{\gamma}_0$ parametrizes the thimble $\vt_{\gamma^{u}_k}$, the conclusion of the proposition follows.
\end{proof}

Propositions~\ref{prop:isotopy} and \ref{prop:matchingpaths} then yield descriptions of the vanishing cycles of the $W^\circ_{\mathbf{a}}$-admissible paths $\delta_k$ as pullbacks of the vanishing thimbles of $W_{\mathbf{b}}$ along the one-dimensional circuit potential $(-1)^{c_2}t / W_{\mathbf{c}}$. The language and notation for the procedure of doubling thimbles to obtain spheres was introduced in Definition~\ref{def:vs}. This description gives the final preparatory input to accomplish the induction step in the proof of Theorem~\ref{thm:Amodel}.
\begin{proposition} \label{prop:alphavc}
Suppose $u \in \mathbb{C}$ with $\Re (u) < 0$ and $\Im (u) = 2\pi k$, and let  $f  = (-1)^{c_2} e^u / W_{\mathbf{c}}$. Then, 
	\begin{enumerate}[label=(\roman*),ref=\theproposition(\roman*)]
		\item \label{prop:alphavc:1} the path $\beta^u_{k}$ is $(W_{\mathbf{b}}, f)$-admissible,
		\item \label{prop:alphavc:2} Taking $\vt_{\beta^u_k}$ to be the vanishing thimble of $W_{\mathbf{b}}$ over $\beta^u_k$, the vanishing cycle $\vc_{\alpha_k}$ is Hamiltonian isotopic to $\vs^{f}_{\beta^u_k} = f^* \vt_{\beta^u_k} $.
	\end{enumerate}
\end{proposition}
\begin{proof}
	The first claim follows from the fact that $(-1)^{c_2}e^u / q_{\mathbf{c}}$ is the unique critical value of $f$ which, by equation~\eqref{eq:bgdef}, also equals $\beta^u_k (0)$. For \ref{prop:alphavc:2}, one applies Propositions~\ref{prop:isotopy} and \ref{prop:matchingpaths:2} to see that the vanishing cycle $\vc_{\alpha_k}$ is isotopic to a matching cycle of $\Psi^{*_{\mathbf{a}}}_{\mathbf{D} (\mathbf{a})}$ over the matching path $\vt_{\gamma^u_k}$. But, by equation~\eqref{eq:betagamma},  $f (\vt_{\gamma_{k}^u}) = (-1)^{c_2} e^u / W_{\mathbf{c}} (\vt_{\gamma^u_k} ) = (-1)^{c_2} e^u /  \gamma^{u}_k = \beta^u_k$ and, more precisely,  $\vt_{\gamma_k^u}$ is the unique pullback of $\beta^u_k$ along $f$ which contains the unique critical point of $f$. Using Proposition~\ref{prop:fibprod}, $f^* \vt_{\beta_{k}^u}$ is the matching cycle fibered over $\vt_{\gamma_k^u}$ implying it is Hamiltonian isotopic to $\vc_{\alpha_k}$. 
\end{proof}
With these propositions in hand, we perform the induction step and complete the proof of our main theorem. 

\subsection{Proof of Theorem~\ref{thm:Amodel}} \label{sec:proof}
	Assume the claim is true for any  $\tilde{\mathbf{a}} \in \mathbb{Z}^{D}$ satisfying equations~\eqref{eq:balance} with $D < d + 2$. Given $\mathbf{a}, \nu_{\mathbf{a}} \in \mathbb{Z}^{d + 2}$, use equations~\eqref{eq:abc} to define $\mathbf{b}, \nu_{\mathbf{b}} \in \mathbb{Z}^{d + 1}$ and $\mathbf{c}, \nu_{\mathbf{c}} \in \mathbb{Z}^3$. We will assume that the signature of $\mathbf{a}$ is $(p, q)$ with $p > 0$, otherwise take $- \mathbf{a}$ and apply Koszul duality. 
	
 Choose $u \in \mathbb{C}$ so that $\Re (u) < 0$, $\Im (u) = 0$  and set a basepoint of $\alpha_k$ as $*_{\mathbf{a}} = e^u q_{\mathbf{a}}$. We fix the basepoint $*_{\mathbf{b}} = e^{u + r_{\mathbf{a}}} q_{\mathbf{b}}$ of $W_{\mathbf{b}}$ where $r_{\mathbf{a}}$ was defined before equation~\eqref{eq:bgdef}.
	
Note that $\Vol (\mathbf{b}) = \Vol (\mathbf{a})$ so that for $0 \leq n < \Vol (\mathbf{a})$ we may apply the induction hypothesis for $\mathbf{b}$. Thus there is a collection $\{L_0^\mathbf{b}, \ldots, L_{n}^\mathbf{b}\}$ of Lagrangian vanishing cycles corresponding to a distinguished basis of paths $\{ \beta^u_0, \ldots, \beta^u_n \} $ and an isomorphism
	\begin{align} \label{eq:equivalence2} \Xi_{\mathbf{b}, n} :  \mathcal{C}_{\mathbf{b}, \nu_{\mathbf{b}}, n}  \to \mathcal{A}_{\mathbf{b}, \nu_{\mathbf{b}},n}  \end{align}
	for which $\Xi_{\mathbf{b}, n} (R_{\mathbf{b}} (k)) = L^\mathbf{b}_k$. 
	
We will now specify a basis for $\Hom_{\mathcal{A}_{\mathbf{a}, \nu_{\mathbf{a}}, n } } (L_j, L_k )$ along with an isomorphism to $\Hom_{\mathcal{C}_{\mathbf{a}, \nu_{\mathbf{a}}, n }} (R_{\mathbf{a}} (j), R_{\mathbf{a}} (k))$.   From Proposition~\ref{prop:alphavc} the vanishing cycles $\vc_{\alpha_k}$ are isotopic to matching cycles of $\Psi^{*_{\mathbf{a}}}_{\mathbf{D} (\mathbf{a})}$ over the thimbles $\vt_{\gamma^u_{k}}$. Using the parametrization $\phi : \mathbb{C} \backslash \{0, -1\} \to P_1$ from Proposition~\ref{prop:1dthimble}, these thimbles were shown to be isotopic to $\mu_k$ in	Proposition~\ref{prop:intnumb}, or their Hamiltonian perturbations in Proposition~\ref{prop:mpd}. We recall that $\mu_k = \exp (\tilde{\mu}_k) : [C / a_0, - C / a_1] \to \mathbb{C} \backslash \{0, -1\}$ was a path where $C$ is defined by the condition 
	\begin{align} \label{eq:Cdef} W_{\mathbf{c}} (\phi (\exp {(C/a_i )}) ) = *_{\mathbf{c}} \end{align}
	for $i = 0, 1$ and $\tilde{\mu}_k$ was defined in equation~\eqref{eq:tildemudef}. 
	
	The intersections of these thimbles were found in equation~\eqref{eq:intsects} to be 
\begin{align*} \mu_j \cap \mu_k = \left\{ z^m : m \in \mathbb{Z}, - \frac{k - j}{a_0}  \leq m \leq \frac{k - j}{a_1}  \right\}
\end{align*}
 where $z^m = \exp (\tilde{\mu}_k ( C m / (j - k) ) )$. Thus, perturbing the matching cycles so that they lie over $\mu_i$, we may partition the intersection $\vc_{\alpha_j} \cap \vc_{\alpha_k}$ to obtain a decomposition of the morphism group 
\begin{align} \label{eq:dirsum} \Hom_{\mathcal{A}_{\mathbf{a}, \nu_{\mathbf{a}}, n} } (L_j, L_k ) & = \bigoplus_{\frac{k - j}{a_0}  \leq m \leq \frac{k - j}{a_1}} \Hom_{\mathbf{a}, m} (L_j, L_k).
\end{align}
Here $\Hom_{\mathbf{a}, m} (L_j, L_k)$ is generated by the intersections of $\vc_{\alpha_j} \cap \vc_{\alpha_k}$ lying over $z^m \in \mu_j \cap \mu_k$ relative to $\Psi^{*_\mathbf{a}}_{\mathbf{D} (\mathbf{a})}$. For the moment we fix $m$ and let 
\begin{align} \label{eq:m0m1} m_0 & = \max \{0, -m\}, & m_1 & = \max \{0, m\} , 
	\end{align}
so that $m$ is uniquely written as $m_1 - m_0$. Also define the constants
\begin{align} \label{eq:sigma}
\sigma_w (m ) & = a_0 m_0 + a_1 m_1, &
\sigma_d (m ) & = 2 ( \nu_{\mathbf{a}} (0) m_0 + \nu_{\mathbf{a}} (1) m_1).
\end{align}
These constants will be used repeatedly to denote shifts in weight and degree, respectively.

\begin{claim} \label{claim1} For any $0 \leq l < \Vol ( \mathbf{a}) - k + j$, there is an isomorphism 
	\begin{align} \label{eq:vcavcb} \xi^{l, m}_{j, k} :  \Hom^{\bullet - \sigma_d (m)}_{\mathcal{A}_{\mathbf{b},\nu_{\mathbf{b}}, n}} (L_{j + l}^{\mathbf{b}}, L^{\mathbf{b}}_{k + l  - \sigma_w (m)}) \stackrel{\cong}{\longrightarrow} \Hom^\bullet_{\mathbf{a}, m} (L_j, L_k).
	\end{align}
\end{claim}
\begin{proof}[Proof of Claim~\ref{claim1}]
	  First, note that either $m_1 = 0$ or $m_0 = 0$. If $m_1 = 0$, then the intersection point $z^m$ of the paths $\mu_j$ and $\mu_k$ in $P_1$ lies in the unit disc (see equation~\eqref{eq:tildemudef}). There are two sub-cases to consider here. First, if $j \equiv k \pmod{a_0}$ and $m = - m_0 =  (j - k)/a_0$ then $z^m \in \mu_j \cap \mu_k$ occurs at the endpoint of the matching paths implying the vanishing spheres intersect in a point. In this case, $k + l - \sigma_w (m) = j + l$ so that  $\mathbb{K} \cdot 1_{L_{j + l}} =  \Hom^0
	  _{\mathcal{A}_{\mathbf{b},\nu_{\mathbf{b}}, n}} (L_{j + l}^{\mathbf{b}}, L^{\mathbf{b}}_{k + l  - \sigma_w (m)})$ is isomorphic to $\Hom^\bullet_{\mathbf{a}, m} (L_j, L_k)$ up to the grading. The grading shift will follow from the discussion of the second sub-case where $m = -m_0 > (j - k) / a_0$. 
	  
	  For $m = -m_0 > (j - k) / a_0$, the line segments of the lifts $\tilde{\mu}_{j} (t) + 2 \pi i m_0 = \ln (\mu_{j} (t) )$ and $\tilde{\mu}_k (t) = \ln (\mu_k(t) )$ for $t \in \left[-C/a_0, - C m_0 / (j - k) \right]$ lie in the real negative half plane and intersect at $\tilde{z}^m = \tilde{\mu}_k ( -C m_0 / (j - k)) = -C m_0 / (j - k) - 2 \pi j m_0 i / (j - k)$. By equation~\eqref{eq:Cdef}, the constant $C$ satisfies $ \phi (\exp ( \tilde{\mu}_i ({-C / a_0}))) \in W^{-1}_{\mathbf{c}} (*_\mathbf{c})$ for any $i$.  
	  Define two paths, $\ell_{j}$ from $\tilde{\mu}_{j} (- C / a_0 ) + 2 \pi i m_0 = [- C  - 2 \pi ( j - m_0 a_0  ) i] / a_0$ to $\tilde{z}^m$, and $\ell_k$ from $\tilde{\mu}_k (- C / a_0 ) = [- C  + 2 \pi  j  i] / a_0$ to $\tilde{z}^m$.  Write $\ell_{j,k}$ for the concatenation of $\ell_j$ with $-\ell_k$ (where the negative sign indicates a reversed orientation). Note that $\ell_{j,k}$ lies in the negative real half plane. These paths are illustrated on the left in  Figure~\ref{fig:arcconnect}.

\begin{figure}[h]
\begin{picture}(0,0)%
\includegraphics{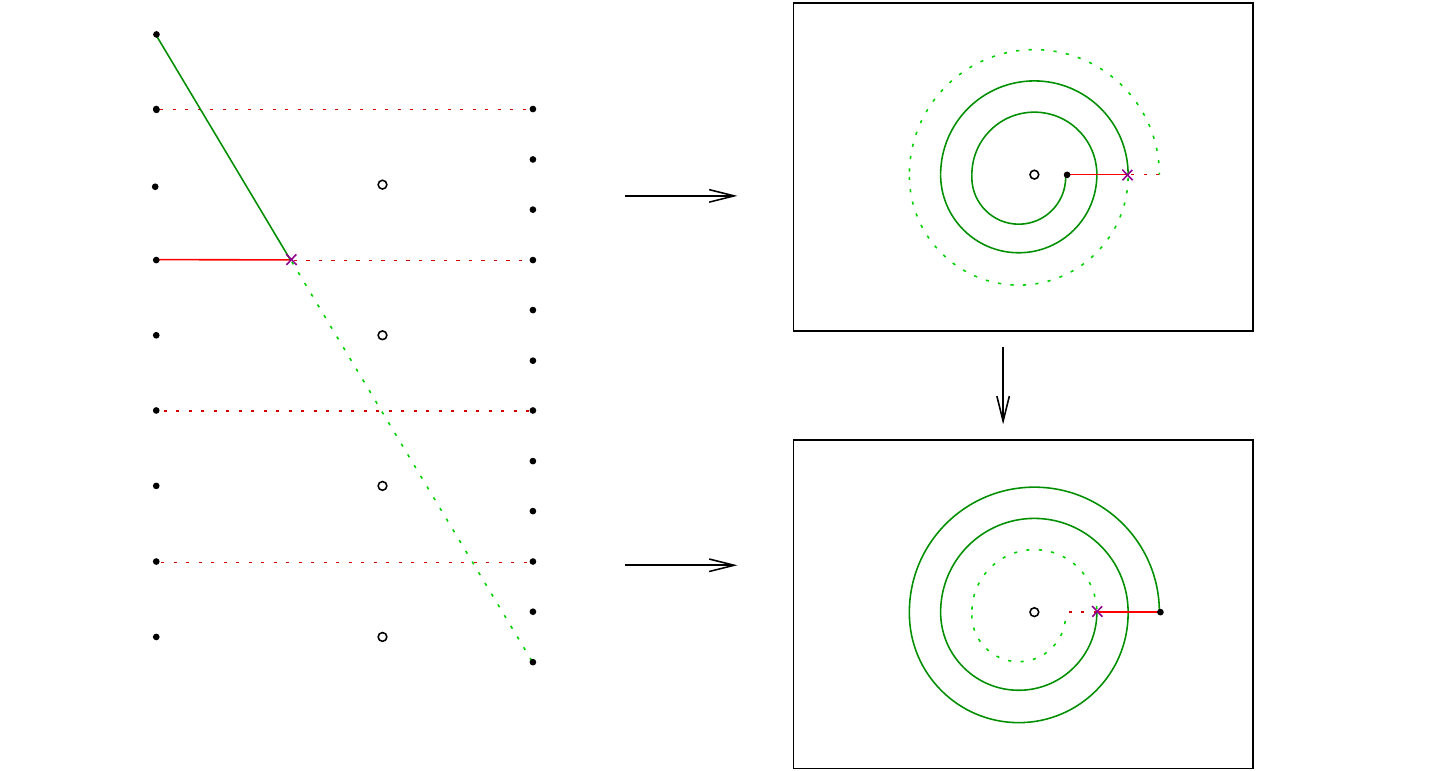}%
\end{picture}%
\setlength{\unitlength}{3947sp}%
\begin{picture}(6854,3699)(5836,-3373)
\put(6801,-41){\makebox(0,0)[lb]{\smash{$\ell_k$}}}
\put(6801,-1111){\makebox(0,0)[lb]{\smash{$\ell_j$}}}
\put(7201,-811){\makebox(0,0)[lb]{\smash{$\tilde{z}^{-1}$}}}
\put(8061,-2800){\makebox(0,0)[lb]{\smash{$\tilde{\mu}_k$}}}
\put(6751,-3211){\makebox(0,0)[lb]{\smash{$(a_0,a_1) = (2,3)$}}}
\put(8551,-436){\makebox(0,0)[lb]{\smash{$W_{\mathbf{c}} \circ \exp \circ \phi$}}}
\put(10876,-1561){\makebox(0,0)[lb]{\smash{$(-1)^{c_2} e^u z^{-1}$}}}
\put(11487,-541){\makebox(0,0)[lb]{\smash{$q_{\mathbf{c}}$}}}
\put(11507,-2674){\makebox(0,0)[lb]{\smash{$q_{\mathbf{b}}$}}}
\put(5851,-1711){\makebox(0,0)[lb]{\smash{$\tilde{\mu}_j$}}}
\put(5851,-2461){\makebox(0,0)[lb]{\smash{$\tilde{\mu}_j - 2 \pi i$}}}
\put(5851,-961){\makebox(0,0)[lb]{\smash{$\tilde{\mu}_j + 2 \pi i$}}}
\put(5851,-211){\makebox(0,0)[lb]{\smash{$\tilde{\mu}_j + 4 \pi i$}}}
\end{picture}%
\caption{\label{fig:arcconnect} Arcs $\ell_{j}$ and $\ell_k$ and their images in the $W_{\mathbf{c}}$ and $W_{\mathbf{b}}$ planes.}
\end{figure}
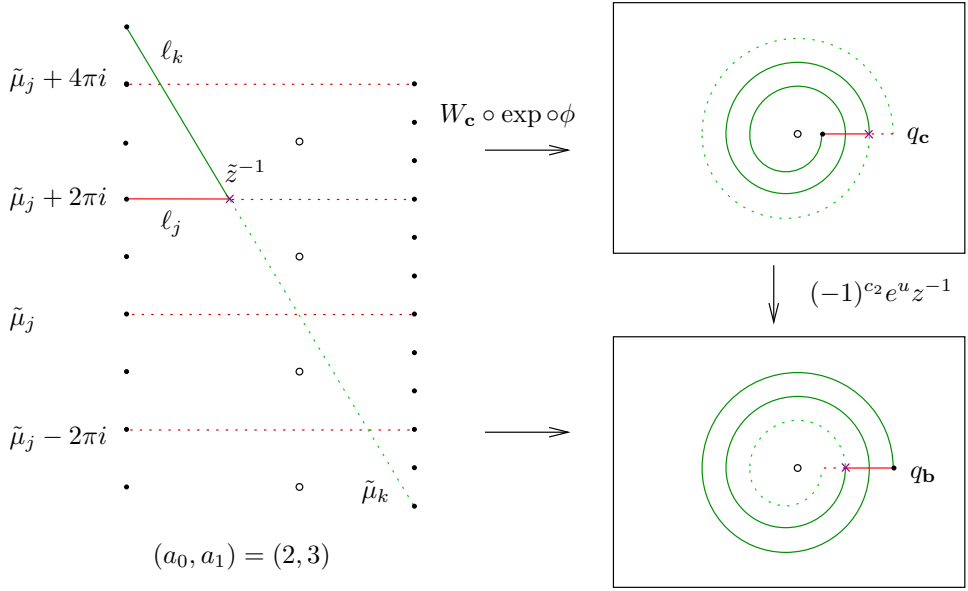

Recall from equation~\eqref{eq:explicit} that  $W_{\mathbf{c}} (\phi (z))$ has an order $a_0$ ramification at $0$ so that the image $W_{\mathbf{c}} (\phi ( \exp (\ell_{j, k})) )$ of $\ell_{j,k}$ has winding number $- ( k - j - a_0 m_0)$ as in the upper right side of Figure~\ref{fig:arcconnect}. Applying $(-1)^{c_2} e^u / W_{\mathbf{c}} \circ \phi $ to $\exp (\ell_j)$ and $\exp (\ell_k)$ give $W_{\mathbf{b}}$-admissible paths $\ell^{\mathbf{b}}_{j}$, $\ell_{k}^{\mathbf{b}}$ from $q_{\mathbf{b}}$ to $q_0 : = (-1)^{c_2} e^u / W_{\mathbf{c}} (\phi(z^m)) $. Calculating winding numbers, one observes that with $*_{\mathbf{b}} = q_0$, the paths  $\ell_j^{\mathbf{b}}$ and $\ell_k^{\mathbf{b}}$ are isotopic to $\beta^v_j$ and $\beta^v_k$ respectively, for an appropriate $v$ in the $W_{\mathbf{b}}$-plane.
	  
	  As the winding number of $\ell^{\mathbf{b}}_{k}$  concatenated with $-\ell^{\mathbf{b}}_{j}$ is $(k - j - a_0 m_0)$  and since it lies in the disc of radius $|q_{\mathbf{b}}|$, it is isotopic to a concatenation of $\delta_{k + l - a_0 m_0}$ and $-\delta_{j + l}$ where the isotopy is obtained via a theta varying  twist map $\tau_{\theta, 0}$ (defined in equation~\eqref{eq:twist}). Take $\tilde{*}$ to be the base point for the paths $\delta_{k + l - a_0 m_0}$ and $-\delta_{j + l}$. Performing  symplectic parallel transport along the $W_{\mathbf{b}}$ fibers over the isotopy gives a symplectomorphism from $W^{-1}_\mathbf{b} (q_0)$ to $W_\mathbf{b}^{-1} (\tilde{*})$ which, upon pulling back relative to $\rho$ in the fiber product~\eqref{eq:circuitbifib}, sends the vanishing cycles  $\vc_{\ell^{\mathbf{b}}_{k}}$ and $\vc_{\ell^{\mathbf{b}}_{j}}$ to $L_{k + l - a_0 m_0} = L_{k + l - \sigma_w (m)}$ and $L_{j + l}$, respectively. By incorporating the Hamiltonian isotopy into a fiberwise perturbation over $\mu_j$ and $\mu_k$ near $z^m$, we may assume this identifies these pullbacks with the fiber of the matching cycles $L_{k + l - \sigma_w (m)}$ and $L_{j + l}$ over $z^m$. This identifies the Floer complex $CF^* (\vc_{\ell^{\mathbf{b}}_{j}}, \vc_{\ell^{\mathbf{b}}_{k}} )$ with $\Hom_{\mathbf{a}, m} (L_{j + l},  L_{k + l - \sigma_w (m)} )$.
	  
Finally, to observe the shift of $\sigma_d (m)$ in the grading, we grade $L_j$ so that the isomorphism 
\begin{align*} \xi^{0, 0}_{j, k} :\Hom^\bullet_{\mathbf{a}, 0} (L_j, L_k) \to \Hom^\bullet_{\mathcal{A}_{\mathbf{b}, \nu_{\mathbf{b}}, n}} (L_j^{\mathbf{b}}, L_k^{\mathbf{b}}). \end{align*} 
respects the grading. We note that this is possible by the description of $\eta_{\nu_{\mathbf{a}}}$ given in Proposition~\ref{prop:fibprod}. Utilizing this description again, one observes that performing symplectic parallel transport relative to $\Psi^{*_{\mathbf{a}}}_{\mathbf{D} (\mathbf{a})}$ along a counter-clockwise path once around the origin yields a grading shift of $2 \nu_{\mathbf{c}} (0) = 2 \nu_{\mathbf{a}} (0)$ (as $\eta_{\nu_{\mathbf{c}}}$ has order $\nu_{\mathbf{c}} (0) - 1$ at the origin).  The winding number of the path from $z_0$ to $z_m$ along $\mu_k$ and then to $z_0$ along $\mu_j$ is precisely $m_0$. Thus the grading shift for $\xi_{j,k}^{l, m}$ is $2 m_0 \nu_{\mathbf{a}} (0) = \sigma_d (m)$.

The case of $m_0 = 0$ yields an identical argument, with the exception of working with $\tilde{\mu}_j$ and $\tilde{\mu}_k$ the positive half plane, so we omit this repetition.
\end{proof} 

We note that $\xi^{l, m}_{j, k}$ is independent of $l$ in the sense that the monodromy of $W^{-1}_{\mathbf{b}} (*_\mathbf{b})$ obtained by winding around the origin $l$ times takes $L^\mathbf{b}_i$ to $L^{\mathbf{b}}_{i + l}$ inducing an equivalence $\Phi_l$ on $\mathcal{F} (W_{\mathbf{b}}^{-1} (*_\mathbf{b}))$. From the construction, it is clear that $\xi^{l + l^\prime, m}_{j, k} \circ \Phi_{l^\prime} = \xi^{l, m}_{ j, k} $.  Taking $\xi_{j, k}^l = \oplus_m \xi_{j, k}^{l, m}$, and using equation~\eqref{eq:dirsum} , we obtain the isomorphism 
\begin{align} \label{eq:xidef}
 	\bigoplus_{-\frac{k - j}{a_0}  \leq m \leq \frac{k - j}{a_1}} \Hom_{\mathcal{A}_{\mathbf{b}, \nu_{\mathbf{b}}, n}}^{\bullet - \sigma_d (m)} (L_{j + l}^{\mathbf{b}}, L^{\mathbf{b}}_{k + l  - \sigma_w (m)})  \stackrel{\xi_{j, k}^l}{\longrightarrow} \Hom^{\bullet}_{\mathcal{A}_{\mathbf{a}, \nu_{\mathbf{a}}, n}} (L_j, L_k).
\end{align}

Turning to the $B$-model, notationally distinguish $R_{\mathbf{a}, \nu_{\mathbf{a}}}$ and $R_{\mathbf{b}, \nu_{\mathbf{b}}}$ by taking the basis $\{v_0, \ldots, v_{d + 1}\}$ for $V$  and $\{w_0, \ldots, w_{d}\}$ for $W$, and defining $R_{\mathbf{a}, \nu_{\mathbf{a}}} = \symalg{V}$ and $R_{\mathbf{b}, \nu_{\mathbf{b}}} = \symalg{W}$. Here, weights and degrees are assigned according to equation~\eqref{eq:degrees} with respect to $\mathbf{a}$, $\nu_{\mathbf{a}}$ and $\mathbf{b}$, $\nu_{\mathbf{b}}$ respectively.  Given $x = \prod_{i = 0}^{d + 1} v_i^{r_i} \in R_{\mathbf{a}, \nu_{\mathbf{a}}}$, define
\begin{align*} m_0 (x) & = r_0 - \min \{r_0, r_1\} , & m_1 (x) & = r_1 - \min \{r_0, r_1\} . 
\end{align*} 
For $ s \in \mathbb{N}$ and $m \in \mathbb{Z}$ satisfying $- s / a_0 \leq m \leq s / a_1$ take $m_0$, $m_1$ as in equation~\eqref{eq:m0m1}, so that $m = m_1 - m_0$, and $\sigma_d (m), \sigma_w (m)$ as in equation~\eqref{eq:sigma}.
Consider the space $\symalgh{V}{s}$ of homogeneous elements of weight $s$ and  define a projection 
\begin{align*}\chi_{s}^m : \symalgh{V}{s} \to \symalgh{W}{s - \sigma_w (m)}[\sigma_d (m)] \end{align*} by taking
\begin{align*}
	\chi_s^m (x ) =  \delta_{m ,m_1(x) -  m_0 (x)} w_0^{\min \{r_0, r_1\}}   \prod_{i = 1}^d w_i^{r_{i + 1}}.
\end{align*} 
The intuition behind $\chi_s^m$ is to identify $v_0 v_1$ with $w_0$ and divide $x$ by $v_0^{m_0 (x)}$ or $v_1^{m_1 (x)}$, thereby decreasing the weight (and degree) of $x$ by $a_0 m_0 (x)$ or $a_1 m_1(x)$ (and $\sigma_d (m)$) respectively. It is an elementary check to show that the direct sum of these maps yields an isomorphism 
 \begin{align} \label{eq:defchis}
 	\chi_s = \oplus \chi_s^m  : \symalgh{V}{s} \stackrel{\cong}{\longrightarrow} \bigoplus_{-\frac{s}{a_0}  \leq m \leq \frac{s}{a_1}} \symalgh{W}{s - \sigma_w (m)}[\sigma_d (m)].
 \end{align} 
By Definition~\ref{def:can} of $\mathcal{C}_{\mathbf{a}, \nu_{\mathbf{a}}, n}$, for any $j, k, l \in \mathbb{Z}$ with $j < k$ and $0 \leq l < \Vol (\mathbf{a})  - k + j$, there are natural maps which define the isomorphism $\chi^l_{j,k}$ in the following commutative diagram
\begin{equation*}
\begin{tikzpicture}[baseline=(current  bounding  box.center), scale=1.5]
\node (A) at (0,0) {$\symalgh{V}{k - j}$};
\node (B) at (4.5,0) {$\bigoplus_{m} \symalgh{W}{k - j - \sigma_w (m)}[\sigma_d (m)]$};
\node (C) at (0,1) {$\Hom^\bullet_{\mathcal{C}_{\mathbf{a}, \nu_{\mathbf{a}}, n}} (R_{\mathbf{a}} (j), R_{\mathbf{a}} (k))$};
\node (D) at (4.5,1) {$\bigoplus_{m} \Hom^{\bullet - \sigma_d (m)}_{\mathcal{C}_{\mathbf{b}, \nu_{\mathbf{b}}, n}} (R_{\mathbf{b}} (j + l), R_{\mathbf{b}} (k + l - \sigma_w (m)))$};
\path[->,font=\scriptsize]
(A) edge node[above]{$\chi_{k - j}$} (B)
(C) edge node[left]{$\cong$} (A)
(B) edge node[right]{$\cong$} (D)
(C) edge node[above]{$\chi_{k,j}^l$} (D);
\end{tikzpicture} 
\end{equation*}	
where the direct sum is for integers $m$ satisfying $ - \frac{k - j}{a_0}  \leq m \leq \frac{k - j}{a_1}$.

Given this preparation, we are now able to extend the functor $\Xi_{\mathbf{a}, n}$ to morphisms. For any $l \in \mathbb{Z}$  define  
\begin{align*}
	\Xi_{\mathbf{a}, n}: \Hom^\bullet_{\mathcal{C}_{\mathbf{a}, \nu_{\mathbf{a}}, n} } (R_\mathbf{a} (j), R_\mathbf{a} (k)) \to \Hom^\bullet_{\mathcal{A}_{\mathbf{a}, \nu_{\mathbf{a}}, n}} (L_j, L_k)  \end{align*}
as the composition 
\begin{align*}
 \Xi_{\mathbf{a}, n}: = \xi^l_{j, k} \circ \Xi_{\mathbf{b}, n} \circ  \chi_{j, k}^l.
\end{align*}
By Claim~\ref{claim1}, equation~\eqref{eq:defchis} and the induction hypothesis, we see that $\Xi_{\mathbf{a}, n}$ is an isomorphism of graded vector spaces. To conclude that $\Xi_{\mathbf{a},n}$ is an isomorphism, we need only show that it commutes with multiplication.

\begin{claim} \label{claim2} $\Xi_{\mathbf{a}, n}$ defines a functor. In particular, for $x_0 \in \Hom_{\mathcal{C}_{\mathbf{a}, \nu_{\mathbf{a}}, n}} (R_{\mathbf{a}} (i), R_\mathbf{a} (j))$ and  $x_1 \in \Hom_{\mathcal{C}_{\mathbf{a}, \nu_{\mathbf{a}}, n}} (R_{\mathbf{a}} (j), R_\mathbf{a} (k))$ we have
	\begin{align} \label{eq:commutes}
	\Xi_{\mathbf{a}, n} (x_1x_0 ) = m_2^{\mathbf{a}} \left(\Xi_{\mathbf{a}, n} (x_1), \Xi_{\mathbf{a}, n} (x_0) \right).
	\end{align}
\end{claim}
\begin{proof}[Proof of Claim~\ref{claim2}] 
 We start on the $A$-model side with an observation. Suppose $0 \leq k_0 < k_1 < \cdots < k_s < n$ and $y_i \in \Hom_{{\mathbf{a}}, l_{i}} (L_{k_{i - 1}}, L_{k_{i}})$ for $1 \leq i \leq s$. If $m^\mathbf{a}_{s}$ denotes the $s$-th $A_\infty$-multiplication map in $\mathcal{A}_{\mathbf{a}, \nu_{\mathbf{a}}, n}$ then for $s \ne 2$ we have
 	\begin{align*}
 		m^\mathbf{a}_{s} (y_{l}, \ldots, y_{0}) = 0 .	
 	\end{align*}
 This follows from the observation that we may choose a regular complex structure for which $\Psi^{*_\mathbf{a}}_{\mathbf{D} (\mathbf{a})}$ is holomorphic (one may prove regularity by induction). Applying Proposition~\ref{prop:mpd:1} to obtain a generic collection of Hamiltonian perturbations $\mathcal{H}$ for $\mathbf{\mu}_{\mathbf{a}, n}$ we may identify the vanishing cycles $L_{k_i}$ as matching paths over the transversely intersecting $\psi^{1} \mu_{k_i}$ so that any holomorphic disc $\varphi : D \to P_{d} = E_{\mathbf{D}(\mathbf{a})} \cup D_{h, -1}$ with Lagrangian boundary conditions on $L_{k_i}$ must map via $\Psi^{*_\mathbf{a}}_{\mathbf{D} (\mathbf{a})}$ to a directed disc in $P_1 \cup \{[-1:1]\} \cong \mathbb{C}^*$ relative to the perturbed collection $\mathbf{\mu}_{\mathbf{a}, n}$. By Proposition~\ref{prop:mpd}, such discs are either constant or holomorphic triangles. In the former case, $\textnormal{im} (\varphi ) = p \in  \mu_i \cap \mu_j$ (where we again drop the perturbation term from the notation), but as no other vanishing cycle lies over this point, we have that $\textnormal{im} (\varphi )$ must be a bi-gon in $\left(\Psi^{*_\mathbf{a}}_{\mathbf{D} (\mathbf{a})} \right)^{-1} (q) \cong W_{\mathbf{b}}^{-1} ((-1)^{c_2}*_\mathbf{a} / W_{\mathbf{c}} (q) )$. By the induction hypothesis utilized in the previous claim to identify the fibers of the matching cycles over $p$ with $L_i^\mathbf{b}$ and $L_j^{\mathbf{b}}$, the differential on the Floer complex $CF^* (L_i^\mathbf{b}, L_j^{\mathbf{b}}) $ is zero. 
 
Before proceeding further, let us simplify and detail our notation by taking $k_0 = i$, $k_1 = j$, $k_2 = k$, $y_0 \in \Hom_{\mathbf{a}, l_0} (L_i, L_j)$ and $ y_1 \in \Hom_{\mathbf{a}, l_1} (L_j, L_k)$. We will assume that $y_0 = \Xi_{\mathbf{a}, n} (x_0 )$ and $y_1 = \Xi_{\mathbf{a}, n} (x_1)$ where $x_0$ and $x_1$ are monomials in $\mathcal{C}_{\mathbf{a}, \nu_{\mathbf{a}}, n}$.  Applying Proposition~\ref{prop:mpd:3} we observe that 
\begin{align*} m^\mathbf{a}_{2} (y_1, y_0) \in \Hom_{\mathbf{a}, l_{0} + l_{1}} (L_{i}, L_{k}).
\end{align*} 
Moreover, we have two potential scenarios as outlined in  Proposition~\ref{prop:mpd:a} and \ref{prop:mpd:b}. 
\\
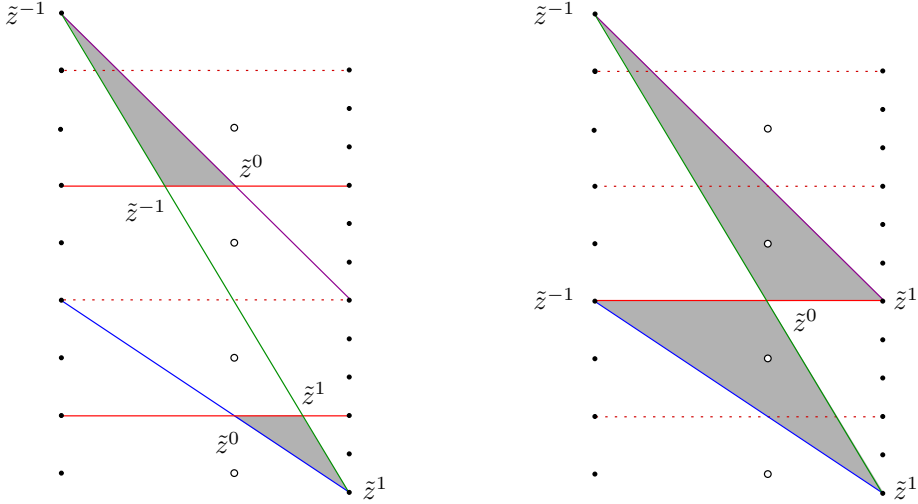
\begin{figure}[h]
\begin{picture}(0,0)%
 \includegraphics{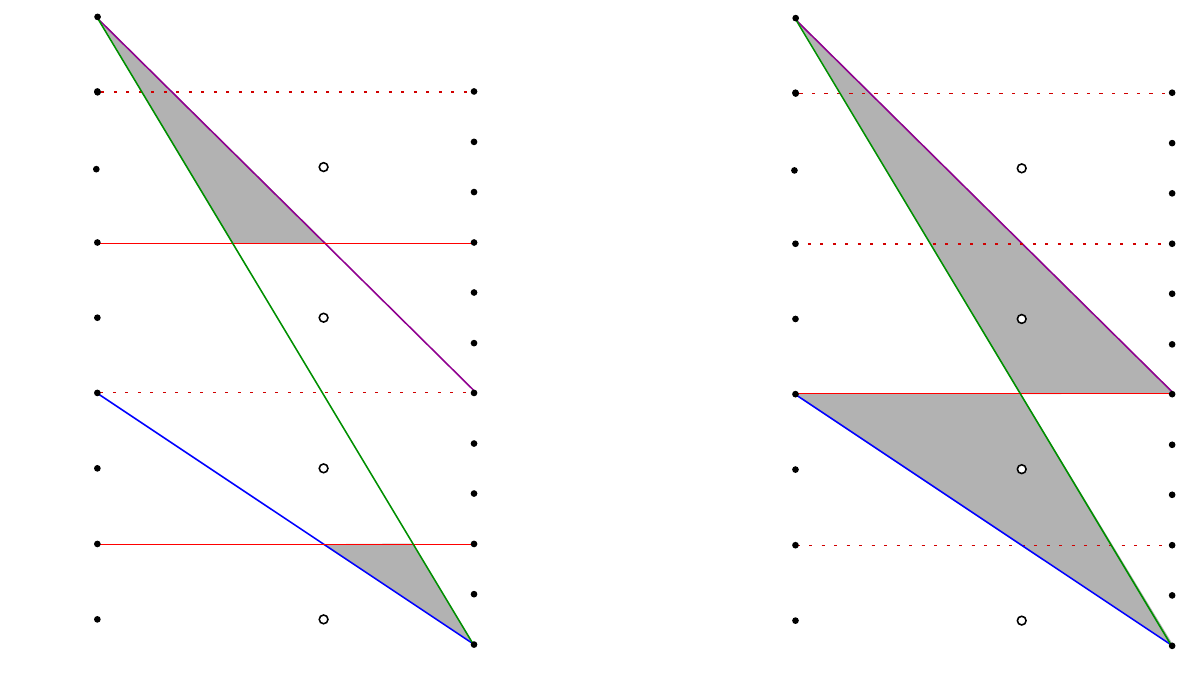}%
 \end{picture}%
 \setlength{\unitlength}{3947sp}%
 \begin{picture}(5741,3223)(2767,-2985)
 \put(4742,-2289){\makebox(0,0)[lb]{\smash{$\tilde{z}^1$}}}
 \put(5126,-2886){\makebox(0,0)[lb]{\smash{$\tilde{z}^1$}}}
 \put(4349,-863){\makebox(0,0)[lb]{\smash{$\tilde{z}^0$}}}
 \put(3632,-1142){\makebox(0,0)[lb]{\smash{$\tilde{z}^{-1}$}}}
 \put(4200,-2600){\makebox(0,0)[lb]{\smash{$\tilde{z}^0$}}}
 \put(2882, 100){\makebox(0,0)[lb]{\smash{$\tilde{z}^{-1}$}}}
 \put(8461,-1695){\makebox(0,0)[lb]{\smash{$\tilde{z}^1$}}}
 \put(8461,-2886){\makebox(0,0)[lb]{\smash{$\tilde{z}^1$}}}
 \put(7828,-1831){\makebox(0,0)[lb]{\smash{$\tilde{z}^0$}}}
 \put(6200,-1695){\makebox(0,0)[lb]{\smash{$\tilde{z}^{-1}$}}}
 \put(6200, 100){\makebox(0,0)[lb]{\smash{$\tilde{z}^{-1}$}}}
 \end{picture}%
\caption{\label{fig:triangles} Directed triangles for $(a_0, a_1) = (2, 3)$.}
\end{figure} 

\textit{Case 1:} $l_0$ and $l_1$ have the same sign.
 
  We assume $l_i < 0$ as the proof when $l_i > 0$ is analogous. 
  Then we may uniquely decompose $x_0 = v_0^{-l_0} x^\prime_0$ and $x_1 = v_0^{-l_1} x^\prime_1$ so that $m_r (x^\prime_s) = 0 $ for $r, s \in \{0, 1\}$. Let $\tilde{x}_0 = \chi_{i,j}^0 (x_0)$ and $\tilde{x}_1 = \chi_{j, k}^{l_0} (x_1)$ be the associated monomials in $\mathcal{C}_{\mathbf{b},n}$. Then, by the definition of $\chi_s^m$ one easily shows 
  \begin{align} \label{eq:basic1}
  	\chi_{i,k}^{0} (x_1 x_0) = \chi^0_{i, k + a_0 ( l_0 + l_1 )} (x^\prime_1 x^\prime_0) =  \tilde{x}_1 \tilde{x}_0.
  \end{align}
  
  Now, letting $l_2 = l_0 + l_1$, Proposition~\ref{prop:mpd} implies there is a unique directed disc $D_{l_1,l_0}$ in $\mathbb{C} \backslash \{-1, 0\} \cong P_1$ which bounds $\mu_i, \mu_j$ and $\mu_k$ and takes the marked points $\zeta_i$ to $z^{l_i}$ for $i = 0, 1, 2$. This is illustrated in the logarithmic pullback on the left hand side of Figure~\ref{fig:triangles}. Write $\bar{D}_\mathbf{b}: = (-1)^{c_2} e^u / W_{\mathbf{c}} ( D_{l_1,l_0} )$ for its image in the $W_{\mathbf{b}}$-plane, take $p_i = (-1)^{c_2} e^u / W_{\mathbf{c}} (z^{l_i})$ to be its vertices and $D_{\mathbf{b}} = \bar{D}_{\mathbf{b}} \backslash \{p_0, p_1, p_2\}$ its pointed disk which we equip with strip like ends. The moduli space of discs with boundary conditions $(L_i, L_j, L_k)$ whose strip like ends converge to $y_0$, $y_1$ and $y_2$ then equals the moduli space of sections of $W_{\mathbf{b}} : P_{d- 1} \to \mathbb{C}^*$  over the disc $D_\mathbf{b}$ with boundary conditions on the image of the vanishing cycles $L_i, L_j, L_k$ in $P_{d - 1}$. Since $D_{\mathbf{b}}$ is a contractible disc, $W_{\mathbf{b}}|_{D_{\mathbf{b}}}$   may be globally trivialized, using symplectic parallel transport along the boundary. The moduli space of sections with boundary conditions given by this transport map is isotopic to one with constant Lagrangian boundary conditions given by taking  $L^{\mathbf{b}}_i$ along $\partial_0 D_{\mathbf{b}}$, $L^{\mathbf{b}}_{j + a_0 l_0}$ along $\partial_1 D_{\mathbf{b}}$ and $L^{\mathbf{b}}_{k + a_0 ( l_0 +  l_1)}$ along $\partial_2 D_{\mathbf{b}}$. It follows from the construction of $\xi$ that the component of this moduli space defining  $m^\mathbf{a}_2 (y_1, y_0)$ has strip like ends converging to $\tilde{y}_0 \in L^\mathbf{b}_i \cap L_{j + a_0 l_0}^{\mathbf{b}}$ and $\tilde{y}_1 \in L^{\mathbf{b}}_{j + a_0 l_0}  \cap L^{\mathbf{b}}_{k + a_0 (l_0 +  l_1)}$ where $\xi_{i,j}^{0, l_0}(  \tilde{y}_0 ) = y_0$ and $\xi_{j,k}^{a_0 l_0, l_1} (\tilde{y}_1 ) = y_1$.  Such a space of sections is cobordant to the moduli space defining $m^{\mathbf{b}}_2$ in $\mathcal{A}_{\mathbf{b}, \nu_{\mathbf{b}}, n}$ which, by the induction hypothesis, is cobordant to a single point defining $m^{\mathbf{b}}_2 (\tilde{y}_1 , \tilde{y}_0) = \tilde{y}_2$ where $ \tilde{y}_2  = \Xi_{\mathbf{b}, n} (\tilde{x}_1 \tilde{x}_0)$. This implies 
  \begin{align} \label{eq:basic2} m_2^{\mathbf{a}} (y_1, y_0) = \xi_{i, k}^{0,l_2} (\tilde{y}_2 ). \end{align} 
  
  Thus, collecting equations~\eqref{eq:basic1} and \eqref{eq:basic2} yields
    \begin{align*}
    	\begin{split}
    		\Xi_{\mathbf{a},n} (x_1 x_0) & =  \xi_{i, k}^{0,l_2} \left(\Xi_{\mathbf{b}, n} ( \chi_{i, k}^0 (x_1 x_0 ) ) \right) , \\ & = \xi_{i, k}^{0,l_2} \left( \Xi_{\mathbf{b}, n} (\tilde{x}_1 \tilde{x}_0) \right) ,  \\ & = \xi_{i, k}^{0,l_2} ( \tilde{y}_2 ) , \\ & =  m_2^{\mathbf{a}} \left( y_1, y_0 \right) , \\ & = m_2^{\mathbf{a}} \left( 	\Xi_{\mathbf{a},n} (x_1 ), 	\Xi_{\mathbf{a},n} (x_0) \right)
    	\end{split}
    \end{align*} 
which validates equation~\eqref{eq:commutes} and the claim for the first case.    \\

\textit{Case 2}: Either $(i,j,k, l_0, l_1 ) = (i, i + a_0  , i +   a_0 + a_1 ,-1 , 1)$ or $(i,j,k, l_0, l_1) = (i, i +  a_1 , i +  a_0 + a_1,  1, -1)$.

The case of $l_0$ and $l_1$ having opposite signs reduces to this case using associativity and a basic induction argument. The two sub-cases are proved in analogous ways, so we assume $-l_0 = 1 =   l_1$ as illustrated in the bottom right triangle in Figure~\ref{fig:triangles}. Using the isomorphism $\xi^{l}_{j,k}$ defined in equation~\eqref{eq:xidef} we have that the morphisms $y_0 = \xi^{0}_{i,j} (\tilde{x}_0 ) \in \Hom_{\mathbf{a}, l_0} (L_i, L_j )$ and $y_1 = \xi^{a_0}_{j,k} (\tilde{x}_1 ) \in \Hom_{\mathbf{a}, l_1} (L_j, L_k )$ where $\tilde{x}_0 , \tilde{x}_1 \in \Hom_{\mathcal{A}_{\mathbf{b}, \nu_{\mathbf{b}}, n}} (L^\mathbf{b}_i, L_{i}^\mathbf{b} )$ both represent the identity morphism. As $\Xi_{\mathbf{b}, n}$ is a unital functor, this implies that $y_0 = \xi^{0}_{i,j} \left(  \Xi_{\mathbf{b}, n} (e_0 ) \right)$ and $y_1 = \xi^{a_0}_{j,k} \left(  \Xi_{\mathbf{b}, n} (e_1 ) \right)$ where $e_0, e_1 \in \Hom (R_{\mathbf{b}} (i), R_{\mathbf{b}} (i))$ are identities. By the definition of $\chi$ we have that $e_0 = \chi_{i, j}^0 (v_0 )$ and $e_1 = \chi_{j, k}^{a_0} (v_1)$.  This implies that $y_0  = 	\Xi_{\mathbf{a}, n} (v_0 )$ and $y_1 = \Xi_{\mathbf{a}, n} (v_1 )$. Thus, to prove the claim, we need only show that
\begin{align} \label{eq:functor}
  m_2^{\mathbf{a}} (y_1, y_0 ) = \Xi_{\mathbf{a}, n} (v_1 v_0 ).
\end{align}

By Proposition~\ref{prop:mpd:b}, the product $m_2^{\mathbf{a}} (y_1, y_0 )$ is a linear combination of the set of intersection points of $L_i$ and $L_k$ lying over $z^0$ or $\Hom_{\mathbf{a}, 0} (L_i, L_{i + a_0 + a_1}) \cong \Hom_{\mathbf{A}_{\mathbf{b}, \nu_{\mathbf{b}}, n}} (L^{b}_i, L^{b}_{i + a_0 + a_1} )$. In other words, there are moduli spaces  $ \mathcal{M}_{\tilde{x}}$ for which 
\begin{align}\label{eq:product}
	m_2^{\mathbf{a}} (y_1, y_0 ) = \sum_{\tilde{x} \in L^{b}_i \cap L^{b}_{i + a_0 + a_1}}  \left( \#  \mathcal{M}_{\tilde{x}} \right) \xi_{i, i + a_0 + a_1}^{0} ( \tilde{x} ).
\end{align}  

Let us consider the spaces $\mathcal{M}_{\tilde{x}}$ in more detail. Take  $D$ to be the holomorphic triangle in $\mathbb{C}$ with boundary along $\tilde{\mu}_i, \tilde{\mu}_j$ and $\tilde{\mu}_k - 2 \pi i$. One can easily compute that there exists exactly one element $ p = - \pi i \in (\pi i + 2 \pi i \mathbb{Z}) \cap D$ and that $\bar{D} = \phi (\exp (D))$ is a holomorphically embedded triangle in $\mathbb{C}^*$ (after applying the appropriate generic Hamiltonian perturbations $\mathcal{H}$).   Note that we are now working in the fiber $F_t$ as opposed to $F^\circ_t$ and applying the right hand diagram in Proposition~\ref{prop:fibprod}. 
The moduli space of $\mathcal{M}_{\tilde{x}}$ consists of the space of sections of $\Psi^t_{\mathbf{D} (\mathbf{a})} : F_t \to \mathbb{C}^*$ over $\bar{D}$ with boundary conditions $L_i, L_j$ and $ L_k$ over $\mu_i, \mu_j$ and $\mu_k$ respectively.  As the paths in $\mathbb{C}^*$ are matching paths over which the Lagrangians are matching cycles, and as the intersections $y_0 \in L_i \cap L_j$ and $y_1 \in L_j \cap L_k$ occur over the endpoints, the Lagrangian boundary conditions  near $y_0$ and $y_1$ are exponentially converging and, applying an implicit function theorem (see, for example, \cite{ms}), we may round the corners $z^{-1}, z^1$ of $\bar{D}$ to obtain a domain $S$ as illustrated in its logarithmic preimage in Figure~\ref{fig:rounding}.
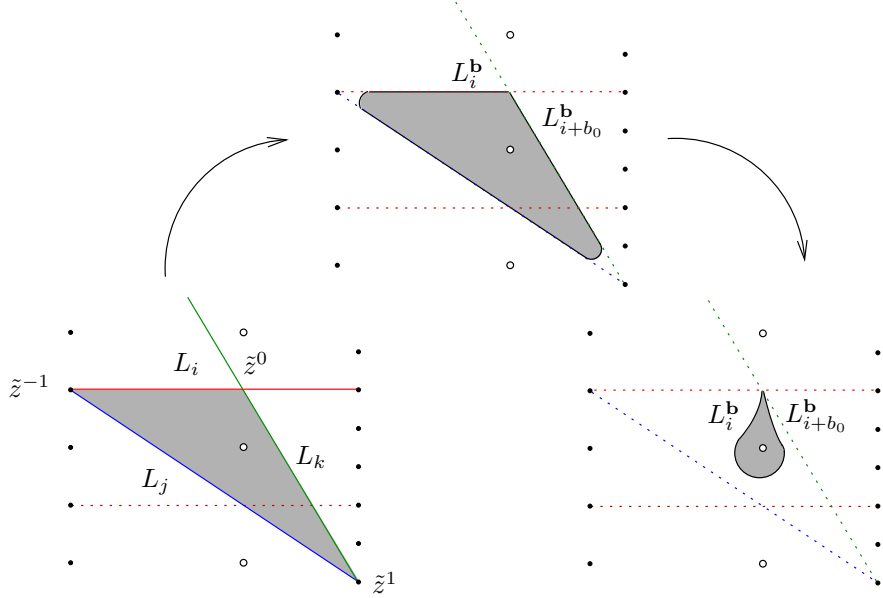
\begin{figure}[h]
	\begin{picture}(0,0)%
	\includegraphics{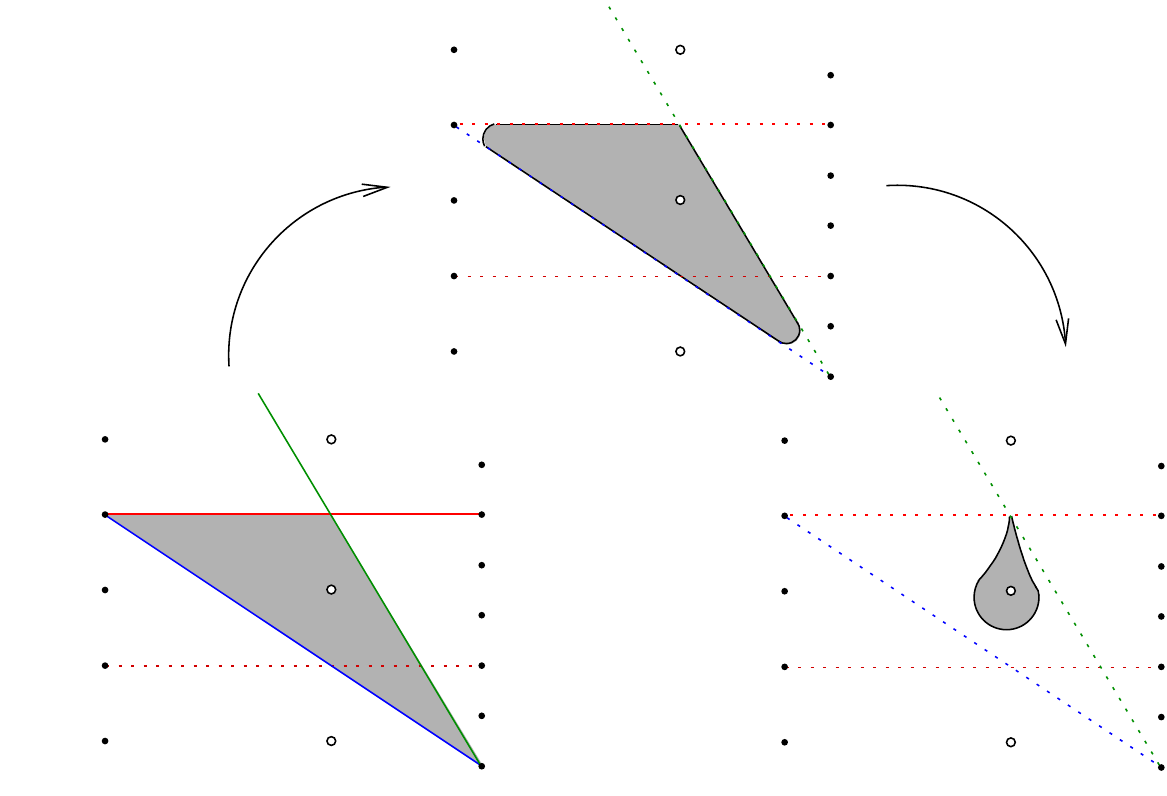}%
	\end{picture}%
	\setlength{\unitlength}{3947sp}%
	\begin{picture}(5598,3811)(3153,-2538)
	\put(4301,-1076){\makebox(0,0)[lb]{\smash{$L_i$}}}
	\put(4101,-1805){\makebox(0,0)[lb]{\smash{$L_j$}}}
	\put(5069,-1658){\makebox(0,0)[lb]{\smash{$L_k$}}}
	\put(6047,748){\makebox(0,0)[lb]{\smash{$L_i^{\mathbf{b}}$}}}
	\put(6615,464){\makebox(0,0)[lb]{\smash{$L_{i + b_0}^{\mathbf{b}}$}}}
	\put(7658,-1402){\makebox(0,0)[lb]{\smash{$L_i^{\mathbf{b}}$}}}
	\put(8144,-1393){\makebox(0,0)[lb]{\smash{$L_{i + b_0}^{\mathbf{b}}$}}}
	\put(4741,-1090){\makebox(0,0)[lb]{\smash{$\tilde{z}^0$}}}
	\put(5552,-2469){\makebox(0,0)[lb]{\smash{$\tilde{z}^1$}}}
	\put(3268,-1235){\makebox(0,0)[lb]{\smash{$\tilde{z}^{-1}$}}}
	\end{picture}%
	\caption{\label{fig:rounding} Perturbing boundary conditions for $\mathcal{M}_{\tilde{x}}$. }
\end{figure}
 The Lagrangian boundary condition thus obtained is the pullback of $L_i^{\mathbf{b}}$ near the lower boundary of the strip like end at $z^0$, along with its parallel transport counter-clockwise around $\partial S$. Sections of $\Psi_{\mathbf{D} (\mathbf{a})}^t$ over $S$ with this boundary condition gives  moduli space $\tilde{M}_{\tilde{x}}$ cobordant to $\mathcal{M}_{\tilde{x}}$ and, as $\Psi^t_{\mathbf{D} (\mathbf{a})}$ is regular on $S$ away from $0$, we may isotope $S$ to be a small disc around $-1$, simultaneously performing parallel transport to maintain the Lagrangian boundary condition of $L_i^{\mathbf{b}}$ in the lower boundary near $z_0$. Take $\bar{E}_S$ to be the inverse image of $\left(\Psi^t_{\mathbf{D} (\mathbf{a})} \right)^{-1} (S)$. As $g : = (-1)^{c_2} e^u / W_{\mathbf{c}} (\phi (z))$ has a $c_2 = (a_0 + a_1)$ order ramification at $-1$, using Proposition~\ref{prop:fibprod}, we have that the pullback
\begin{equation*}
\begin{tikzpicture}[baseline=(current  bounding  box.center), scale=2.3]
\node (A2) at (2.5,1) {$E_S$};
\node (B2) at (3.5,1) {$P_{d - 1} \cup C_0$};
\node (C2) at (2.5,0) {$S$};
\node (D2) at (3.5,0) {$\mathbb{C}.$};
\path[->,font=\scriptsize]
(A2) edge node[above]{$\rho$} (B2)
(B2) edge node[right]{$\bar{W}_{\mathbf{b}}$} (D2)
(A2) edge node[left]{$\Psi^t_{\mathbf{D} (\mathbf{a})}$} (C2)
(C2) edge node[below]{$g$} (D2);
\end{tikzpicture}
\end{equation*}
is symplectically isomorphic to the pullback in equation~\eqref{eq:pullback2}  with $n = a_0 + a_1$ and  $\tilde{E}_S = E_S$. By performing symplectic monodromy, we  may replace the Lagrangian boundary condition of the symplectic transport of $L_i^\mathbf{b}$ with $L_0^{\mathbf{b}}$. Then from equation~\eqref{eq:refmodspace}, moduli space of sections decomposes as 
\begin{align*}
\mathcal{M}_{c_2} ( \bar{W}_\mathbf{b} ) = \sqcup_{\tilde{x} \in L_i^{\mathbf{b}} \cap L_{i + a_0 + a_1}^\mathbf{b}} \mathcal{M}^{\tilde{x}, 0}_{a_0 + a_1} (\bar{W}_{\mathbf{b}} ),
\end{align*}
where $ \mathcal{M}^{\tilde{x}, 0}_{a_0 + a_1} (\bar{W}_\mathbf{b} ) = \mathcal{M}_{\tilde{x}}$. Since $b_0 = a_0 + a_1$, the definition of $\kappa_{b_0}^j (p)$ in equation~\eqref{eq:defkappa} gives \begin{align*}
\# \mathcal{M}_{\tilde{x}} = \# \mathcal{M}^{\tilde{x}, 0}_{a_0 + a_1} (\bar{W}_\mathbf{b} )=
\kappa^0_{b_0} (\tilde{x}). \end{align*}
	
By our induction hypothesis, we then have that $\# \mathcal{M}_{\tilde{x}}  = 1$ if and only if $\tilde{x} = \Xi_{\mathbf{b}, n} (w_0)$ and zero otherwise. Returning to equation~\eqref{eq:product} we have 
\begin{align*}
	m_2^{\mathbf{a}} (y_1, y_0) & = \xi^{0}_{i, i + a_0 + a_1} (\Xi_{\mathbf{b}, n} (w_0)),\\ & = \xi^{0}_{i, i + a_0 + a_1} \left(\Xi_{\mathbf{b}, n} (\chi_{i, i + a_0 + a_1}^0 (v_1 v_0)) \right), \\ & = \Xi_{\mathbf{a}, n} (v_1 v_0 ),
\end{align*}
validating equation~\eqref{eq:functor} and the claim.
\end{proof} 

To conclude the proof of Theorem~\ref{thm:Amodel}, we only need to verify equation~\eqref{eq:divisor} for $W_{\mathbf{a}}$ and $n = a_0$. Indeed, by permuting the indices, this argument then gives the result for any $a_i$ with $0 \leq i \leq p$ and by inverting $W_{\mathbf{a}}$ (which transposes the signature of $\mathbf{a}$), the argument can be run for $p + 1 \leq i \leq d + 1$. 

We recall the necessary preliminaries to define $\kappa^0_{a_0}$ in this case. The bottom row of diagram~\eqref{eq:circuitbifib2} gives that map $\bar{W}_{\mathbf{a}} : \bar{E}_{\mathbf{a}, 0} \to \mathbb{C}$ factors through $\Psi_{\mathbf{D} (\mathbf{a})} : \bar{E}_{\mathbf{a}, 0} \to \mathbb{C} \times \mathbb{C}$. Take $\bar{\bar{E}}_{\mathbf{a},0}$ to be the pullback from equation~\eqref{eq:pullback2} illustrated on the left of diagram~\eqref{eq:pullback4}.
\begin{equation} \label{eq:pullback4}
\begin{tikzpicture}[baseline=(current  bounding  box.center), scale=1.7]
\node (A) at (0,1) {$\Bar{\Bar{E}}_{\mathbf{a},0}$};
\node (B) at (1,1) {${\bar{E}}_{\mathbf{a},0}$};
\node (C) at (0,0) {$\mathbb{C}$};
\node (D) at (1,0) {$\mathbb{C}$};
\node (E) at (2.5,.5) {$\mathbb{C} \times \mathbb{C}$};
\path[->,font=\scriptsize]
(A) edge node[above]{$\lambda$} (B)
(A) edge node[left]{$\Bar{\Bar{W}}_{\mathbf{a}}$} (C)
(B) edge node[right]{$\bar{W}_\mathbf{a}$} (D)
(C) edge node[above]{$z^{a_0}$} (D)
(B) edge node[above]{$\Psi_{\mathbf{D} (\mathbf{a})}$} (E)
(E) edge node[below]{$\pi_1$} (D);
\end{tikzpicture} 
\end{equation}
Let $S \subset \mathbb{C}$ be a small pointed disc of radius less than $|q_{\mathbf{a}}|^{1 / a_0}$ with marked point $\zeta^0  \in \partial S$ equipped with the moving Lagrangian boundary condition of $\rho^* ( L_0 )$ near $\zeta^0$ and moving counter-clockwise by parallel transport. Write $F$ for this boundary condition and for $s \in \partial S$, $L_{0, s} = F|_{s}$. Let $u : \partial S \backslash \zeta^0 \to \mathbb{C}$ be the function $u(s) := \ln (s^{a_0}/q_{\mathbf{a}})$ with the assumption that $ \lim_{s \to \zeta^0} u(s) \in \mathbb{R}_{> 0}$ as $s$ approaches $\zeta^0$ from a clockwise direction. 
By Propositions~\ref{prop:fibprod} and \ref{prop:matchingpaths}, the image of $L_{0, s}$ via $\lambda$ is  $\vc^{u(s)}_{\alpha_{k}}$ which is a matching cycle relative to $\Psi^{u(s)}_{\mathbf{D} (\mathbf{a})}$ over the matching path and vanishing thimble $\vt_{\gamma^{u(s)}_k}$.

From equation~\eqref{eq:refmodspace}, and the fact that only the divisor $C_0$ was added to $\bar{E}_{\mathbf{a},0}$  (as opposed to $\cup_{i = 0}^p C_i$) the moduli space $\mathcal{M}_{a_0}^0$ of sections of $\bar{\bar{W}}_{\mathbf{a}}$ over $S$ equals
\begin{align*} \mathcal{M}^0_{a_0} (\bar{W}_{\mathbf{a}} ) = \bigsqcup_{x \in L_0  \cap L_{a_0}}  \mathcal{M}^{x, 0}_{a_0} (\bar{W}_{\mathbf{a}} ).
\end{align*}
Now, equation~\eqref{eq:divisor} follows from two observations. 

First, we note that for $x_0 := \Xi_{\mathbf{a}, n} (v_0) \in \Hom_{\mathcal{A}_{\mathbf{a}, \nu_{\mathbf{a}}, n}} (L_0, L_{a_0})$ there exists a section $\varphi_0 \in \mathcal{M}^{x_0, 0}_{a_0} (\bar{W}_{\mathbf{a}} ) $. To see this, recall from the proof of the previous claim that $x_0$ corresponds to the intersection $L_0 \cap L_{a_0}$ lying over the endpoint intersection $z^{-1} = \mu_0 \cap \mu_{a_0}$ via $\Psi^{\zeta^0}_{\mathbf{D} (\mathbf{a})}$. In particular, $x_0$ lies in the critical point set $\cp{\Psi_{\mathbf{D} (\mathbf{a})}}$ and, as $s \in \partial S$ moves counter clockwise around $\partial S$, since $L_{0, s}$ is a matching cycle over $\vt_{\gamma^u_k}$ which must contain the critical endpoint over $z^{-1}$, there is a map $\varphi_0|_{\partial S}: \partial S   \to \lambda^* \left( \cp{\Psi_{\mathbf{D} (\mathbf{a})}} \right)$. By equation~\eqref{eq:pullback4}, $W_{\mathbf{a}} \circ \lambda \circ  \varphi_0 |_{\partial S} = \pi_1 \circ \Psi_{\mathbf{D} (\mathbf{a})} \circ \varphi_0 |_{\partial S}$ has winding number $a_0$, and using the explicit form of $\cp{\Psi_{\mathbf{D} (\mathbf{a})}}$ in equation~\eqref{eq:cpPsi}, we see that it intersects $C_0$ with order $a_0$. Thus the closure of $\lambda^* \left(\cp{\Psi_{\mathbf{D} (\mathbf{a})}} \right)$ over $S$ is a smooth complex curve and $\varphi_0|_{\partial S}$ can be completed to a holomorphic section. We observe that, since restricting $W_{\mathbf{a}}$ to $\cp{\Psi_{\mathbf{D} (\mathbf{a})}}$ is a one-dimensional circuit potential, this argument is identical to that given in the proof of Proposition~\ref{prop:thma1d}. 

For the second observation, we check that $\varphi_0 \in \mathcal{M}^0_{a_0} (\bar{W}_{\mathbf{a}} ) $ is unique. Given any $\varphi \in \mathcal{M}^0_{a_0}(\bar{W}_{\mathbf{a}} ) $, consider $g_{\varphi} : S \to \mathbb{C}$ defined to by $g_\varphi = \phi^{-1} \circ \pi_2 \circ \Phi_{\mathbf{D} (\mathbf{a})} \circ \lambda \circ \varphi$. Then for any $\theta \in (0, 2 \pi)$, we have $g_{\varphi} (e^{i \theta} \zeta^0) \in \exp ({\tilde{\tau}_\theta (\tilde{\mu}_0))})$ where $\tilde{\tau}_C$ was defined as the shear map in equation~\eqref{eq:shear}. For $\varphi \in \mathcal{M}^0_{a_0} (\bar{W}_{\mathbf{a}} ) $, the Maslov index of $g_{\varphi}$ must be $2$ which implies it has a first order zero at $0$. In particular, $h : = g_{\varphi_0}^{-1} \circ g_{\varphi}$ is well defined with image in $\mathbb{C}^*$ and vanishing winding number about $0$ implying that one may take a logarithmic branch $\ln (h)$. Extending $\tilde{\tau}_\theta (\tilde{\mu}_0)$ to a line $\ell_\theta$ for every $\theta$, the boundary conditions of $\ln (h)$ are such that $\ln (h(e^{i\theta})) \in \ell_\theta$. This implies that the Maslov index of $\ln (h) < 2$ and thus $h$ is constant at the intersection $\ell_0 \cap \ell_{2\pi} = 0$. Thus $h = 1$ and $g_{\varphi_0} = g_\varphi$. Consequently,  for every $s \in \partial S$,  $\lambda (\varphi (s)) \in \lambda (L_{0, s})$ lies over the same endpoint $z^{-1}$ of the matching path as $\lambda(\varphi_0 (s))$. But as was observed above, there is a unique point in $L_{0, s}$ mapping to such a critical value of $\Psi_{\mathbf{D} (\mathbf{a})}$ implying that $\varphi_0 |_{\partial S} = \varphi |_{\partial S}$. By unique analytic continuation, $\varphi = \varphi_0$ and 
\begin{align} \label{eq:finaleq} \mathcal{M}^0_{a_0} (\bar{W}_{\mathbf{a}} )  = \mathcal{M}_{a_0}^{x_0, 0} (\bar{W}_{\mathbf{a}} )  = \{\varphi_0\} \end{align} 
By the definition in \eqref{eq:defkappa}, we have
$\kappa_{a_0}^0 (x) = \# \mathcal{M}^{x , 0}_{a_0} (\bar{W}_{\mathbf{a}})$ is $1$ for $x = x_0 = \Xi_{\mathbf{a}, n} (v_0)$ and $0$ otherwise. This validates equation~\eqref{eq:divisor} and concludes the proof of Theorem~\ref{thm:Amodel}.

\bibliography{mybib}{}
\bibliographystyle{plain}
\end{document}